\documentclass[12pt,reqno,intlimits]{amsart}

\usepackage{amssymb,amscd}

\newcommand{\nc}{\newcommand}

\nc{\al}{\alpha}
\nc{\bt}{\beta}
\nc{\gm}{\gamma}
\nc{\Gm}{\Gamma}
\nc{\dl}{\delta}
\nc{\Dl}{\Delta}
\nc{\lb}{\lambda}
\nc{\om}{\omega}
\nc{\Om}{\Omega}
\nc{\sg}{\sigma}
\nc{\tht}{\theta}
\nc{\Tht}{\Theta}
\nc{\vf}{\varphi}
\nc{\ve}{\varepsilon}
\nc{\zt}{\zeta}

\nc{\bF}{\mathbb{F}}
\nc{\bR}{\mathbb{R}}
\nc{\bL}{{\bf L}}
\nc{\bP}{{\bf P}}

\nc{\cA}{\mathcal{A}}
\nc{\cB}{\mathcal{B}}
\nc{\cE}{\varepsilon} 
\nc{\cF}{\mathcal{F}}
\nc{\cH}{\mathcal{H}}
\nc{\cJ}{\mathcal{J}}
\nc{\cK}{\mathcal{K}}
\nc{\cM}{\mathcal{M}}
\nc{\cN}{\mathcal{N}}
\nc{\cL}{\mathcal{L}}
\nc{\cX}{\mathcal{X}}
\nc{\cP}{\mathcal{P}}
\nc{\cV}{\mathcal{V}}
\nc{\cv}{v}

\nc{\pa}{\partial}
\nc{\sbs}{\subset}
\nc{\sbseq}{\subseteq}
\nc{\wt}{\widetilde}
\nc{\wh}{\widehat}
\nc{\Ra}{\Rightarrow}
\nc{\Lra}{\Leftrightarrow}
\nc{\LLra}{\Longleftrightarrow}
\nc{\la}{\langle}
\nc{\ra}{\rangle}
\nc{\ua}{\uparrow}
\nc{\str}{\stackrel}
\nc{\ol}{\overline}
\nc{\ul}{\underline}
\nc{\os}{\overset}
\nc{\us}{\underset}
\nc{\fa}{\forall}
\nc{\ds}{\displaystyle}
\nc{\XA}{\{X\to\}}
\nc{\Pas}{P\text{-\it a.s.}}
\nc{\vnth}{\varnothing}

\nc{\loc}{\operatorname{loc}}
\nc{\rank}{\operatorname{rank}}
\nc{\tr}{\operatorname{tr}}
\nc{\diag}{\operatorname{diag}}
\nc{\Ball}{\operatorname{Ball}}
\newcommand{\const}{\operatorname{const}}
\newcommand{\sign}{\operatorname{sign}}

\newtheorem{theorem}{Theorem}[subsection]
\newtheorem{lemma}{Lemma}[subsection]
\newtheorem{corollary}{Corollary}[subsection]
\newtheorem{proposition}{Proposition}[subsection]

\theoremstyle{definition}
\newtheorem{definition}{Definition}[subsection]
\newtheorem{example}{Example}

\theoremstyle{remark}
\newtheorem{remark}{Remark}[subsection]

\numberwithin{equation}{subsection}


\begin{document}

\title[Semimartingale Stochastic Approximation Procedure]{Semimartingale Stochastic Approximation Procedure and Recursive Estimation}

\author[N. Lazrieva, T. Sharia and T. Toronjadze]{N. Lazrieva$^{1),\,2)}$, T. Sharia$^{3)}$ and T. Toronjadze$^{1),\,2)}$}

\maketitle

\begin{center}
$^{1)}$ Georgian--American University, Business School,
3, Alleyway II, Chavchavadze Ave. 17\,a, Tbilisi, Georgia, E-mail: toronj333@yahoo.com \\[2mm]
$^{2)}$ A. Razmadze Mathematical Institute, 1, M. Aleksidze St., Tbilisi, Georgia \\[2mm]
$^{3)}$ Department of Mathematics, Royal Holloway, University of London, Egham, Surrey TW200EX, E-mail: t.sharia@rhul.ac.uk
\end{center}

\begin{abstract}
The semimartingale stochastic approximation procedure, na\-me\-ly, the Robbins--Monro type SDE is introduced which naturally includes both generalized stochastic approximation algorithms with martingale no\-ises and recursive parameter estimation procedures for statistical models associated with semimartingales. General results concerning the asymptotic behaviour of the solution are presented. In particular, the conditions ensuring the convergence, rate of convergence and asymptotic expansion are established. The results concerning the Polyak weighted averaging procedure are also presented.
\end{abstract}

\section*{Contents}

\begin{enumerate}
\item[0.] Introduction
\dotfill{ 2}

\item[1.] Convergence
\dotfill{ 6}

\begin{enumerate}
\item[1.1.] The semimartingales convergence sets
\dotfill{ 6}

\item[1.2.] Main theorem
\dotfill{ 9}

\item[1.3.] Some simple sufficient conditions for (I) and (II)
\dotfill{ 12}

\item[1.4.] Examples
\dotfill{ 13}

\item[1.] Recursive parameter estimation procedures for statistical \\models associated with semimartingale
\dotfill{ 13}

\item[2.] Discrete time
\dotfill{ 20}

\item[3.] RM Algorithm with Deterministic Regression Function
\dotfill{ 23}

\end{enumerate}

\item[2.] Rate of Convergence and Asymptotic Expansion
\dotfill{ 26}

\begin{enumerate}
\item[2.1.] Notation and preliminaries
\dotfill{ 26}

\item[2.2.] Rate of convergence
\dotfill{ 28}

\item[2.3.] Asymptotic expansion
\dotfill{ 41}

\end{enumerate}

\item[3.] The Polyak Weighted Averaging Procedure
\dotfill{ 46}

\begin{enumerate}
\item[3.1.] Preliminaries
\dotfill{ 48}

\item[3.2.] Asymptotic properties of $\ol{z}$. ``Linear'' Case
\dotfill{ 54}

\item[3.3.] Asymptotic properties of $\ol{z}$. General case
\dotfill{ 58}

\end{enumerate}
\item[] References
\dotfill{ 60}

\end{enumerate}

\addtocounter{section}{-1}
\numberwithin{equation}{section}
\section{Introduction}

In 1951 in the famous paper of H. Robbins and S. Monro ``Stochastic approximation method'' \cite{36n} a method was created to address the problem of location of roots of functions, which can only be observed with random errors. In fact, they carried in the classical Newton's method a ``stochastic'' component.

This method is known in probability theory as the Robbins--Monro (RM) stochastic approximation algorithm (procedure).

Since then, a considerable  amount of works has been done to relax assumptions on the regression functions, on the structure of the measurement errors as well (see, e.g., \cite{17n}, \cite{23n}, \cite{26n}, \cite{27n}, \cite{14}, \cite{29n}, \cite{16}, \cite{41n}, \cite{42n}). In particular, in \cite{14} by A. V. Melnikov the generalized stochastic approximation algorithms with deterministic regression functions and martingale noises (do not depending on the phase variable) as the strong solutions of semimartingale SDEs were introduced.

Beginning from the paper \cite{1n} of A. Albert and L. Gardner a link between RM stochastic approximation algorithm and recursive parameter estimation procedures was intensively exploited. Later on recursive parameter estimation procedures for various special models (e.g., i.i.d models, non i.i.d. models in discrete time, diffusion models etc.) have been studied by a number of authors using methods of stochastic approximation (see, e.g., \cite{7n}, \cite{17n}, \cite{23n}, \cite{26n}, \cite{27n}, \cite{proc2}, \cite{20}, \cite{21}). It would be mentioned the fundamental book \cite{32n} by M. B. Nevelson and R.Z. Khas'minski (1972) between them.

In 1987 by N. Lazrieva and T. Toronjadze an heuristic algorithm of a construction of the recursive parameter estimation procedures for statistical models associated with semimartingales (including both discrete and continuous time semimartingale statistical models) was proposed  \cite{18n}. These procedures could not be covered by the generalized stochastic approximation algorithm proposed by Melnikov, while in i.i.d. case the classical RM algorithm contains recursive estimation procedures.

To recover the link between the stochastic approximation and recursive parameter estimation in \cite{proc10}, \cite{20n}, \cite{21n} by Lazrieva, Sharia and Toronjadze the semimartingale stochastic differential equation was introduced, which naturally includes both generalized RM stochastic approximation algorithms with martingale noises and recursive parameter estimation procedures for semimartingale statistical models.

Let on the stochastic basis $(\Om, \cF, F=(\cF_t)_{t\geq 0},P)$ satisfying the usual conditions the following objects  be given:

\begin{enumerate}
\item[a)] the random field $H=\{H_t(u)$, $t\geq 0$, $u\in R^1\}=\{H_t(\om,u)$, $t\geq 0$, $\om\in \Om$, $u\in R^1\}$ such that for each $u\in R^1$ the process $H(u)=(H_t(u))_{t\geq 0}\in \cP$ (i.e. is predictable);

\item[b)] the random field $M=\{M(t,u)$, $t\geq 0$, $u\in R^1\}=\{M(\om,t,u)$, $\om\in \Om$, $t\geq 0$, $u\in R^1\}$ such that for each $u\in R^1$ the process $M(u)=(M(t,u))_{t\geq 0}\in \cM_{\loc}^2(P)$;

\item[c)] the predictable increasing process $K=(K_t)_{t\geq 0}$ (i.e. $K\in \cV^+\cap \cP$).
\end{enumerate}

In the sequel we restrict ourselves to the consideration of the following particular cases:

\begin{enumerate}
\item[$1^\circ.$] $M(u)\equiv m\in \cM_{\loc}^2(P)$;

\item[$2^\circ.$] for each $u\in R^1$ $M(u)=f(u)\cdot m +g(u)\cdot n$, where $m\in \cM_{\loc}^c(P)$, $n\in \cM_{\loc}^{d,2}(P)$, the processes $f(u)=(f(t,u))_{t\geq 0}$ and $g(u)=(g(t,u))_{t\geq 0}$ are predictable, the corresponding stochastic integrals are well-defined and $M(u)\in \cM_{\loc}^2(P)$;

\item[$3^\circ.$] for each $u\in R^1$ $M(u)=\vf(u)\cdot m+W(u)*(\mu-\nu)$, where $m\in \cM_{\loc}^c(P)$, $\mu$ is an integer-valued random measure on $(R\times E, \cB(R_+)\times \cE)$, $\nu$ is its $P$-compensator, $(E,\cE)$ is the Blackwell space, $W(u)=(W(t,x,u)$, $t\geq 0$, $x\in E)\in \cP \otimes \cE$. Here we also mean that all stochastic integrals are well-defined.
\end{enumerate}

Later on by the symbol $\int\limits_0^t M(ds, u_s)$, where $u=(u_t)_{t\geq 0}$ is some predictable process, we denote the following stochastic line integrals:
$$
    \int_0^t f(s,u_s)\, dm_s+\int_0^t g(s, u_s)\, dn_s
        \quad \text{(in case $2^\circ$)}
$$
or
$$
    \int_0^t \vf(s,u_s)\, dm_s+\int_0^t \int_E W(s,x, u_s)(\mu-\nu)(ds,dx)
        \quad \text{(in case $3^\circ$)}
$$
provided the latters are well-defined.

Consider the following semimartingale stochastic differential equation
\begin{equation}\label{0.1}
    z_t=z_0+\int_0^t H_s(z_{s-}) \, dK_s +
        \int_0^t M(ds, z_{s-}), \quad z_0\in \cF_0.
\end{equation}

We call SDE (\ref{0.1}) the Robbins--Monro (PM) type SDE if the drift coefficient $H_t(u)$, $t\geq 0$, $u\in R^1$ satisfies the following conditions: for all $t\in [0,\infty)$ $\Pas$

\smallskip

(A) $\quad \begin{aligned}
            & H_t(0)=0, \\
            & H_t(u)u<0 \;\;\;\text{for all} \;\;\; u\neq 0.
        \end{aligned} $

\smallskip

The question of strong solvability of SDE (\ref{0.1}) is well-investigated (see, e.g., \cite{2}, \cite{3}, \cite{7}).

We assume that there exists an unique strong solution $z=(z_t)_{t\geq 0}$ of equation (\ref{0.1}) on the whole time interval $[0,\infty)$ and such that $\wt M\in \cM_{\loc}^2(P)$, where
$$
    \wt M_t =\int_0^t M(ds, z_{s-}).
$$
Some sufficient conditions for the latter can be found in  \cite{2}, \cite{3}, \cite{7}.

The unique solution $z=(z_t)_{t\geq 0}$ of RM type SDE (\ref{0.1}) can be viewed as a semimartingale stochastic approximation procedure.

In the present work we are concerning with the asymptotic behaviour of the process $(z_t)_{t\geq 0}$ and also of the averized procedure $\ol{z}=\ve^{-1}(z\circ \ve)$ (see Section 3 for the definition of $\ol{z}$) as $t\to \infty$.

The work is organized as follows:

In Section 1 we study the problem of convergence
\begin{equation}\label{0.2}
    z_t\to 0 \;\;\; \text{as} \;\;\; t\to \infty \;\;\Pas
\end{equation}

Our approach to this problem is based, at first, on the description of the non-negative semimartingale convergence sets given in subsection 1.1 \cite{proc10} (see also \cite{proc10} for other references) and, at the second, on two representations ``standard'' and ``nonstandard'' of the predictable process $A=(A_t)_{t\geq 0}$ in the canonical decomposition of the semimartingale $(z_t^2)_{t\geq 0}$, $z_t^2=A_t+{\rm mart}$, in the form of difference of two predictable increasing processes $A^1$ and $A^2$. According to these representations two groups of conditions (I) and (II) ensuring the convergence (\ref{0.2}) are introduced.

in subsection 1.2 the main theorem concerning (\ref{0.2}) is formulated. Also the relationship between groups (I) and (II) are investigated. In subsection 1.3 some simple conditions for (I) and (II) are given.

In subsection 1.4 the series of examples illustrating the efficience of all aspects of our approach are given. In particular, we introduced in Example 1 the recursive parameter estimation procedure for semimartingale statistical models and showed how can it be reduced to the SDE (\ref{0.1}). In Example 2 we show that the recursive parameter estimation procedure for discrete time general statistical models can also be embedded in stochastic approximation procedure given by (\ref{0.1}). This procedure was studied in \cite{20} in a full capacity.

In Example 3 we demonstrate that the generalized  stochastic approximation algorithm proposed in \cite{14} is covered by SDE (\ref{0.1}).

In Section 2 we establish the rate of convergence (see subsection 2.2) and also show that under very mild conditions the process $z=(z_t)_{t\geq 0}$ admits an asymptotic representation where the main term is a normed locally square integrable martingale. In the context of the parametric statistical estimation this implies the local asymptotic linearity of the corresponding recursive estimator. This result enables one to study the asymptotic behaviour of process $z=(z_t)_{t\geq 0}$ using a suitable form of the Central limit theorem for martingales (see Refs. \cite{11n}, \cite{12n}, \cite{14n}, \cite{proc12}, \cite{2-21}).

In subsection 2.1 we introduce some notations and present the normed process $\chi^2 z^2$ in form
\begin{equation}\label{0.3}
    \chi_t^2 z_t^2 =\frac{L_t}{\la L\ra_t^{1/2}} +R_t,
\end{equation}
where $L=(L_t)_{t\geq 0} \in \cM_{\loc}^2(P)$ and $\la L\ra_t$ is the shifted square characteristic of $L$, i.e. $\la L\ra_t:=1+\la L\ra_t^{F,P}$. See also subsection 2.1 for the definition of all objects presented in (\ref{0.3}).

In subsection 2.2 assuming $z_t\to 0$ as $t\to \infty$ $\Pas$, we give various sufficient conditions to ensure the convergence
\begin{equation}\label{0.4}
    \gm_t^\dl z_t^2 \to 0 \;\;\;\text{as} \;\;\; t\to \infty \;\; (\Pas)
\end{equation}
for all $\dl$, $0<\dl<\dl_0$, where $\gm=(\gm_t)_{t\geq 0}$ is a predictable increasing process and $\dl_0$, $0<\dl_0\leq 1$, is some constant. In this subsection we also give series if examples illustrating these results.

In subsection 2.3 assuming that Eq. (\ref{0.4}) holds with the process asymptotically equivalent to $\chi^2$, we study sufficient conditions to ensure the convergence
\begin{equation}\label{0.5}
    R_t\str{P}{\to} 0\;\;\;\text{as} \;\;\; t\to \infty
\end{equation}
which implies the local asymptotic linearity of recursive procedure $z=(z_t)_{t\geq 0}$. As an example illustrating the efficience of introduced conditions we consider RM stochastic approximation procedure with slowly varying gains (see \cite{2-19}).

An important approach to stochastic approximation problems has been proposed by Polyak in 1990 \cite{proc1} and Ruppert in 1988 \cite{proc2}. The main idea of this approach is the use of averaging iterates obtained from primary schemes. Since then the averaging procedures were studied by a number of authors for various schemes of stochastic approximation (\cite{1n}, \cite{proc7}, \cite{proc6}, \cite{proc8}, \cite{5n}, \cite{proc5}, \cite{7n}, \cite{2-19}, \cite{proc3}). The most important results of these studies is that the averaging procedures lead to the asymptotically optimal estimates, and in some cases, they converges to the limit faster than the initial algorithms.

In Section 3 the Polyak weighted averaging procedures of the initial process $z=(z_t)_{t\geq 0}$ are considered. They are defined as
\begin{equation}\label{0.6}
    \ol{z}_t =\cE_t^{-1}(g\circ K) \int_0^t z_s \, d\cE_s (g\circ K),
\end{equation}
where $g=(g_t)_{t\geq 0}$ is a predictable process, $g_t\geq 0$, $\int\limits_0^t g_s dK_s<\infty$, $\int\limits_0^\infty g_t d K_t=\infty$ and $\cE_t(X)$ as usual is the Dolean exponential.

Here the conditions are stated which guarantee the asymptotic normally of process $\ol{z}=(\ol{z}_t)_{t\geq 0}$ in case of continuous process under consideration.

The main result of this section is presented in Theorem 3.3.1, where assuming that Eq. (\ref{0.4}) holds true with some increasing process $\gm=(\gm_t)_{t\geq 0}$ asymptotically equivalent to the process $(\Gm_t^2\la L\ra_t^{-1})_{t\geq 0}$ the conditions are given that ensure the convergence
\begin{equation}\label{0.7}
    \cE_t^{1/2} \ol{z}_t \str{d}{\to} \sqrt{2}\, \xi, \quad \xi\in N(0,1),
\end{equation}
where $\cE_t=1+\int\limits_0^t \Gm_s^2 \la L\ra_s^{-1} \bt_s d K_s$.

As special cases we have obtained the results concerning averaging procedures for standard RM stochastic approximation algorithms and those with slowly varying gains.

All notations and fact concerning the martingale theory used in the presented work can be found in \cite{12n}, \cite{14n}, \cite{proc12}.


\numberwithin{equation}{subsection}
\setcounter{equation}{0}
\section{Convergence}

\subsection{The semimartingales convergence sets}\label{s2}

Let $(\Om,\cF,F=(\cF_t)_{t\geq 0},P)$ be a stochastic basis satisfying the usual conditions, and let $X=(X_t)_{t\geq 0}$ be an $F$-adapted process with trajectories in Skorokhod space $D$ (notation $X=F\cap D$). Let $X_\infty=\lim\limits_{t\to \infty} X_t$ and let $\XA$ denote the set, where $X_\infty$ exists and is a finite random variable (r.v.).

In this section we study the structure of the set $\XA$ for nonnegative special semimartingale $X$. Our approach is based on the multiplicative decomposition of the positive semimartingales.

Denote $\cV^+$ $(\cV)$ the set of processes $A=(A_t)_{t\geq 0}$, $A_0=0$, $A\in F\cap D$ with nondecreasing (bounded variation on each interval $[0,t[$) trajectories. We write $X\in \cP$ if $X$ is a predictable process. Denote $S_P$ the class of special semimartingales, i.e. $X\in S_p$ if $X\in F\cap D$ and
$$
    X=X_0+A+M,
$$
where $A\in \cV\cap \cP$, $M\in \cM_{\loc}$.

Let $X\in S_P$. Denote $\cE(X)$ the solution of the Dolean equation
$$
    Y=1+Y_-\cdot X,
$$
where $Y_-\cdot X_t:=\int\limits_0^t Y_{s-} dX_s$.

If $\Gm_1,\Gm_2\in \cF$, then $\Gm_1=\Gm_2$ ($P$-{\it a.s.}) or $\Gm_1\sbseq \Gm_2$ ($P$-{\it a.s.}) means $P(\Gm_1\Dl \Gm_2)=0$ or $P(\Gm_1\cap(\Om\setminus \Gm_2))=0$ respectively, where $\Dl$ is the sign of the symmetric difference of sets.

Let $X\in S_P$. Put $A=A^1-A^2$, where $A^1,A^2\in \cV^+\cap \cP$. Denote
$$
    \wh A=(1+X_-+A_-^2)^{-1} \circ A^2 \quad
        \Bigg( :=\int_0^{\cdot} (1+X_{s-}+A_{s-}^2)^{-1} dA_s^1 \Bigg).
$$

\begin{theorem}\label{t2.1}
Let $X\in S_P$, $X\geq 0$. Then
$$
    \{\wh A_\infty <\infty\} \sbseq \XA \cap \{A_\infty^2<\infty\}
        \quad (\Pas).
$$
\end{theorem}

\begin{proof}
Consider the process $Y=1+X+A^2$. Then
$$
    Y=Y_0+A^1+M, \quad Y_0=1+X_0,
$$
$Y\geq 1$, $Y_-^{-1}\Dl A^1\!\geq \!0$. Thus the processes $\wh A=Y_-^{-1} \circ A^1$ and $\wh M=(Y_-+\Dl A^1)^{-1}\cdot M$ are well-defined and besides $\wh A\in \cV^+\cap \cP$, $\wh M\in\cM_{\loc}$. Then, using Theorem~1, \S 5, Ch. 2 from \cite{proc12} we get the following multiplicative decomposition
$$
    Y=Y_0\cE(\wh A) \cE(\wh M),
$$
where $\cE(\wh A)\in \cV^+\cap \cP$, $\cE(\wh M)\in \cM_{\loc}$.

Note that $\Dl \wh M>-1$. Indeed, $\Dl\wh M=(Y_-+\Dl A^1)^{-1} \Dl M$. But $\Dl M=\Dl Y-\Dl A^1=Y-(Y_-+\Dl A^1)>-(Y_-+\Dl A^1)$. Therefore $\cE(\wh M)>0$ and $\{\cE(\wh M)\to\}=\Om$ $(\Pas)$. On the other hand (see, e.g., \cite{16}, Lemma 2.5)
$$
    \cE_t(\wh A)\ua \infty \Longleftrightarrow \wh A_t \ua \infty \quad \text{as}
        \quad t\to \infty.
$$
Hence
$$
    \{\wh A_\infty <\infty\} \sbseq \{Y\to\} =\XA \cap
        \{A_\infty^2<\infty\},
$$
since $A^2<Y$ and $A^2\in \cV^+$.

Theorem is proved.
\end{proof}

\begin{corollary}\label{c2.1}
$$
    \{A_\infty^1<\infty\}=\{(1+X_-)^{-1} \circ A_\infty^1<\infty\} =
        \{\wh A_\infty<\infty\} \;\; (\Pas).
$$
\end{corollary}

\begin{proof}
It is evident that
\begin{align*}
    \{A_\infty^1<\infty\} & \sbseq \{(1+X_-)^{-1} \circ A_\infty^1<\infty\}
        \sbseq \{\wh A_\infty<\infty\} \\
    & \sbseq \XA\cap \{A_\infty^2<\infty\}\;\;(\Pas).
\end{align*}
It remains to note that
\begin{align*}
    \{A_\infty^1 & <\infty\}\cap \XA \cap \{A_\infty^2 <\infty\} \\
    & =\{\wh A_\infty <\infty\} \cap \XA \cap \{A_\infty^2 <\infty\}\;\;(\Pas).
\end{align*}
Corollary is proved.
\end{proof}

\begin{corollary}\label{c2.2}
$$
    \{\wh A_\infty<\infty\}\cap \{\cE_\infty(\wh M)>0\}=
        \XA\cap \{A_\infty^2<\infty\} \cap \{\cE_\infty(\wh M)>0\}\;\; (\Pas),
$$
as it easily follows from the proof of Theorem $\ref{t2.1}$.
\end{corollary}

\begin{remark}\label{r2.1}
The relation
$$
    \{A_\infty^1<\infty\}\sbseq \XA \cap \{A_\infty^2<\infty\}\;\;(\Pas)
$$
has been proved in \cite{proc12}, Ch. 2, \S 6, Th. 7. Under the following additional assumptions:

1. $EX_0<\infty$;

2. one of the following conditions $(\al)$ or $(\bt)$ are satisfied:
\begin{enumerate}
\item[$(\al)$] there exists $\ve>0$ such that $A_{t+\ve}^1\in \cF_t$ for all $t>0$,
\item[$(\bt)$] for any predictable Markov moment $\sg$
$$
    E\Dl A_\sg^1 I_{\{\sg<\infty\}}<\infty.
$$
\end{enumerate}
\end{remark}


Let $A,B\in F\cap D$. We write $A\prec B$ if $B-A\in \cV^+$.

\begin{corollary}\label{c2.3}
Let $X\in S_P$, $X\geq 0$, $A\leq A^1-A^2$ and $A\prec A^1$, where $A^1,A^2\in \cV^+\cap \cP$. Then
$$
    \{A_\infty^1<\infty\}=\{(1+X_-)^{-1} \circ A_\infty^1<\infty\}
        \sbseq \XA \cap \{A_\infty^2<\infty\} \;\;(\Pas).
$$
\end{corollary}

\begin{proof}
Rewrite $X$ in the form
$$
    X=X_0+A^1-\wt A{\,}^2+M,
$$
where $\wt A{\,}^2=A^1-A\in \cV^1\cap \cP$. Then the desirable follows from Theorem \ref{t2.1}, Corollary \ref{c2.1} and trivial inclusion $\{\wt A_\infty^2<\infty\} \sbseq \{A_\infty^2<\infty\}$.

The corollary is proved.
\end{proof}

\begin{corollary}\label{c2.4}
Let $X\in S_P$, $X\geq 0$ and
$$
    X=X_0+X_-\circ B+A+M
$$
with $B\in \cV^+\cap \cP$, $A\in \cV\cap \cP$ and $M\in \cM_{\loc}$.

Suppose that for $A^1,A^2\in \cV^+\cap \cP$
$$
    A\leq A^1-A^2 \quad \text{and} \quad A\prec A^1.
$$
Then
$$
    \{A_\infty^1<\infty\} \cap \{B_\infty<\infty\} \sbseq
        \XA \cap \{A_\infty^2<\infty\} \;\;(\Pas).
$$
\end{corollary}

The proof is quite similar to the prof of Corollary \ref{c2.3} if we consider the process $X\cE^{-1}(B)$.

\begin{remark}\label{r2.3}
Consider the discrete time case.

Let $\cF_0,\cF_1,\dots$ be a non-decreasing sequence of $\sg$-algebras and $X_n$, $\bt_n$, $\xi_n$, $\zt_n\in \cF_n$, $n\geq 0$, are nonnegative r.v. and
$$
    X_n=X_0+\sum_{i=0}^n X_{i-1} \bt_{i-1}+A_n+M_n
$$
(we mean that $X_{-1}=X_0$, $\cF_{-1}=\cF_0$ and $\bt_{-1}=\xi_{-1}=\zt_{-1}=0$), where $A_n\in \cF_{n-1}$ with $A_0=0$ and $M\in \cM_{\loc}$. Note that $X_n$ can always be represented in this form taking $A_n=\sum\limits_{i=0}^n (E(X_i| \cF_{i-1})-X_{i-1}) -\sum\limits_{i=0}^n X_{i-1}\bt_{i-1}$.

Denote
$$
    A_n^1=\sum_{i=0}^n \xi_{i-1} \quad \text{and} \quad
    A_n^2=\sum_{i=0}^n \zt_{i-1}.
$$

It is clear that in this case
$$
    A\prec A^1 \Longleftrightarrow \Dl A_n\leq \xi_{n-1}
$$
$(\Dl A_n:=A_n-A_{n-1}$, $n\geq 1)$.

So, in this case Corollary \ref{c2.4} can be formulated in the following way:

{\it Let for each $n$
$$
    A_n \leq \sum_{i=0}^n (\xi_{i-1}-\zt_{i-1})
$$
and
$$
    \Dl A_n \leq \xi_{n-1}.
$$

Then }
$$
    \bigg\{ \sum_{i=0}^\infty \xi_{i-1}<\infty \bigg\} \cap
    \bigg\{ \sum_{i=0}^\infty \bt_{i-1}<\infty \bigg\} \sbseq \XA \cap
    \bigg\{ \sum_{i=0}^\infty \zt_{i-1}<\infty \bigg\}\;(\Pas).
$$
\end{remark}

From this corollary follows the result by Robbins and Siegmund (see Robbins, Siegmund \cite{37n}). Really, the above inclusion holds if in particular $\Dl A_n \leq \xi_{n-1} -\zt_{n-1}$, $n\geq 1$, i.e. when
$$
    E(X_n \mid \cF_{n-1}) \leq X_{n-1} (1+\bt_{n-1}) +\xi_{n-1}-
        \zt_{n-1}, \quad n\geq 0.
$$
In our terms the previous inequality means $A\prec A^1-A^2$.

\subsection{Main theorem}\label{s3}

Consider the stochastic equation (RM procedure)
\begin{equation}\label{3.1}
    z_t=z_0+\int_0^t H_s(z_{s-})\,dK_s +\int_0^t M(ds, z_{s-}),
        \quad t\geq 0, \quad z_0\in \cF_0,
\end{equation}
or in the differential form
$$
    dz_t=H_t(z_{t-}) dK_t+M(dt,z_{t-}), \quad z_0\in \cF_0.
$$
Assume that there exists an unique strong solution $z=(z_t)_{t\geq 0}$ of (\ref{3.1}) on the whole time interval $[0,\infty)$, $\wt M\in \cM_{\loc}^2$, where
$$
    \wt M_t:=\int_0^t M(ds,z_{s-}).
$$

We study the problem of $P$-$a.s.$ convergence $z_t\to 0$, as $t\to \infty$.

For this purpose apply Theorem \ref{t2.1} to the semimartingale $X_t=z_t^2$, $t\geq 0$. Using the Ito formula we get for the process $(z_t^2)_{t\geq 0}$
\begin{equation}\label{3.2}
    dz_t^2=dA_t+dN_t,
\end{equation}
where
\begin{align*}
    & dA_t =V_t^-(z_{t-}) dK_t+V_t^+(z_{t-}) dK_t^d+d\la\wt M\ra_t, \\
    & dN_t=2z_{t-} d\wt M_t+H_t(z_{t-}) \Dl K_t \,d\wt M_t^d+
        d([\wt M]_t-\la \wt M\ra_t),
\end{align*}
with
\begin{align*}
    & V_t^-(u):= 2H_t(u)u, \\
    & V_t^+(u):= H_t^2(u)\Dl K_t.
\end{align*}
Note that $A=(A_t)_{t\geq 0} \in \cV\cap \cP$, $N\in \cM_{\loc}$.

Represent the process $A$ in the form
\begin{equation}\label{3.3}
    A_t=A_t^1-A_t^2
\end{equation}
with \\
$(1) \quad \begin{cases}
                \quad dA_t^1=V_t^+(z_{t-}) dK_t^d+d\la \wt M\ra_t, \\
                \;-dA_t^2=V_t^-(z_{t-}) dK_t,
            \end{cases} $ \\
or \\
$(2) \quad \begin{cases}
                \quad dA_t^1=[V_t^-(z_{t-})I_{\{\Dl K_t\neq 0\}}+
                    V_t^+(z_{t-})]^+ dK_t^d+d\la \wt M\ra_t, \\
                \;-dA_t^2=\{V_t^-(z_{t-})I_{\{\Dl K_t= 0\}}-
                    [V_t^-(z_{t-}) I_{\{\Dl K_t\neq 0\}}
                +V_t^+(z_{t-})]^-\} dK_t,
            \end{cases} $ \\
where $[a]^+=\max(0,a)$, $[a]^-=-\min(0,a)$.

As it follows from condition (A) $\al_t(z_{t-})\leq 0$ for all $t\geq 0$ and so, the representation (\ref{3.3})(1) directly corresponds to the usual (in stochastic approximation procedures) standard form of process $A$ (in (\ref{3.2}) $A=A^1-A^2$ with $A^1,A^2$ from (\ref{3.3})(1)). Therefore we call representation (\ref{3.3})(1) ``standard'', while the representation (\ref{3.3})(2) is called ``nonstandard''.

Introduce the following group of conditions: For all $u\in R^1$ and $t\in [0,\infty)$ \\[2mm]
(A)
For all $t\in [0,\infty)$ $\Pas$
\begin{enumerate}
\item[] $H_t(0)=0$,
\item[] $H_t(0)u<0$ for all $u\neq 0$;
\end{enumerate}

\smallskip

\noindent (B)
\begin{enumerate}
\item[(i)] $\la M(u)\ra \ll K$,
\smallskip

\item[(ii)] $h_t(u)\leq B_t(1+u^2)$, $B_t\geq 0$, $B=(B_t)_{t\geq 0} \in \cP$, $B\circ K_\infty<\infty$, \\[1mm] where $h_t(u)=\frac{d\la M(u)\ra_t}{dK_t}$;
\end{enumerate}
\smallskip

\noindent (I)
\begin{enumerate}
\item[(i)] $(i_1) \;\; I_{\{\Dl K_t\neq 0\}} |H_t(u)| \leq C_t(1+|u|)$, $C_t\!\geq \!0$, $C=(C_t)_{t\geq 0}\!\in \!\cP$,
 $C\circ K_t\!<\!\infty$,  \\[1mm]
$(i_2) \;\; C^2\Dl K \circ K_\infty^d<\infty$,
\smallskip

\item[(ii)] for each $\ve>0$
$$
    \inf_{\ve\leq |u|\leq 1/\ve}|V^-(u)|\circ K_\infty=\infty;
$$
\end{enumerate}
\smallskip

\noindent (II)
\begin{enumerate}
\item[(i)] $[V_t^-(u) I_{\{\Dl K_t\neq 0\}}+V_t^+(u)]^+\leq D_t(1+u^2)$, $D_t\geq 0$, \\
$D=(D_t)_{t\geq 0}\in \cP$, $D\circ K_\infty^d<\infty,$
\smallskip

\item[(ii)] for each $\ve>0$
$$
    \inf_{\ve\leq |u|\leq 1/\ve} \{|V^-(u)| I_{\{\Dl K_t= 0\}}+
        [V^-(u) I_{\{\Dl K_t\neq 0\}}+V^+(u)]^-\} \circ K_\infty=\infty.
$$
\end{enumerate}

\begin{remark}\label{r3.1}
When $M(u)\equiv m\in \cM_{\loc}^2$, we do not require the condition $\la m\ra\ll K$ and replace the condition (B) by \\[2mm]
(B$'$) \hskip+2cm $\la m\ra_\infty <\infty.$
\end{remark}
\smallskip

\begin{remark}\label{r3.2}
Everywhere we assume that all conditions are satisfied $P$-$a.s.$
\end{remark}

\begin{remark}\label{r3.3}
It is evident that (I)~(ii)$\Longrightarrow C\circ K_\infty=\infty$.
\end{remark}

\begin{theorem}\label{t3.1}
Let conditions {\rm (A), (B), (I)} or {\rm (A), (B), (II)} be satisfied. Then
$$
    z_t\to 0 \;\; \Pas \quad \text{as} \;\;\; t\to \infty.
$$
\end{theorem}

\begin{proof}
Assume, for example, that the conditions (A), (B) and (I) are satisfied. Then by virtue of Corollary \ref{c2.1} and (\ref{3.2}) with standard representation (\ref{3.3})(1) of process $A$ we get
\begin{equation}\label{3.4}
    \{(1+z_-^2)^{-1} \circ A_\infty^1<\infty\} \sbseq \{z^2\to\} \cap
        \{A_\infty^2<\infty\}.
\end{equation}
But from conditions (B) and (I)~(i) we have
$$
    \{(1+z_-^2)^{-1} \circ A_\infty^1<\infty\}=\Om \;\; (\Pas)
$$
and so
\begin{equation}\label{3.5}
    \{z^2\to \}\cap \{ A_\infty^2<\infty\}=\Om \;\; (\Pas).
\end{equation}
Denote $z_\infty^2=\lim\limits_{t\to \infty} z_t^2$, $N=\{z_\infty^2>0\}$ and assume that $P(N)>0$. In this case from (I)~(ii) by simple arguments we get
$$
    P(|V^-(z_-)|\circ K_\infty=\infty)>0,
$$
which contradicts with (\ref{3.4}). Hence $P(N)=0$.

The proof of the second case is quite similar.

The theorem is proved.
\end{proof}

In the following propositions the relationship between conditions (I) and (II) are given.

\begin{proposition}\label{p3.1}
{\rm (I)}$\Ra${\rm (II)}.
\end{proposition}

\begin{proof}
From (I) ($i_1$) we have
$$
    [V_t^-(u) I_{\{\Dl K_t\neq 0\}} +V_t^+(u)]^+ \leq
        V_t^+(u) \leq C_t^2\Dl K_t(1+u^2)
$$
and if take $D_t=C_t^2 \Dl K_t$, then (II) (i) follows from (I)~$(i_2)$.

Further, from (I)~$(i_1)$ we have for each $\ve>0$ and $u$ with  $\ve\leq |u|\leq 1/\ve$
\begin{multline*}
    \quad |V_t^-(u)| I_{\{\Dl K_t=0\}} +[V_t^-(u)+V_t^+(u)]^-I_{\{\Dl K_t\neq 0\}}\\
    \geq |V_t^-(u)|-V_t^+(u) \geq |V_t^-(u)|-C_t^2 \Dl K_t
        \Big(1+\frac{1}{\ve^2}\Big). \quad
\end{multline*}
Now (II)~(ii) follows from (I)~$(i_2)$ and (I)~(ii).

The proposition is proved.
\end{proof}

\begin{proposition}\label{p2.2}
    Under {\rm (I) (i)} we have {\rm (I) (ii)} $\Lra$ {\rm (II) (ii)}.
\end{proposition}

Proof immediately follows from previous proposition and trivial implication (II)~(ii)$\Ra$(I)~(ii).

\subsection{Some simple sufficient conditions for (I) and (II)}\label{s4}

Introduce the following group of conditions: for each $u\in R^1$ and $t\in [0,\infty)$ \\[2mm]
(S.1)
\begin{align}
&{\rm (i)}  &\hskip-1cm (i_1) \;\;& G_t|u|\leq |H_t(u)| \leq \wt G_t|u|, \;\; G_t\geq 0, \;\;
    G=(G_t)_{t\geq 0}, \notag \\
     &&& \wt G=(\wt G_t)_{t\geq 0} \in \cP, \;\; \wt G\circ K_t<\infty, \notag \\
&&\hskip-1cm (i_2) \;\;& {\wt G}{\,}^2 \Dl K\circ K_\infty^d<\infty;  \notag \\
&{\rm (ii)} && G\circ K_\infty=\infty; \label{4.1}
\end{align}

\noindent (S.2)
\begin{align}
{\rm (i)} \qquad & \wt G[-2+\wt G \Dl K]^+ \circ K_\infty^d<\infty; \quad \label{4.2} \\
{\rm (ii)} \qquad & G\{2I_{\{\Dl K=0\}} +[-2+\wt G\Dl K]^- I_{\{\Dl K\neq 0\}} \circ
    K_\infty=\infty. \quad \label{4.3}
\end{align}

\begin{proposition}\label{p4.1}
\begin{gather*}
    {\rm (S.1)} \Ra {\rm (I)}, \\
    {\rm (S.1)}(i_1),\; {\rm (S.2)} \Ra {\rm (II)}.
\end{gather*}
\end{proposition}

\begin{proof}
The first implication is evident. For the second, note that
\begin{align}
    V_t^-(u) I_{\{\Dl K_t\neq 0\}}+V_t^+(u) & =
        -2|H_t(u)|\,|u| I_{\{\Dl K_t\neq 0\}} +H_t^2(u) \Dl K_t \notag \\
    & \leq |H_t(u)|\,|u|\, [-2I_{\{\Dl K_t\neq 0\}} +\wt G_t \Dl K_t]. \label{4.4}
\end{align}
So
\begin{align*}
    [V_t^-(u) I_{\{\Dl K_t\neq 0\}}+V_t^+(u)]^+ & \leq
        |H_t(u)|\,|u|\, [-2I_{\{\Dl K_t\neq 0\}} +\wt G_t \Dl K_t]^+  \\
    & \leq \wt G_t [-2I_{\{\Dl K_t\neq 0\}} +\wt G_t \Dl K_t]^+|u^2|
\end{align*}
and (II)~(i) follows from (\ref{4.2}) if we take
$$
    D_t=\wt G_t[-2+\wt G_t \Dl K_t]^+ I_{\{\Dl K_t\neq 0\}}.
$$

Further, from (\ref{4.4}) we have
\begin{align*}
    |V_t^-(u)| I_{\{\Dl K_t=0\}} & +[V_t^-(u) I_{\{\Dl K_t\neq 0\}} +V_t^+(u)]^- \\
    & \geq u^2 G_t\{2I_{\{\Dl K_t=0\}} +[-2I_{\{\Dl K_t\neq 0\}} +\wt G_t\Dl K_t]^-\}
\end{align*}
and (II) (ii) follows from (\ref{3.3}).

Proposition is proved.
\end{proof}

\begin{remark}\label{r4.1}\

    a) (S.1)$\Ra$(S.2),

    b) under (S.1) (i) we have (S.1) (ii)$\Lra$(S.2) (ii),

    c) (S.2) (ii)$\Ra$(S.1) (ii).
\end{remark}

Summarizing the above we come to the following conclusions: a) if the condition (S.1) (ii) is not satisfied, then (S.2) (ii) is not satisfied also; b)~if (S.1) $(i_1)$ and (S.1) (ii) are satisfied, but (S.1) $(i_2)$ is violated, then nevertheless the conditions (S.2) (i) and (S.2) (ii) can be satisfied.

In this case the nonstandard representations (\ref{3.3})(2) is useful.

\begin{remark}\label{r4.2}
Denote
$$
    \wt G_t \Dl K_t =2+\dl_t, \quad \dl_t\geq -2 \;\;\;\text{for all} \;\;\;
        t\in [0,\infty).
$$
It is obvious that if $\dl_t\leq 0$ for all $t\in [0,\infty)$, then $[-2+\wt G_t \Dl K_t]^+=0$. So (S.2) (i) is trivially satisfied and (S.2) (ii) takes the form
\begin{equation}\label{4.5}
    G\{ 2I_{\{\Dl K=0\}} +|\dl| I_{\{\Dl K\neq 0\}} \circ K_\infty=\infty.
\end{equation}

Note that if $G\cdot \min(2,|\dl|)\circ K_\infty=\infty$, then (\ref{4.5}) holds, and the simplest sufficient condition (\ref{4.5}) is: for all $t\geq 0$
$$
    G\circ K_\infty=\infty, \quad |\dl_t|\geq const>0.
$$
\end{remark}

\begin{remark}\label{r4.3}
Let the conditions (A), (B) and (I) be satisfied. Since we apply Theorem \ref{t2.1} and its Corollaries on the semimartingales convergence sets given in subsection \ref{s2}, we get rid of many of ``usual'' restrictions: ``moment'' restrictions, boundedness of regression function, etc.
\end{remark}


\subsection{Examples}\label{s5}

\subsubsection{Recursive parameter estimation procedures for statistical models associated with semimartingale}
\subsubsection*{$1.$ Basic model and regularity}

Our object of consideration is a parametric filtered statistical model
$$
    \cE=(\Om,\cF,\bF=(\cF_t)_{t\geq 0},\{P_\tht;\tht\in R\})
$$
associated with one-dimensional $\bF$-adapted RCLL process $X=(X_t)_{t\geq 0}$ in the following way: for each $\tht\in R^1$ $P_\tht$ is an unique measure on $(\Om,\cF)$ such that under this measure $X$ is a semimartingale with predictable characteristics $(B(\tht),C(\tht),\nu_\tht)$ (w.r.t. standard truncation function $h(x)=xI_{\{|x|\leq 1\}})$. Assume for simplicity that all $P_\tht$ coincide on $\cF_0$.

Suppose that for each pair $(\tht,\tht')$ $P_\tht \os{\loc}{\sim} P_{\tht'}$. Fix $\tht=0$ and denote $P=P_0$, $B=B(0)$, $C=C(0)$, $\nu=\nu_0$.

Let $\rho(\tht)=(\rho_t(\tht))_{t\geq 0}$ be a local density process (likelihood ratio process)
$$
    \rho_t(\tht) =\frac{dP_{\tht,t}}{dP_t}\,,
$$
where for each $\tht$ $P_{\tht,t}:= P_\tht |\cF_t$, $P_t:=P|\cF_t$ are restrictions of measures $P_\tht$ and $P$ on $\cF_t$, respectively.

As it is well-known (see, e.g., \cite{14n}, Ch. III, \S 3d, Th. 3.24) for each $\tht$ there exists a $\wt{\cP}$-measurable positive function
$$
    Y(\tht)=\{Y(\om,t,x;\tht), \quad (\om,t,x)\in \Om\times R_+\times R\},
$$
and a predicable process $\bt(\tht)=(\bt_t(\tht))_{t\geq 0}$ with
$$
    |h(Y(\tht)-1)|*\nu\in \cA_{\loc}^+(P), \quad
        \bt^2(\tht)\circ C\in \cA_{\loc}^+(P),
$$
and such that
\begin{equation}\label{d5.1}
\begin{aligned}
    (1) \quad & B(\tht)=B+\bt(\tht)\circ C+h(Y(\tht)-1)*\nu, \\
    (2) \quad & C(\tht)=C, \qquad (3) \quad \nu_\tht=Y(\tht)\cdot\nu.
\end{aligned}
\end{equation}

In addition the function $Y(\tht)$ can be chosen in such a way that
$$
    a_t:=\nu(\{t\},R)=1 \LLra a_t(\tht): =\nu_\tht(\{t\},R)=
        \int Y(t,x;\tht) \nu(\{t\})dx =1.
$$

We assume that the model is regular in the Jacod sense (see \cite{15n}, \S 3, Df. 3.12) at each point $\tht$, that is the process $(\rho_{\tht'}/\rho_\tht)^{1/2}$ is locally differentiable w.r.t $\tht'$ at $\tht$ with the derivative process
$$
    L(\tht) =(L_t(\tht))_{t\geq 0}\in \cM_{\loc}^2(P_\tht).
$$
In this case the Fisher information process is defined as
\begin{equation}\label{d5.2}
    \wh I_t(\tht)=\la L(\tht),L(\tht)\ra_t.
\end{equation}

In \cite{15n} (see \S 2-c, Th. 2.28) was proved that the regularity of the model at point $\tht$ is equivalent to the differentiability of characteristics $\bt(\tht)$, $Y(\tht)$, $a(\tht)$ in the following sense: there exist a predictable process $\dot{\bt}(\tht)$ and $\wt\cP$-measurable function $W(\tht)$ with
$$
    \dot{\bt}^2(\tht)\circ C_t<\infty, \quad W^2(\tht)*\nu_{\tht,t}<\infty
        \;\;\;\text{for all} \;\;\; t\in R_+
$$
and such that for all $t\in R_+$ we have as $\tht'\to \tht$
\begin{equation}\label{d5.3}
\begin{aligned}
    (1) \;\; & (\bt(\tht')-\bt(\tht)-\dot{\bt}(\tht)(\tht'-\tht))^2 \circ C_t/
        (\tht'-\tht)^2 \str{P_\tht}{\to} 0, \\
    (2) \;\; & \Bigg( \left(\frac{Y(\tht')}{Y(\tht)}\right)^{1/2} -1 -
        \frac{1}{2}\,W(\tht)(\tht'-\tht)\Bigg)^2 *\nu_{\tht,t} \Big/ (\tht'-\tht)^2
        \str{P_\tht}{\to} 0, \\
    (3) \;\; & \sum_{\substack{s\leq t \\ a_s(\tht)<1}}
        \bigg[ (1-a_s(\tht'))^{1/2} -(1-a_s(\tht))^{1/2}  \\
    & \quad +\frac{1}{2}\,\frac{\wh W_s^\tht(\tht)}{(1-a_s(\tht))^{1/2}}\,
        (\tht'-\tht)\bigg]^2 \Big/ (\tht'-\tht)^2 \str{P_\tht}{\to} 0,
\end{aligned}
\end{equation}
where
$$
    \wh W_t^\tht(\tht) =\int W(t,x;\tht) \nu_\tht(\{t\},dx).
$$

In this case $a_s(\tht)=1 \Ra \wh W_s^\tht(\tht)=0$ and the process $L(\tht)$ can be written as
\begin{equation}\label{d5.4}
    L(\tht)=\dot{\bt}(\tht) \cdot (X^c-\bt(\tht)\circ C)+\bigg(\wh W^\tht(\tht)+
        \frac{\wh W^\tht(\tht)}{1-a(\tht)} \bigg)*(\mu-\nu_\tht),
\end{equation}
and
\begin{equation}\label{d5.5}
    \wh I(\tht)=\dot{\bt}{}^2(\tht)\circ C+(\wh W^\tht(\tht))^2*\nu_\tht+
        \sum_{s\leq \cdot} \frac{(\wh W_s^\tht(\tht))^2}{1-a_s(\tht)}.
\end{equation}

Denote
$$
    \Phi(\tht)=W(\tht)+\frac{\wh W^\tht(\tht)}{1-a(\tht)}\,.
$$

One can consider the another alternative definition of the regularity of the model (see, e.g., \cite{2-21}) based on the following representation of the process $\rho(\tht)$:
$$
    \rho(\tht)=\cE(M(\tht)),
$$
where
\begin{equation}\label{d5.6}
    M(\tht)=\bt(\tht)\cdot X^c+\bigg(Y(\tht)-1+\frac{\wh Y(\tht)-a}{1-a} \,
        I_{\{0<a<1\}}\bigg)*(\mu-\nu)\in \cM_{\loc}(P).
\end{equation}
Here $X^c$ is a continuous martingale part of $X$ under measure $P$ (see, e.g., \cite{16n}, \cite{14}).

We say that the model is regular if for almost all $(\om,t,x)$ the functions $\bt:\tht\to \bt_t(\om;\tht)$ and $Y:\tht\to Y(\om,t,x;\tht)$ are differentiable (notation $\dot{\bt}(\tht):=\frac{\pa}{\pa\tht} \bt(\tht)$, $\dot{Y}(\tht):=\frac{\pa}{\pa\tht} Y(\tht)$) and differentiability under integral sign is possible. Then
$$
    \frac{\pa}{\pa\tht} \,\ln \rho(\tht)=L(\dot{M}(\tht),M(\tht)) :=
        \wt L(\tht)\in \cM_{\loc}(P_\tht),
$$
where $L(m,M)$ is the Girsanov transformation defined as follows: if $m,M\in \cM_{\loc}(P)$ and $Q\ll P$ with $\frac{dQ}{dP}=\cE(M)$, then
$$
    L(m,M):=m-(1+\Dl M)^{-1}\circ [m,M]\in \cM_{\loc}(Q).
$$

It is not hard to verify that
\begin{equation}\label{d5.7}
    \wt L(\tht)=\dot{\bt}(\tht)\cdot (X^c-\bt(\tht)\circ C) +\wt \Phi(\tht) *(\mu-\nu_\tht),
\end{equation}
where
$$
    \wt \Phi(\tht)=\frac{\dot{Y}(\tht)}{Y(\tht)}+\frac{\dot{a}(\tht)}{1-a(\tht)}
$$
with $I_{\{a(\tht)=1\}} \dot{a}(\tht)=0$.

If we assume that for each $\tht\in R^1$ $\wt L(\tht)\in \cM_{\loc}^2(P_\tht)$, then the Fisher information process is
$$
    \wh I_t(\tht)=\la \wt L(\tht),\wt L(\tht)\ra_t.
$$

It should be noticed that from the regularity of the model in the Jacod sense it follows that $L(\tht)\in \cM_{\loc}^2(P_\tht)$, while under the latter regularity conditions $\wt L(\tht)\in \cM_{\loc}^2(P_\tht)$ is an assumption, in general.

In the sequel we assume that the model is regular in both above given senses. Then
$$
    W(\tht)=\frac{\dot{Y}(\tht)}{Y(\tht)}\,, \quad
        \wh W^\tht(\tht)=\dot{a}(\tht), \quad
        L(\tht)=\wt L(\tht).
$$

\subsubsection*{{\rm 2.} Recursive estimation procedure for MLE}

In \cite{18n} an heuristic algorithm was proposed for the construction of recursive estimators of unknown parameter $\tht$ asymptotically equivalent to the maximum likelihood estimator (MLE).

This algorithm was derived using the following reasons:

Consider the MLE $\;\wh \tht=(\wh\tht_t)_{t\geq 0}$, where $\wh\tht_t$ is a solution of estimational equation
$$
    L_t(\tht)=0.
$$

Assume that
\begin{enumerate}
\item[1)] for each $\tht\!\in\! R^1$ the process $(\wh I_t(\tht))^{1/2} (\wh \tht_t-\tht)$ is $P_\tht$-stochastically bounded and, in addition, the process $(\wh\tht_t)_{t\geq 0}$ is a $P_\tht$-semimartingale;

\item[2)] for each pair $(\tht',\tht)$ the process $L(\tht')\in \cM_{\loc}^2(P_{\tht'})$ and is a $P_\tht$-special semimartingale;

\item[3)] the family $(L(\tht),\tht\in R^1)$ is such that the Ito--Ventzel formula is applicable to the process $(L(t,\wh \tht_t))_{t\geq 0}$ w.r.t. $P_\tht$ for each $\tht\in R^1$;

\item[4)] for each $\tht\in R^1$ there exists a positive increasing predictable process $(\gm_t(\tht))_{t\geq 0}$ asymptotically equivalent to $\wh I_t^{-1}(\tht)$, i.e.
$$
    \gm_t(\tht) \wh I_t(\tht) \str{P_\tht}{\to} 1 \quad \text{as} \quad t\to \infty.
$$
\end{enumerate}

Under these assumptions using the Ito--Ventzel formula for the process \linebreak $(L(t,\wh\tht_t))_{t\geq 0}$ we get an  ``implicit'' stochastic equation for $\wh\tht=(\wh\tht_t)_{t\geq 0}$. Analyzing the orders of infinitesimality of terms of this equation and rejecting the high order terms we get the following SDE (recursive procedure)
\begin{equation}\label{d5.8}
    d\tht_t=\gm_t(\tht_{t-}) L(dt,\tht_{t-}),
\end{equation}
where $L(dt,u_t)$ is a stochastic line integral w.r.t. the family $\{L(t,u)$, $u\in R^1$, $t\in R_+\}$ of $P_\tht$-special semimartingales along the predictable curve $u=(u_t)_{t\geq 0}$.

To give an explicit form to the SDE (\ref{d5.8}) for the statistical model associated with the semimartingale $X$ assume for a moment that for each $(u,\tht)$ (including the case $u=\tht$)
\begin{equation}\label{d5.9}
    |\Phi(u)|*\mu\in \cA_{\loc}^+(P_\tht).
\end{equation}

Then for each pair $(u,\tht)$ we have
$$
    \Phi(u)*(\mu-\nu_u)=\Phi(u)* (\mu-\nu_\tht) +\Phi(u)
        \bigg(1-\frac{Y(u)}{Y(\tht)}\bigg)*\nu_\tht.
$$

Based on this equality one can obtain the canonical decomposition of $P_\tht$-special semimartingale $L(u)$ (w.r.t. measure $P_\tht$):
\begin{align}
    L(u) & =\dot{\bt}(u) \circ (X^c-\bt(\tht)\circ C)+\Phi(u)*(\mu-\nu_\tht) \notag \\
    & \quad + \dot{\bt}(u) (\bt(\tht)-\bt(u))\circ C +
        \Phi(u) \bigg(1-\frac{Y(u)}{Y(\tht)}\bigg) *\nu_\tht. \label{d5.10}
\end{align}

Now, using (\ref{d5.10}) the meaning of $L(dt,u_t)$ is
\allowdisplaybreaks
\begin{multline*}
    \int_0^t L(ds,u_{s-}) = \int_0^t \dot{\bt}_s(u_{s-}) d(X^c-\bt(\tht)\circ C)_s \\
    + \int_0^t \int \Phi(s,x,u_{s-}) (\mu-\nu_\tht)(ds, dx)+
        \int_0^t \dot{\bt}_s(u_{s}) (\bt_s(\tht)-\bt_s(u_s))dC_s \\
    + \int_0^t \int \Phi(s,x,u_{s-}) \bigg(1-\frac{Y(s,x,u_{s-})}{Y(s,x,\tht)}\bigg)
        \nu_\tht(ds,dx).
\end{multline*}

Finally, the recursive SDE (\ref{d5.8}) takes the form
\allowdisplaybreaks
\begin{align}
    \tht_t=\tht_0& + \int_0^t \gm_s(\tht_{s-}) \dot{\bt}_s(\tht_{s-})
        d(X^c-\bt(\tht)\circ C)_s \notag \\
    & + \int_0^t \int \gm_s(\tht_{s-}) \Phi(s,x,\tht_{s-}) (\mu-\nu_\tht)
        (ds,dx) \notag \\
    & + \int_0^t \gm_s(\tht) \dot{\bt}_s(\tht_s)(\bt_s(\tht)-\bt_s(\tht_s))dC_s \notag \\
    & + \int_0^t \int \gm_s(\tht_{s-}) \Phi(s,x,\tht_{s-})
        \bigg(1-\frac{Y(s,x,\tht_{s-})}{Y(s,x,\tht)}\bigg) \nu_\tht(ds,dx). \label{d5.11}
\end{align}

\begin{remark}\label{dr5.1}
One can give more accurate than (\ref{d5.9}) sufficient conditions (see, e.g., \cite{12n}, \cite{14n}, \cite{proc12}) to ensure the validity of decomposition (\ref{d5.10}).
\end{remark}

Assume that there exists an unique strong solution $(\tht_t)_{t\geq 0}$ of the SDE (\ref{d5.11}).

To investigate  the asymptotic properties of recursive estimators $(\tht_t)_{t\geq 0}$ as $t\to \infty$, namely, a strong consistency, rate of convergence and asymptotic expansion we reduce the SDE (\ref{d5.11}) to the Robbins--Monro type SDE.

For this aim denote $z_t=\tht_t-\tht$. Then (\ref{d5.11}) can be rewritten as
\begin{align}
    z_t=z_0& + \int_0^t \gm_s(\tht+z_{s-}) \dot{\bt}(\tht+z_{s-})
        (\bt_s(\tht)-\bt_s(\tht+z_{s-}) dC_s \notag \\
    & + \int_0^t \int \gm_s(\tht+z_{s-}) \Phi(s,x,\tht+z_{s-})
        \bigg(1-\frac{Y(s,x,\tht+z_{s-})}{Y(s,x,\tht)}\bigg) \nu_\tht
        (ds,dx) \notag \\
    & + \int_0^t \gm_s(\tht+z_{s}) \dot{\bt}_s(\tht+z_s)d(X^c-\bt(\tht)\circ C)_s
        \notag \\
    & + \int_0^t \int \gm_s(\tht+z_{s-}) \Phi(s,x,\tht+z_{s-})
        (\mu-\nu_\tht)(ds,dx). \label{d5.12}
\end{align}

For the definition of the objects $K^\tht$, $\{H^\tht(u)$, $u\in R^1\}$ and $\{M^\tht(u)$, $u\in R^1\}$ we consider such a version of characteristics $(C,\nu_\tht)$ that
\begin{align*}
    C_t & =C^\tht \circ A_t^\tht, \\
    \nu_\tht(\om,dt,dx) & =dA_t^\tht B_{\om,t}^\tht (dx),
\end{align*}
where $A^\tht=(A_t^\tht)_{t\geq 0} \in \cA_{\loc}^+(P_\tht)$, $C^\tht=(C_t^\tht)_{t\geq 0}$ is a nonnegative predictable process, and $B_{\om,t}^\tht(dx)$ is a transition kernel from $(\Om\times R_+,\cP)$ in $(R,\cB(R))$ with $B_{\om,t}^\tht(\{0\})=0$ and
$$
    \Dl A_t^\tht B_{\om,t}^\tht(R)\leq 1
$$
(see \cite{14n}, Ch. 2, \S 2, Prop. 2.9).

Put $K_t^\tht =A_t^\tht,$
\begin{align}
    H_t^\tht(u) & =\gm_t(\tht+u) \bigg\{ \dot{\bt}_t(\tht+u)
        (\bt_t(\tht)-\bt_t(\tht+u)) C_t^\tht \notag \\
    & \quad +\int \phi(t,x,\tht+u) \bigg(1-\frac{Y(t,x,\tht+u)}{Y(t,x,\tht)}\bigg)
        B_{\om,t}^\tht(dx)\bigg\}, \label{d5.13} \\
    M^\tht(t,u) & =\int_0^t \gm_s(\tht+u) \dot{\bt}_s(\tht+u)
        d(X^c-\bt(\tht)\circ C)_s \notag \\
    & \quad +\int_0^t \int \gm_s(\tht+u) \Phi(s,x,\tht+u)(\mu-\nu_\tht)(ds,dx).
        \label{d5.14}
\end{align}

Assume that for each $u$ $M^\tht(u)=(M^\tht(t,u))_{t\geq 0} \in \cM_{\loc}^2(P_\tht)$. Then
\begin{align*}
    \la M^\tht(u)\ra_t & =\int_0^t (\gm_s(\tht+u)\dot{\bt}_s(\tht+u))^2
        C_s^\tht dA_s^\tht \\
    & \quad + \int_0^t \gm_s^2(\tht+u) \bigg( \int \Phi^2(s,x,\tht+u)
        B_{\om,s}^\tht(dx)\bigg) dA_s^{\tht,c} \\
    & \quad + \int_0^t \gm_s^2(\tht+u) B_{\om,t}^\tht(R)
        \bigg\{ \int \Phi^2(s,x,\tht+u)  q_{\om,s}^\tht(dx) \\
    & \quad -a_s(\tht)
        \bigg( \int \Phi(s,x,\tht+u) q_{\om,s}^\tht(dx)\bigg)^2\bigg\}
            dA_s^{\tht,d},
\end{align*}
where $a_s(\tht)=\Dl A_s^\tht B_{\om,s}^\tht(R)$, $q_{\om,s}^\tht(dx) I_{\{a_s(\tht)>0\}}= \frac{B_{\om,s}^\tht(dx)}{B_{\om,s}^\tht(R)} \,I_{\{a_s(\tht)>0\}}$.

Now we give a more detailed description of $\Phi(\tht)$, $\wh I(\tht)$, $H^\tht(u)$ and $\la M^\tht(u)\ra$. Denote
$$
    \frac{d\nu_\tht^c}{d\nu^c} :=F(\tht), \quad
        \frac{q_{\om,t}^\tht(dx)}{q_{\om,t}(dx)} :=
        f_{\om,t}(x,\tht) \;\;\;(:=f_t(\tht)).
$$
Then
$$
    Y(\tht) =F(\tht) I_{\{a=0\}}+\frac{a(\tht)}{a} \,f(\tht) I_{\{a>0\}}
$$
and
$$
    \dot{Y}(\tht) =\dot{F}(\tht) I_{\{a=0\}} +
        \left( \frac{\dot{a}(\tht)}{a}\,f(\tht) +
        \frac{a(\tht)}{a}\,\dot{f}(\tht)\right) I_{\{a>0\}}.
$$
Therefore
\begin{equation}\label{d5.15}
    \Phi(\tht) =\frac{\dot{F}(\tht)}{F(\tht)}\, I_{\{a=0\}} +
        \bigg\{ \frac{\dot{f}(\tht)}{f(\tht)} +
        \frac{\dot{a}(\tht)}{a(\tht)(1-a(\tht))} \bigg\} I_{\{a>0\}}
\end{equation}
with $I_{\{a(\tht)>0\}} \int \frac{\dot{f}(\tht)}{f(\tht)}\,q^\tht(dx)=0$.

Denote $\dot{\bt}(\tht) =\ell^c(\tht)$, $\frac{\dot{F}(\tht)}{F(\tht)} := \ell^\pi(\tht)$, $\frac{\dot{f}(\tht)}{f(\tht)} :=\ell^\dl(\tht)$, $\frac{\dot{a}(\tht)}{a(\tht)(1-a(\tht))} :=\ell^b(\tht)$.

Indices $i=c,\pi,\dl,b$ carry the following loads: ``$c$'' corresponds to the continuous part, ``$\pi$''   to the Poisson type part, ``$\dl$'' to the predictable moments of jumps (including a main special case -- the discrete time case), ``$b$'' to the binomial type part of the likelihood score $\ell(\tht)=(\ell^c(\tht), \ell^\pi(\tht), \ell^\dl(\tht), \ell^b(\tht))$.

In these notations we have  for the Fisher information process:
\allowdisplaybreaks
\begin{align}
    \wh I_t(\tht) & =\int_0^t (\ell_s^c(\tht))^2 dC_s +
        \int_0^t \int (\ell_s^\pi(x;\tht))^2 B_{\om,s}^\tht(dx) dA_s^{\tht,c} \notag \\
    & \quad + \int_0^t B_{\om,s}^\tht(R)
        \bigg[ \int(\ell_s^\dl(x;\tht))^2 q_{\om,s}^\tht(dx) \bigg] dA_s^{\tht,d}
            \notag \\
    & \quad +\int_0^t (\ell_s^b(\tht))^2 (1-a_s(\tht))dA_s^{\tht,d}. \label{d5.16}
\end{align}

For the random field $H^\tht(u)$ we have:
\begin{align}
    H_t^\tht(u) & =\gm_t(\tht+u) \bigg\{ \ell_t^c(\tht+u)
        (\bt_t(\tht)-\bt_t(\tht+u))C_t^\tht \notag \\
    & \quad +\int \ell_t^\pi(x;\tht+u) \bigg( 1-\frac{F_t(x;\tht+u)}{F_t(x;\tht)}\bigg)
        B_{\om,t}^\tht(dx) I_{\{\Dl A_t^\tht=0\}} \notag \\
    & \quad +\bigg\{ \int \ell_t^\dl(x;\tht+u) q_{\om,t}^\tht(dx)  \notag \\
    & \hskip+2cm +
        \ell_t^b(\tht+u) \,\frac{a_t(\tht)-a_t(\tht+u)}{a_t(\tht)}\bigg\}
        B_{\om,t}^\tht(R) I_{\{\Dl A_t^\tht>0\}}. \label{d5.17}
\end{align}
Finally, we have for $\la M^\tht(u)\ra$:
\begin{align}
    \la M^\tht(u)\ra_t & = \left(\gm(\tht+u) \ell^c(\tht+u)\right)^2 C^\tht \circ A_t^\tht \notag \\
    & \quad + \int_0^t \gm_s^2(\tht+u)\int (\ell_s^\pi(x;\tht+u))^2
        B_{\om,t}^\tht(dx)dA_s^{\tht,c} \notag \\
    & \quad + \int_0^t \gm_s^2(\tht+u) B_{\om,s}^\tht(R)
        \bigg\{ \int (\ell_s^\dl(x;\tht+u)+\ell_s^b(\tht+u))^2 q_{\om,s}^\tht(dx)
        \notag \\
    & \quad - a_s(\tht) \bigg( \int (\ell_s^\dl(x;\tht+u)+\ell_s^b(\tht+u))
        q_{\om,s}^\tht(dx) \bigg)^2 \bigg\} dA_s^{\tht,d}. \label{d5.18}
\end{align}

Thus, we reduced SDE (\ref{d5.12}) to the Robbins--Monro type SDE with $K_t^\tht=A_t^\tht$, and $H^\tht(u)$ and $M^\tht(u)$ defined by (\ref{d5.17}) and (\ref{d5.14}), respectively.

As it follows from (\ref{d5.17})
$$
    H_t^\tht(0)=0 \quad \text{for all} \;\;\; t\geq 0, \;\; P_\tht\text{-}a.s.
$$

As for condition (A)  to be satisfied it ie enough to require that for all $t\geq 0$, $u\neq 0$ $P_\tht$-$a.s$.
\begin{gather*}
    \dot{\bt}_t(\tht+u) (\bt_t(\tht)-\bt_t(\tht+u))<0, \\
    \bigg( \int \frac{\dot{F}(t,x,\tht+u)}{F(t,x,\tht+u)}
        \bigg(1-\frac{F(t,x;\tht+u)}{F(t,x;\tht)} \bigg) B_{\om,t}^\tht (dx) \bigg)
            I_{\{\Dl A_t^\tht=0\}}u<0, \\
    \bigg( \int \frac{\dot{f} (t,x;\tht+u)}{f(t,x;\tht+u)} \,q_t^\tht(dx)\bigg)
        I_{\{\Dl A_t^\tht>0\}}u<0,\\
    \dot{a}_t(\tht+u)(a_t(\tht)-a_t(\tht+u))u<0,
\end{gather*}
and the simplest sufficient conditions for the latter ones is the monotonicity ($P$-a.s.) of functions $\bt(\tht)$, $F(\tht)$, $f(\tht)$ and $a(\tht)$ w.r.t $\tht$.

\begin{remark}\label{dr5.3}
In the case when the model is regular in the  Jacod sense only we save the same form of all above-given objects (namely of $\Phi(\tht)$) using the formal definitions:
\begin{align*}
    & \frac{\dot{F}(\tht)}{F(\tht)}\,I_{\{a(\tht)=0\}} := W(\tht) I_{\{a(\tht)=0\}}, \\
    & \dot{a}(\tht):= \wh W^\tht, \\
    & \frac{\dot{f}(\tht)}{f(\tht)} := W(\tht) I_{\{a(\tht)>0\}} -
        \frac{\wh W^\tht(\tht)}{a(\tht)} \,I_{\{a(\tht)>0\}}.
\end{align*}
\end{remark}

\subsubsection{Discrete time}
\subsubsection*{a) Recursive MLE in parameter statistical models}
Let $X_0,X_1,\dots,X_n,\dots$ be observations taking values in some measurable space $(\cX,\cB(\cX))$ such that the regular conditional densities of distributions (w.r.t. some measure $\mu$) $f_i(x_i,\tht| x_{i-1},\dots,x_0)$, $i\leq n$, $n\geq 1$ exist, $f_0(x_0,\tht)\equiv f_0(x_0)$, $\tht\in R^1$ is the parameter to be estimated. Denote $P_\tht$ corresponding distribution on $(\Om,\cF):=(\cX^\infty, \cB(\cX^\infty))$. Identify the process $X=(X_i)_{i\geq 0}$ with coordinate process and denote $\cF_0=\sg(X_0)$, $\cF_n=\sg$ $(X_i$, $i\leq n)$. If $\psi=\psi(X_i,X_{i-1},\dots,X_0)$ is a r.v., then under $E_\tht(\psi|\cF_{i-1})$ we mean the following version of conditional expectation
$$
    E_\tht(\psi \mid \cF_{i-1}):= \int \psi(z,X_{i-1},\dots,X_0)
        f_i(z,\tht \mid X_{i-1},\dots,X_0)\mu(dz),
$$
if the last integral exists.

Assume that the usual regularity conditions are satisfied and denote
$$
    \frac{\pa}{\pa\tht} \,f_i(x_i,\tht \mid x_{i-1},\dots,x_0) :=
        \dot{f}_i(x_i,\tht \mid x_{i-1},\dots,x_0),
$$
the maximum likelihood scores
$$
    l_i(\tht) :=\frac{\dot{f}_i}{f_i}\,(X_i,\tht \mid X_{i-1},\dots,X_0)
$$
and the empirical Fisher information
$$
    I_n(\tht):= \sum_{i=1}^n E_\tht(l_i^2(\tht) \mid \cF_{i-1}).
$$
Denote also
$$
    b_n(\tht,u):= E_\tht(l_n(\tht+u) \mid \cF_{n-1})
$$
and indicate that for each $\tht \in R^1$, $n\geq 1$
\begin{equation}\label{5.1}
    b_n(\tht,0)=0 \;\;(P_\tht\text{-}a.s.).
\end{equation}
Consider the following recursive procedure
$$
    \tht_n=\tht_{n-1} +I_n^{-1} (\tht_{n-1}) l_n(\tht_{n-1}), \quad
        \tht_0\in \cF_0.
$$
Fix $\tht$, denote $z_n=\tht_n-\tht$ and rewrite the last equation in the form
\begin{equation}\label{5.2}
\begin{aligned}
    & z_n=z_{n-1}+I_n^{-1} (\tht+z_{n-1}) b_n(\tht,z_{n-1}) +
        I_n^{-1}(\tht+z_{n-1}) \Dl m_n, \\
    & z_0 =\tht-\tht,
\end{aligned}
\end{equation}
where $\Dl m_n=\Dl m(n,z_{n-1})$ with $\Dl m(n,u)=l_n(\tht+u)-E_\tht(l_n(\tht+u) |\cF_{n-1})$.

Note that the algorithm (\ref{5.2}) is embedded in stochastic approximation scheme (\ref{3.1}) with
\begin{align*}
    H_n(u) & =I_n^{-1}(\tht+u)b_n(\tht,u)\in \cF_{n-1}, \quad \Dl K_n=1, \\
    \Dl M(n,u) & =I_n^{-1} (\tht+u) \Dl m(n,u).
\end{align*}

This example clearly shows the necessity  of consideration of random fields $H_n(u)$ and $M(n,u)$.

In Sharia \cite{20} the convergence $z_n\to 0$ $\Pas$ as $n\to \infty$ was proved under conditions equivalent to (A), (B) and (I) connected with standard representation (\ref{3.2})(1).

\begin{remark}\label{r5.1}
Let $\tht\in \Tht\sbs R^1$ where $\tht$ is open proper subset of $R^1$. It may be possible that the objects $l_n(\tht)$ and $I_n(\tht)$ are defined only on the set $\Tht$, but for each fixed $\tht\in \Tht$ the objects $H_n(u)$ and $M(n,u)$ are well-defined functions of variable $u$ on whole $R^1$. Then under conditions of Theorem \ref{t3.1} $\tht_n\to \tht$ $P_\tht$-$a.s.$ as $n\to \infty$ starting from arbitrary $\tht_0$. The example given below illustrates this situation. The same example illustrates also efficiency of the representation (\ref{3.3})(2).
\end{remark}

\subsubsection*{b) Galton--Watson Branching Process with Immigration}
Let the observable process be
$$
    X_i=\sum_{j=1}^{X_{i-1}} Y_{i,j}+1, \quad i=1,2,\dots,n; \quad X_0=1,
$$
$Y_{i,j}$ are i.i.d. random variables having the Poisson distribution with parameter $\tht$, $\tht>0$, to be estimated. If $\cF_i=\sg(X_j$, $j\leq i)$, then
$$
    P_\tht(X_i=m \mid \cF_{i-1}) = \frac{(\tht X_{i-1})^{m-1}}{(m-1)!}\,
        e^{-\tht X_{i-1}}, \quad i=1,2,\dots\,; \;\;\; m\geq 1.
$$
From this we have
$$
    l_i(\tht) =\frac{X_i-1-\tht X_{i-1}}{\tht}\,, \quad
    I_n(\tht)=\tht^{-1} \sum_{i=1}^n X_{i-1}.
$$
The recursive procedure has the form
\begin{equation}\label{5.3}
    \tht_n=\tht_{n-1} +\frac{X_n-1-\tht_{n-1}X_{n-1}}{\sum_{i=1}^n X_{i-1}}\,,
        \quad \tht_0\in \cF_0,
\end{equation}
and if, as usual $z_n=\tht_n-\tht$, then
\begin{equation}\label{5.4}
    z_n=z_{n-1} -\frac{z_{n-1}X_{n-1}}{\sum_{i=1}^n X_{i-1}} +
        \frac{\ve_n}{\sum_{i=1}^n X_{i-1}}\,,
\end{equation}
where $\ve_n=X_n-1-\tht X_n$ is a $P_\tht$-square integrable martingale-difference. In fact, $E_\tht(\ve_n \,|\,\cF_{n-1})=0$, $E_\tht(\ve_n^2 \,|\, \cF_{n-1})=\tht X_{n-1}$. In this case $H_n(u)=-uX_{n-1} /\sum\limits_{i=1}^n X_{i-1}$, $\Dl M(n,u)=\Dl m_n=\ve_n /\sum\limits_{i=1}^n X_{i-1}$, $\Dl K=1$ and so are well-defined on whole $R^1$.

Indicate now that the solution of Eq. (\ref{5.3}) coincides with MLE
$$
    \wh\tht_n =\frac{\sum_{i=1}^n (X_i-1)}{\sum_{i=1}^n X_{i-1}}
$$
and it is easy to see that $(\wh\tht_n)_{n\geq 1}$ is strongly consistent for all $\tht>0$.

Indeed,
$$
    \wh\tht_n-\tht =\frac{\sum_{i=1}^n \ve_i}{\sum_{i=1}^n X_{i-1}}
$$
and desirable follows from strong law of large numbers for martingales and well-known fact (see, e.g., \cite{4}) that for all $\tht>0$
\begin{equation}\label{5.5}
    \sum_{i=1}^\infty X_{i-1} =\infty \;\; (P_\tht\text{-}a.s.).
\end{equation}

Derive this result as the corollary of Theorem \ref{t3.1}.

Note at first that for each $\tht>0$ the conditions (A) and (B$'$) are satisfied. Indeed,

(A) $\ds \qquad \qquad H_n(u)u=\frac{-u^2X_{n-1}}{\sum_{i=1}^n X_{i-1}}<0$ \\
for all $u\neq 0$ $(X_i>0$, $i\geq 0)$;
\smallskip

(B$'$) $\ds \qquad \qquad \la m\ra_\infty =\tht \sum_{n=1}^\infty
    \frac{X_{n-1}}{(\sum_{i=1}^n X_{i-1})^2} <\infty,$ \\
thanks to (\ref{5.5}).

Now to illustrate the efficiency of group of conditions (II) let us consider two cases:

1) $0<\tht\leq 1$ and \;\; 2) $\tht$ is arbitrary, i.e. $\tht>0$.

In case 1) conditions (I) are satisfied. In fact, $|H_n(u)|\!=\!\Big(X_{n-1}\big/\sum\limits_{i=1}^n X_{i-1}\Big) |u|$ and $\sum\limits_{n=1}^\infty X_{n-1}^2 \big/ \Big(\sum\limits_{i=1}^n X_{i-1}\Big)^2<\infty$, $P_\tht$-$a.s.$ But if $\tht>1$ the last series diverges, so the condition (I)~(i) is not satisfied.

On the other hand, the proving of desirable convergence by checking  the conditions (II) is almost trivial. Really, use Remark \ref{r4.2} and take $\wt G_n=G_n=X_{n-1}/\sum\limits_{i=1}^n X_{i-1}$. Then $\sum\limits_{n=1}^\infty G_n=\infty$ $P_\tht$-$a.s.$, for all $\tht>0$. Besides $\dl_n=-2+\wt G_n<0$, $|\dl_n|\geq 1$.

\subsubsection{RM Algorithm with Deterministic Regression Function}

Consider the particular case of algorithm (\ref{3.1}) when $H_t(\om,u)=\gm_t(\om) R(u)$, where the process $\gm=(\gm_t)_{t\geq 0} \in \cP$, $\gm_t>0$ for all $t\geq 0$, $dM(t,u)=\gm_t dm_t$, $m\in \cM_{\loc}^2$, i.e.
$$
    dz_t=\gm_t R(z_{t-}) dK_t +\gm_t dm_t, \quad z_0\in \cF_0.
$$

a) Let the following conditions be satisfied:

(A) $R(0)=0$, $R(u)u<0$ for all $u\neq 0$,

(B$'$) $\gm^2\circ \la m\ra_\infty<\infty$,

(1) $|R(u)|\leq C(1+|u|)$, $C>0$ is constant,

(2) for each $\ve>0$, $\inf\limits_{\ve\leq u\leq \frac{1}{\ve}} |R(u)|>0$,

(3) $\gm\circ K_t<\infty$, $\fa t\geq 0$, $\gm\circ K_\infty=\infty$,

(4) $\gm^2 \Dl K\circ K_\infty^d <\infty$.

Then $z_t\to 0$ $P$-$a.s.$, as $t\to \infty$.

Indeed, it is easy to see that (A), (B$'$), (1)--(4)$\Ra$(A), (B) and (I) of Theorem \ref{t3.1}.

In Melnikov \cite{14} this result has been proved on the basis of the theorem on the semimartingale convergence sets noted in Remark \ref{r2.1}. In the case when $K^d\neq 0$ this automatically leads to the ``moment'' restrictions and the additional assumption $|R(u)|\leq const$.

b) Let, as in case a), conditions (A) and (B$'$) be satisfied. Besides assume that for each $u\in R^1$ and $t\in [0,\infty)$:
\smallskip

(1$'$) \qquad $ V_t^-(u)+V_t^+(u)\leq 0$,
\smallskip

(2$'$) \qquad for all $\ve>0$
$$
    I_\ve:= \inf_{\ve\leq u\leq \frac{1}{\ve}} \{-(V^-(u)+V^+(u))\} \circ K_\infty=\infty.
$$
\smallskip

Then $z_t\to 0$ $P$-$a.s.$, as $t\to \infty$.

Indeed, it is not hard to verify that (1$'$), (2$')\Ra$(II).

The following question arises: is it possible (1$'$) and (2$'$) to be satisfied? Suppose in addition that
\begin{equation}\label{5.6}
    C_1|u|\leq |R(u)| \leq C_2|u|, \;\;\; C_1,C_2 \;\;\text{are constants},
\end{equation}

(3$'$) \qquad $2-C_2\gm_t\Dl K_t\geq 0$,
\smallskip

(4$'$) \qquad $\gm(2-C_2\gm \Dl K)\circ K_\infty=\infty.$ \\[2mm]
Then $(3')\Ra(1')$ and $(4') \Ra(2')$.

Indeed,
\begin{gather*}
    V_t^-(u) +V_t^+(u)  \leq C_1 \gm_t|u|^2[-2+C_2\gm_t \Dl K_t]\leq 0, \\
    I_\ve  \geq C_1 \ve^2\{\gm(2-C_2\gm \Dl K)\circ K_\infty\}=\infty.
\end{gather*}

\begin{remark}\label{r5.2}
$(4')\Ra \gm\circ K_\infty=\infty$.
\end{remark}

In \cite{16} the convergence $z_t\to 0$ $P$-$a.s.$, as $t\to \infty$ was proved under the following conditions:

(A) $R(0)=0$, $R(u)u<0$ for all $u\neq 0$;

(M) there exists a non-negative predictable process $r=(r_t)_{t\geq 0}$ integrable w.r.t process $K=(K_t)_{t\geq 0}$ on any finite interval $[0,t]$ with properties:
\begin{enumerate}
\item[(a)] $r\circ K_\infty=\infty$,
\item[(b)] $A_\infty^1 =\gm^2 \cE^{-1} (-r\circ K)\circ \la m\ra_\infty<\infty$,
\item[(c)] all jumps of process $A^1$ are bounded,
\item[(d)] $r_tu^2+\gm_t^2 \Dl K_tR^2(u)\leq -2\gm_t R(u)u$, \\ for all $u\in R^1$ and $t\in [0,\infty)$.
\end{enumerate}
Show that (M)$\Ra$(B$'$), (1$'$) and (2$'$).

It is evident that (b)$\Ra$(B$'$). Further, (d)$\Ra$(1$'$), Finally, (2$'$) follows from (a) and (d) thanks to the relation
$$
    I_\ve :=\inf_{\ve\leq |u|\leq \frac{1}{\ve}} -(V^-(u)+V^+(u))\circ K_\infty \geq
        \ve^2r\circ K_\infty=\infty.
$$
The implication is proved.

In particular case when (\ref{5.6}) holds and for all $t\geq 0$ $\gm_t\Dl K_t\leq q$, $q>0$ is a constant and $C_1$ and $C_2$ in (\ref{5.6}) are chosen  such that $2C_1-qC_2^2>0$, if we take $r_t=b\gm_t$, $b>0$, with $b<2C_1-qC_2^2$, then (a) and (d) are satisfied if $\gm\circ K_\infty=\infty$.

But these conditions imply (3$'$) and (4$'$). In fact, on the one hand, $0<2C_1-qC_2^2\leq C_1(2-qC_2)$ and so (3$'$) follows, since $2-C_2\gm_t\Dl K_t\geq 2-qC_2>0$. On the other hand, (4$'$) follows from $\gm(2-C_2\gm \Dl K)\circ K_\infty\geq (2-qC_2)\gm \circ K_\infty=\infty$.

From the above we may conclude that if the conditions (A), (B$'$), (\ref{5.6}), $\gm_t\Dl K_t\leq q$, $q>0$, $2-qC_2>0$ and $\gm\circ K_\infty=\infty$ are satisfied, then the desirable convergence $z_t\to 0$ $P$-$a.s.$ takes place and so, the choosing of process $r=(r_t)_{t\geq 0}$ with properties (M) is unnecessary (cf. \cite{16}, Remark \ref{r3.3} and Subsection~\ref{s4}).

\subsubsection*{c) Linear Model $($see, e.g., \cite{14}$)$}
Consider the linear RM procedure
$$
    dz_t=b\gm_t z_{t-} dK_t+\gm_t dm_t, \quad z_0\in \cF,
$$
where $b\in B\sbseq (-\infty,0)$, $m\in \cM_{\loc}^2$.

Assume that
\begin{align}
    \gm^2\circ \la m\ra_\infty & <\infty, \label{5.7} \\
    \gm \circ K_\infty & =\infty, \label{5.8} \\
    \gm^2 \Dl K\circ K^d & <\infty. \notag
\end{align}
Then for each $b\in B$ the conditions (A), (B$'$) and (I) are satisfied. Hence
\begin{equation}\label{5.9}
    z_t\to 0 \;\;P\text{-}a.s., \;\;\; \text{as} \;\;\; t\to \infty.
\end{equation}
Now let (\ref{5.7}) and (\ref{5.8}) be satisfied, but $P(\gm^2 \Dl K\circ K^d=\infty)>0$.

At the same time assume that $B=[b_1,b_2]$, $-\infty<b_1\leq b_2<0$ and for all $t>0$ $\gm_t\Dl K_t <|b_1|^{-1}$.

Then for each $b\in B$ (\ref{5.9}) holds.

Indeed,
\begin{align*}
    [V_t^-(u) I_{\{\Dl K_t\neq 0\}} +V_t^+(u)]^+ & = |b| \gm_tu^2
        [-2+|b| \gm_t\Dl K_t I_{\{\Dl K_t\neq 0\}}]^+ \\
    & \leq  I_{\{\Dl K_t\neq 0\}} |b| \gm_tu^2[-2+|b| \gm_t\Dl K_t]^+=0
\end{align*}
and therefore (II)~(i) is satisfied.

On the other hand,
\begin{multline*}
    \inf_{\ve\leq |u| \leq \frac{1}{\ve}} u^2 \{2\gm|b| I_{\{\Dl K\neq 0\}}+
        b\gm[2-|b|\gm \Dl K] I_{\{\Dl K\neq 0\}} \} \circ K_\infty \\
    \geq \ve^2 |b| \gm [2-|b|\gm \Dl K]\circ K_\infty \geq
        \ve^2 |b| \gm \circ K_\infty=\infty.
\end{multline*}
So (II)~(ii) is satisfied also.


\bigskip
\section{Rate of Convergence and Asymptotic Expansion}
\subsection{Notation and preliminaries}\label{2-s1}
We consider the RM type stochastic differential equation (SDE)
\begin{equation}\label{2-1.1}
    z_t=z_0+\int_0^t H_s(z_{s-}) dK_s +\int_0^t M(ds, z_{s-}).
\end{equation}


As usual, we assume that there exists a unique strong solution $z=(z_t)_{t\geq 0}$ of Eq. (\ref{2-1.1}) on the whole time interval $[0,\infty[$ and $\wt M=(\wt M_t)_{t\geq 0}\in \cM_{\loc}^2(P)$, where $\wt M =\int\limits_0^t M(ds, z_{s-})$ (see \cite{2}, \cite{3}, \cite{7}).

Let us denote
$$
    \bt_t=-\lim_{u\to 0} \frac{H_t(u)}{u}
$$
assuming that this limit exists and is finite for each $t\geq 0$ and define the random field
$$
    \bt_t(u) =\begin{cases}
                -\frac{H_t(u)}{u} & \text{if} \;\; u\neq 0, \\
                \;\;\;\bt_t & \text{if} \;\; u=0.
            \end{cases}
$$

It follows from (A) that for all $t\geq 0$ and $u\in R^1$,
$$
    \bt_t\geq 0 \quad \text{and} \quad \bt_t(u)\geq 0 \;\;(\Pas).
$$
Further, rewrite Eq. (\ref{2-1.1}) as
\begin{align*}
    z_t=z_0 & - \int_0^t  \bt_s z_{s-} I_{\{\bt_s \Dl K_s\neq 1\}} dK_s +
        \int_0^t M(ds,0)-\sum_{s\leq t} z_{s-}I_{\{\bt_s\Dl K_s=1\}} \\
    & + \int_0^t (\bt_s-\bt_s(z_{s-})) z_{s-} dK_s +
        \int_0^t (M(ds,z_{s-})-M(ds,0))
\end{align*}
(we suppose that $M(\cdot,0)\not\equiv 0$).

Denote
\begin{gather*}
    \ol{\bt}_t =\bt_t I_{\{\bt_t \Dl K_t\neq 1\}}, \;\;\;
        \ol{R}{}_t^{(1)} =-\sum_{s\leq t} z_{s-} I_{\{\bt_s \Dl K_s=1\}}, \\
    \ol{R}{}_t^{(2)}=\int_0^t (\bt_s-\bt_s(z_{s-}))z_{s-} dK_s, \;\;\;
        \ol{R}{}_t^{(3)} =\int_0^t (M(ds,z_{s-})-M(ds,0)).
\end{gather*}

In this notation,
$$
    z_t=z_0-\int_0^t \ol{\bt}_s z_{s-} dK_s +
        \int_0^t M(ds,0)+\ol{R}_t,
$$
where
\begin{gather*}
    \ol{R}_t=\ol{R}{}_t^{(1)}+\ol{R}{}_t^{(2)}+\ol{R}{}_t^{(3)}.
\end{gather*}

Solving this equation w.r.t $z$ yields
\begin{equation}\label{2-1.2}
    z_t=\Gm_t^{-1} \bigg( z_0+\int_0^t \Gm_s M(ds,0) +
        \int_0^t \Gm_s d \ol{R}_s\bigg),
\end{equation}
where
$$
    \Gm_t=\ve_t^{-1} (-\ol{\bt}\circ K).
$$

Here, $\al\circ K_t=\int\limits_0^t \al_s dK_s$ and $\ve_t(A)$ is the Dolean exponent.

The process $\Gm=(\Gm_t)_{t\geq 0}$ is predictable (but not positive in general) and therefore, the process $L=(L_t)_{t\geq 0}$ defined by
$$
    L_t=\int_0^t \Gm_s M(ds,0)
$$
belongs to the class $\cM_{\loc}^2(P)$. It follows from Eq. (\ref{2-1.2}) that
$$
    \chi_t z_t=\frac{L_t}{\la L_t\ra_t^{1/2}} +R_t,
$$
where
\begin{gather*}
    \chi_t=\Gm_t \la L\ra_t^{-1/2}, \\
    R_t =\frac{z_0}{\la L\ra _t^{1/2}} +\frac{1}{\la L\ra_t^{1/2}}
        \int_0^t \Gm_s d\ol{R}_s
\end{gather*}
and $\la L\ra$ is the shifted square characteristic of $L$, i.e. $\la L\ra_t:=1+\la L\ra_t^{F,P}$.

This section is organized  as follows. In subsection \ref{2-s2} assuming $z_t\to 0$ as $t\to \infty$ $(\Pas)$, we give various sufficient conditions to ensure the convergence
\begin{equation}\label{2-1.3}
    \gm_t^\dl z_t^2 \to 0 \quad \text{as} \quad t\to \infty \;\; (\Pas)
\end{equation}
for all $0\leq \dl\leq \dl_0$, where $\gm=(\gm_t)_{t\geq 0}$ is a predictable increasing process and $\dl_0$, $0\leq \dl_0\leq 1$, is some constant. There we also give series of examples illustrating these results.

In subsection \ref{2-s3} assuming that Eq. (\ref{2-1.3}) holds with $\gm$ asymptotically equivalent to $\chi^2$ (see the definition in subsection \ref{2-s2}, we study sufficient conditions to ensure the convergence
$$
    R_t \os{P}{\to} 0 \quad \text{as} \quad t\to \infty,
$$
which implies the local asymptotic linearity of the solution. 


{\it We say that the process $\xi=(\xi_t)_{t\geq 0}$ has some property eventually if for every $\om$ in a set $\Om_0$ of $P$ probability $1$, the trajectory $(\xi_t(\om))_{t\geq 0}$ of the process has this property on the set $[t_0(\om),\infty)$ for some $t_0(\om)<\infty$. }

Everywhere in this section we assume that $z_t\to 0$ as $t\to \infty$ $(\Pas)$.

\subsection{Rate of convergence}\label{2-s2}

Throughout subsection \ref{2-s2} we assume that $\gm=(\gm_t)_{t\geq 0}$ is a predictable increasing process such that $(\Pas)$
$$
    \gm_0=1, \quad \gm_\infty=\infty.
$$

Suppose also that for each $u\in \bR^1$ the processes $\la M(u)\ra$ and $\gm$ are locally absolutely continuous w.r.t. the process $K$ and denote
$$
    h_t(u,v)=\frac{d\la M(u),M(v)\ra_t}{d K_t} \quad \text{and} \quad
        g_t =\frac{d\gm_t}{d K_t}
$$
assuming for simplicity that $g_t>0$ and hence, $I_{\{\Dl K_t\neq 0\}}=I_{\{\Dl \gm_t\neq 0\}}$ $(\Pas)$ for all $t>0$.

In this subsection, we study the problem of the convergence
$$
    \gm_t^\dl z_t\to 0 \quad \text{as} \quad t\to \infty \;\; (\Pas)
$$
for all $\dl$, $0<\dl<\dl_0/2$, $0<\dl_0\leq 1$.

It should be stressed that the consideration of the two control parameters $\dl$ and $\dl_0$ substantially simplifies application of the results and also clarifies their relation with the classical ones (see Examples 1 and 6).

We shall consider two approaches to this problem. The first approach is based on the results on the convergence sets of non-negative semimartingales and on the so-called ``non-standard representations''.

The second approach presented exploits the stochastic version of the Kronecker Lemma. This approach is employed in \cite{20} for the discrete time case under the assumption (\ref{2-2.20}). The comparison of the results obtained in this section with those obtained before is also presented.

Note also that the two approaches give different sets of conditions in general. This fact is illustrated by the various examples.

Let us formulate some auxiliary results based on the convergence sets.

Suppose that $r=(r_t)_{t\geq 0}$ is a non-negative predictable process such that
$$
    r_t \Dl K_t<0, \quad r\circ K_t <\infty \;\; (\Pas)
$$
for each $t>0$ and
$$
    r\circ K_\infty=\infty \;\; (\Pas).
$$

Denote by $\ve_t=\ve_t(-r\circ K)$ the Dolean exponential, i.e.
$$
    \ve_t=e^{-\int_0^t r_s d K_s^c} \prod_{s\leq t} (1-r_s \Dl K_s).
$$

Then, as it is well known (see \cite{proc12}, \cite{14}), the process $\ve_t^{-1} =\{\ve_t(-r\circ K)\}^{-1}$ is the solution of the linear SDE
$$
    \ve_t^{-1} =\ve_t^{-1} r_t dK_t, \quad \ve_0^{-1}=1
$$
and $\ve_t^{-1}\to \infty $ as $t\to \infty$ ($\Pas$).

\begin{proposition}\label{2-p2.1}
Suppose that
\begin{equation}\label{2-2.1}
    \int_0^\infty \ve_t^{-1} \ve_{t-} [r_t-2\bt_t(z_{t-}) +
        \bt_t^2 (z_{t-})\Dl K_t]^+ dK_t<\infty \;\; (\Pas)
\end{equation}
and
\begin{equation}\label{2-2.2}
    \int_0^\infty \ve_t^{-1} h_t(z_{t-},z_{t-}) dK_t<\infty \;\; (\Pas),
\end{equation}
where $[x]^+$ denotes the positive part of $x$.

Then $\ve^{-1} z^2 \to (\Pas)$ $($the notation $X\to$ means that $X=(X_t)_{t\geq 0}$ has a finite limit as $t\to \infty)$.
\end{proposition}

\begin{proof}
Using the Ito formula,
\begin{align*}
    d(\ve_t^{-1} z_t^2) & =z_{t-}^2 d\ve_t^{-1} +\ve_t^{-1} dz_t^2 \\
    & =\ve_t^{-1} z_{t-}^2 (r_t-2\bt_t(z_{t-})+\bt_t^2(z_{t-}) \Dl K_t) dK_t \\
    & \quad + \ve_t^{-1} h_t(z_{t-},z_{t-}) dK_t +d(\text{Mart}) \\
    & =\ve_t^{-1} z_{t-}^2 dB_t+d A_t^1-d A_t^2+d(\text{Mart}),
\end{align*}
where
\begin{align*}
    dB_t & =\ve_t^{-1} \ve_{t-} \left[ r_t-2\bt_t(z_{t-})+\bt_t^2(z_{t-})
        \Dl K_t\right]^+ d K_t, \\
    dA_t^1 & =\ve_t^{-1} h_t (z_{t-},z_{t-}) dK_t, \\
    dA_t^2 & =\ve_t^{-1} \ve_{t-} \left[ r_t-2\bt_t(z_{t-})+\bt_t^2(z_{t-})
        \Dl K_t\right]^- d K_t.
\end{align*}

Now, applying Corollary \ref{c2.4} to the non-negative semimartingale $(\ve_t^{-1} z_t^2)_{t\geq 0}$, we obtain
$$
    \{B_\infty<\infty\}\cap \{A_\infty^1<\infty\} \sbseq \{\ve^{-1} z^2\to\} \cap
        \{A_\infty^2<\infty\}
$$
and the result follows from Eqs. (\ref{2-2.1}) and (\ref{2-2.2}).
\end{proof}

The following lemma is an immediate consequence of the Ito formula applying to the process $(\gm_t^\dl)_{t\geq 0}$, $0<\dl<1$.

\begin{lemma}\label{2-l2.1}
Suppose that $0<\dl<1$. Then
$$
    \gm_t^\dl=\ve_t^{-1} (-r^\dl \circ K),
$$
where
$$
    r_t^\dl =\ol{r}{}_t^\dl g_t /\gm_t
$$
and
$$
    \ol{r}_t^\dl =\dl I_{\{\Dl \gm_t=0\}} +
        \frac{1-(1-\Dl \gm_t/\gm_t)^\dl}{\Dl \gm_t/\gm_t} \,
        I_{\{\Dl \gm_t\neq 0\}}.
$$
\end{lemma}

The following theorem is the main result based on the first approach.

\begin{theorem}\label{2-t2.1}
Suppose that for each $\dl$, $0<\dl<\dl_0$, $0<\dl_0\leq 1$,
\begin{equation}\label{2-2.3}
    \int_0^\infty \left( \frac{\gm_{t-}}{\gm_t}\right)^{-\dl}
        [r_t^\dl -2\bt_t(z_{t-})+\bt_t^2(z_{t-})\Dl K_t]^+ dK_t<\infty \;\; (\Pas)
\end{equation}
and
\begin{equation}\label{2-2.4}
    \int_0^\infty \gm_t^\dl h_t(z_{t-},z_{t-}) dK_t<\infty \;\; (\Pas).
\end{equation}

Then $\gm_t^\dl z_t^2 \to 0$ as $t\to \infty$ $(\Pas)$ for each $\dl$, $0<\dl<\dl_0$, $0<\dl_0\leq 1$.
\end{theorem}

\begin{proof}
It follows from Proposition \ref{2-p2.1}, Lemma \ref{2-l2.1} and the conditions (\ref{2-2.3}) and (\ref{2-2.4}) that
$$
    P\{\gm^\dl z^2\to\}=1
$$
for all $\dl$, $0<\dl<\dl_0$, $0<\dl_0\leq 1$. Now the result follows since
\[
    \{\gm^\dl z^2 \to \;\;\text{for all} \;\; \dl, \; 0<\dl<\dl_0\} \Ra
    \{\gm^\dl z^2 \to 0 \;\;\text{for all} \;\; \dl, \; 0<\dl<\dl_0\}. \qedhere
\]
\end{proof}

\begin{remark}\label{2-r2.1}
Note that if Eq. (\ref{2-2.3}) holds for $\dl=\dl_0$, than it holds for all $\dl\leq \dl_0$.
\end{remark}

Some simple conditions ensuring Eq. (\ref{2-2.3}) are given in the following corollaries.

\begin{corollary}\label{2-c2.1}
Suppose that the process
\begin{equation}\label{2-2.5}
    \frac{\gm}{\gm_-} \;\; \text{is eventually bounded.}
\end{equation}

Then for each $\dl$, $0<\dl<\dl_0$, $0<\dl_0\leq 1$,
\begin{align*}
    & \bigg\{ \bigg[ (\dl I_{\{\Dl \gm=0\}} +I_{\{\Dl \gm\neq 0\}} \,
        \frac{g}{\gm} -2\bt(z_-)+\bt^2(z_-) \Dl K\bigg]^+
        \circ K_\infty<\infty \bigg\} \\
    & \quad \sbseq \bigg\{ \bigg[ \bigg(\dl+(1-\dl)\,\frac{\Dl\gm}{\gm}\bigg)
        \frac{g}{\gm} -2\bt(z_-)+\bt^2(z_-)\Dl K\bigg]^+\circ K_\infty<\infty \bigg\}\\
    & \quad \sbseq \bigg\{ \bigg( \frac{\gm_-}{\gm}\bigg)^{-\dl} [r^\dl-2\bt(z_-) +
        \bt^2(z_-) \Dl K]^+ \circ K_\infty<\infty \bigg\}.
\end{align*}
\end{corollary}

\begin{proof}
The proof immediately follows from the following simple inequalities
$$
    1-(1-x)^\dl\leq \dl x+(1-\dl)x^2\leq x
$$
if $0<x<1$ and $0<\dl<1$, which taking $x=\Dl \gm_t/\gm_t$ gives
$$
    \ol{r}{}_t^\dl \leq \left( \dl+(1-\dl)\,\frac{\Dl\gm_t}{\gm_t}\right) \leq
        \big(\dl I_{\{\Dl \gm_t=0\}} +I_{\{\Dl \gm_t\neq 0\}}\big).
$$

It remains only to apply the condition (\ref{2-2.5}).
\end{proof}

In the next corollary we will need the following group of conditions:

For $\dl$, $0<\dl<\dl_0/2$,
\begin{gather}
    \bigg[ \dl\,\frac{g}{\gm} -\bt(z)\bigg]^+ \circ K_\infty^c<\infty
        \;\;(\Pas), \label{2-2.6} \\
    \sum_{t\geq 0} \Bigg[ (1-\bt_t(z_{t-}) \Dl K_t -
        \bigg(1-\frac{\Dl\gm_t}{\gm_t}\bigg)^\dl\Bigg]^+
        I_{\{\bt_t(z_{t-})\Dl K_t\leq 1\}}<\infty \;\;(\Pas), \label{2-2.7} \\
    \sum_{t\geq 0} \Bigg[ (\bt_t(z_{t-}) \Dl K_t -1 -
        \bigg(1-\frac{\Dl\gm_t}{\gm_t}\bigg)^\dl\Bigg]^+
        I_{\{\bt_t(z_{t-})\Dl K_t\geq 1\}}<\infty \;\;(\Pas). \label{2-2.8}
\end{gather}

\begin{corollary}\label{2-c2.2}
Suppose that the process
\begin{equation}\label{2-2.9}
    (\bt_t(z_{t-})\Dl K_t)_{t\geq 0} \;\;\; \text{is eventually bounded.}
\end{equation}

Then if Eq. $(\ref{2-2.5})$ holds,
\begin{enumerate}
\item[(1)] $\{(\ref{2-2.6})$, $(\ref{2-2.7})$, $(\ref{2-2.8})$ for all $\dl$, $0<\dl<\dl_0/2\}\Ra \{(\ref{2-2.3})$ for all $\dl$, $0<\dl<\dl_0\}$;

\item[(2)] if, in addition, the process $\xi=(\xi_t)_{t\geq 0}$, with $\xi_t=\sup\limits_{s\geq t} (\Dl \gm_s/\gm_s)$ is eventually $<1$, then the reverse implication ``$\Leftarrow$'' holds in $(1)$;

\item[(3)] $\{(\ref{2-2.6})$, $(\ref{2-2.7})$, $(\ref{2-2.8})$ for $\dl=\dl_0/2\}\Ra \{ (\ref{2-2.6})$, $(\ref{2-2.7})$, $(\ref{2-2.8})$ for all $\dl$, $0<\dl<\dl_0/2\}$ $($ here $\dl_0$ is some fixed constant with $0<\dl_0\leq 1)$.
\end{enumerate}
\end{corollary}

\begin{proof}
By the simple calculations, for all $\dl$, $0<\dl<\dl_0$, $0<\dl_0\leq 1$,
\allowdisplaybreaks
\begin{align}
    & \int_0^\infty \left(\frac{\gm_{t-}}{\gm_t}\right)^{-\dl}
        \Bigg[\left(\dl I_{\{\Dl \gm_t=0\}} +
            \frac{1-(1-\Dl\gm_t/\gm_t)^\dl}{\Dl\gm_t/\gm_t} \,
            I_{\{\Dl \gm_t\neq 0\}} \right)\frac{g_t}{\gm_t} \notag \\
    & \quad -2\bt_t(z_{t-})+\bt_t^2(z_{t-}) \Dl K_t\Bigg]^+ dK_t =
        \int_0^\infty \left[\dl\,\frac{g_t}{\gm_t}-2\bt_t(z_{t-})\right]^+dK_t^c
        \notag \\
    & \quad + \sum_{t\geq 0} \left(\frac{\gm_{t-}}{\gm_t}\right)^{-\dl}
        \left(1-\bt_t(z_{t-}) \Dl K_t-(1-\Dl \gm_t/\gm_t)^{\dl/2}\right) \notag \\
    & \quad \times \left[1-\bt_t(z_{t-}) \Dl K_t+(1-\Dl \gm_t/\gm_t)^{\dl/2}\right]^+
        I_{\{\bt_t(z_{t-})\Dl K_t\leq 1\}} \notag \\
    & \quad + \sum_{t\geq 0} \left(\frac{\gm_{t-}}{\gm_t}\right)^{-\dl}
        \left(\bt_t(z_{t-}) \Dl K_t-1+(1-\Dl \gm_t/\gm_t)^{\dl/2}\right) \notag \\
    & \quad \times \left[\bt_t(z_{t-}) \Dl K_t-1-(1-\Dl \gm_t/\gm_t)^{\dl/2}\right]^+
        I_{\{\bt_t(z_{t-})\Dl K_t\geq 1\}} . \label{2-4.1}
\end{align}

Now for the validity if implications (1) and (2) it is enough to show that under conditions  (\ref{2-2.5}) and (\ref{2-2.9}), the processes
$$
    \left(1-\bt(z_{-}) \Dl K+(1-\Dl \gm/\gm)^{\dl/2}\right)
        I_{\{\bt(z_-)\Dl K\leq 1\}}
$$
and
$$
    \left(\bt(z_{-}) \Dl K-1 +(1-\Dl \gm/\gm)^{\dl/2}\right)
        I_{\{\bt(z_-)\Dl K\geq 1\}}
$$
are eventually bounded and, moreover, if $\xi<1$ eventually, these processes are bounded from below by a strictly positive random constant. Indeed, for each $0<\dl<1$ and $t\geq 0$, if $\bt_t(z_{t-})\Dl K_t\leq 1,$
\begin{equation}\label{2-4.2}
    1-\sup_{s\geq t} \frac{\Dl \gm_s}{\gm_s} \leq 1-\bt_t(z_{t-})\Dl K_t +
        (1-\Dl \gm_t/\gm_t)^{\dl/2} \leq 2
\end{equation}
and,  if $\bt_t(z_{t-})\Dl K_t\geq 1$,
\begin{equation}\label{2-4.3}
    1-\sup_{s\geq t} \frac{\Dl \gm_s}{\gm_s} \leq \bt_t(z_{t-})\Dl K_t -1 +
        (1-\Dl \gm_t/\gm_t)^{\dl/2} \leq \bt_t(z_{t-}) \Dl K_t.
\end{equation}

The implication (3) simply follows from the inequality $(1-x)^\dl\leq (1-x)^{1/2}$ if $0<x<1$ and $0<\dl<1/2$.
\end{proof}

The following result is an immediate consequence of Corollary \ref{2-c2.2}.

\begin{corollary}\label{2-c2.3}
Suppose that
\begin{equation}\label{2-2.10}
    \sum_{t\geq 0} I_{\{\bt_t(z_{t-}) \Dl K_t\geq 1\}}<\infty \;\;\;
        \text{and} \;\;\; \sum_{t\geq 0} \left(\frac{\Dl\gm_t}{\gm_t}\right)^2
        <\infty \;\; (\Pas).
\end{equation}

Then Eq. $(\ref{2-2.7})$ is equivalent to
\begin{equation}\label{2-2.11}
    \int_0^\infty \bigg[\dl-\frac{\gm_t\bt_t(z_{t-})}{\gm_t}\bigg]^+
        \frac{d\gm_t^d}{\gm_t} <\infty \;\; (\Pas)
\end{equation}
and
$$
    \{(\ref{2-2.6}), \;(\ref{2-2.11}) \;\;\text{for all} \;\; \dl, \;
        0\leq \dl\leq \dl_0/2\} \Lra \{(\ref{2-2.3}) \;\;\text{for all} \;\; \dl, \;
        0<\dl<\dl_0\}.
$$
\end{corollary}

\begin{proof}
The conditions (\ref{2-2.8}) and (\ref{2-2.9}) are automatically satisfied and also $\xi<1$ eventually $(\xi=(\xi_t)_{t\geq 0}$ is the process with $\xi_t=\sup\limits_{s\geq t} (\Dl\gm_s/\gm_s)$). So it follows from Corollary \ref{2-c2.2} (2) that
$$
    \{(\ref{2-2.6}), (\ref{2-2.7}) \;\;\text{for all} \;\; \dl, \; 0<\dl<\dl_0/2\} \Ra
    \{(\ref{2-2.3}) \;\;\text{for all} \;\; \dl, \; 0<\dl<\dl_0\}.
$$

It remains to prove that Eq. (\ref{2-2.7}) is equivalent to Eq. (\ref{2-2.11}). This immediately follows from the inequalities
\begin{gather*}
    [a+b]^+\leq [a]^++[b]^+, \;\; \dl x \leq 1-(1-x)^\dl\leq
        \dl x+(1-\dl)x^2,\\
    0<x<1, \quad  0<\dl<1,
\end{gather*}
applying to the $x=(\Dl\gm_s/\gm_s)$ and to the expression
$$
    \left[1-\bt_t(z_{t-}) \Dl K_t +(1-\Dl \gm_t/\gm_t)^{\dl}\right]^+,
$$
and from the condition $\sum\limits_{t\geq 0} (\Dl\gm_t/\gm_t)^2<\infty$ $(\Pas)$.
\end{proof}

\begin{remark}\label{2-r.2.2}
The condition $(\ref{2-2.11})$ can be written as
$$
    \sum_{t\geq 0} \bigg[ \dl\,\frac{\Dl\gm_t}{\gm_t} -\bt_t(z_{t-}) \Dl K_t\bigg]^+
        <\infty \;\; (\Pas).
$$
\end{remark}

Below using the stochastic version of Kronecker Lemma, we give an alternative group of conditions to ensure the convergence
$$
    \gm_t^\dl z_t\to 0 \;\;\;\text{as} \;\;\; t\to \infty \;\;(\Pas)
$$
for all $0<\dl<\dl_0/2$, $0<\dl_0\leq 1$.

Rewrite Eq. (\ref{2-1.1}) in the following form
$$
    z_t=z_0+\int_0^t z_{s-} dB_s +G_t,
$$
where
$$
    dB_t=-\ol{\bt}_t(z_{t-}) dK_t, \;\;\;
    \ol{\bt}_t(u) =\bt_t(u) I_{\{\bt_t(u)\Dl K_t\neq 1\}}
$$
and
\begin{equation}\label{2-2.12}
    G_t=-\sum_{s\leq t} z_{s-} I_{\{\bt_t(z_{t-})\Dl K_t=1\}} +
        \int_0^t M(ds,z_{s-}).
\end{equation}

Since $\Dl B_t=-\ol{\bt}_t(z_{t-}) \Dl K_t \neq -1$ we can represent $z$ as
$$
    z_t=\ve_t(B)\bigg( z_0+\int_0^t \ve_s^{-1}(B) \,dG_s\bigg)
$$
and multiplying this equation by $\gm_t^\dl$ yields
\begin{equation}\label{2-2.13}
    \gm_t^\dl z_t =\sign \ve_t(B) \Gm_t^{(\dl)}
        \bigg( z_0+\int_0^t \sign \ve_s(B) \{\Gm_s^{(\dl)}\}^{-1}
            \gm_s^\dl \,dG_s\bigg),
\end{equation}
where $\Gm_t^{(\dl)}=\gm_t^\dl |\ve_t(B)|$.

\begin{definition}
We say that predictable processes $\xi=(\xi_t)_{t\geq 0}$ and $\eta=(\eta_t)_{t\geq 0}$ are equivalent as $t\to \infty$ and write $\xi\simeq \eta$ if there exists a process $\zt=(\zt_t)_{t\geq 0}$ such that
$$
    \xi_t=\zt_t\eta_t,
$$
and
$$
    0<\zt^1<|\zt|<\zt^2<\infty
$$
eventually, for some random constants $\zt^1$ and $\zt^2$.
\end{definition}

The proof of the following result is based on the stochastic version of the Kronecker Lemma.

\begin{proposition}\label{2-p2.2}
Suppose that for all $\dl$, $0<\dl<\dl_0/2$, $0<\dl_0\leq 1$,

\begin{enumerate}
\item[(1)] there exists a positive and decreasing predictable process $\ol{\Gm}{}^{(\dl)}=(\ol{\Gm}{}_t^{(\dl)})_{t\geq 0}$ such that
$$
    \ol{\Gm}{}_0^{(\dl)}=1\;\;(\Pas), \quad
    P\big\{ \lim_{t\to 0} \ol{\Gm}{}_t^{(\dl)}=0\big\}=1, \;\;\;
    \Gm^{(\dl)} \simeq \ol{\Gm}{}^{(\dl)}
$$
and

\item[(2)]
\begin{gather}
    \sum_{t\geq 0} I_{\{\bt_t(z_{t-})\Dl K_t=1\}} <\infty \;\; (\Pas),\label{2-2.14}\\
    \int_0^\infty \gm_t^{2\dl} h_t(z_{t-},z_{t-}) dK_t
        <\infty \;\; (\Pas).\label{2-2.15}
\end{gather}
\end{enumerate}

Then
$$
    \gm_t^\dl z_t\to 0 \;\;\;\text{as} \;\;\; t\to \infty \;\;(\Pas)
$$
for all $0<\dl<\dl_0/2$, $0<\dl_0\leq 1$.
\end{proposition}

\begin{proof}
Recall the stochastic version of Kronecker Lemma (see, e.g., \cite{proc12}, Ch.~2, Section 6):

{\bf Kronecker Lemma.}
{\it Suppose that $X=(X_t)_{t\geq 0}$ is s semimartingale and $L=(L_t)_{t\geq 0}$ is s predictable increasing process. Then
$$
    \{L_\infty=\infty\} \cap \{Y\to\} \sbseq \left\{ \frac{X}{L}\to 0\right\}
        \;\;(\Pas),
$$
where $Y=(1+L)^{-1} \cdot X$.}

Put $(1+L_t)^{-1}=\ol{\Gm}{}_t^{(\dl)}$ and $X_t=\int\limits_0^t (\Gm_s^{(\dl)})^{-1} \sign \ve_s(B)\gm_s^\dl dG_s$. Then it follows from the condition (1) that $L$ is an increasing process with $L_\infty=\infty$ $(\Pas)$ and
\begin{align*}
    A & =\{\ol{\Gm}{}_\infty^{(\dl)}=0\} \cap
        \bigg\{ \int_0^{\cdot} \ol{\Gm}{}_s^{(\dl)}(\Gm_s^{(\dl)})^{-1}
            \sign \ve_s(B) \gm_s^\dl d G_s\to \bigg\} \\
    & \sbseq \bigg\{ \frac{\ol{\Gm}{}^{(\dl)}}{1-\ol{\Gm}{}^{(\dl)}}
         \int_0^{\cdot} (\Gm_s^{(\dl)})^{-1}
            \sign \ve_s(B) \gm_s^\dl d G_s\to 0 \bigg\}
        \sbseq \{\gm^\dl z\to 0\},
\end{align*}
where the latter inequality follows from the relation $\ol{\Gm}{}^{(\dl)}\simeq \Gm^{(\dl)}$ and Eq. (\ref{2-2.13}).

At the same time, from Eq. (\ref{2-2.12}) and from the well-known fact that if $M\in \cM_{\loc}^2$, then $\{\la M\ra_\infty<\infty\} \sbseq \{ M\to \}$ (see, e.g., \cite{proc12}), we have
$$
    \{\ol{\Gm}{}_\infty^{(\dl)}=0\} \cap
        \bigg\{ \sum_{t\geq 0} I_{\{\bt_t(z_{t-})\Dl K_t=1\}}<\infty\bigg\} \cap
        \bigg\{ \int_0^\infty\gm_t^{2\dl} h_t(z_{t-},z_{t-}) d K_t<\infty \bigg\}
        \sbseq A.
$$

The result now follows from Eqs. (\ref{2-2.14}) and (\ref{2-2.15}).
\end{proof}

Now we establish some simple results which are useful for verifying the condition (1) of Proposition \ref{2-p2.2}.

By the definition of $\ve_t(B)$,
$$
    \ve_t(B)=e^{B_t^c} \prod_{s\leq t} (1+\Dl B_s)
$$
and since
$$
    \gm_t^\dl =\exp \Bigg( \dl \int_0^t \frac{d\gm_s^c}{\gm_s} -
        \sum_{s\leq t} \log \left( 1-\frac{\Dl\gm_s}{\gm_s}\right)^\dl\Bigg)
$$
we obtain
\begin{align}
    \Gm_t^{(\dl)} & =\exp \Bigg( B_t^c+\dl\int_0^t \frac{d\gm_s^c}{\gm_s} +
        \sum_{s\leq t} \log \frac{|1+\Dl B_s|}{\big(1-\frac{\Dl\gm_s}{\gm_s}\big)^\dl}
            \Bigg) \notag \\
    & =\exp \bigg( -\int_0^t D_s dC_s^{(\dl)}\bigg), \label{2-2.16}
\end{align}
where $D_t=1/\gm_t$ and
\begin{align}
    C_t^{(\dl)} & =\int_0^t \Bigg( \left\{ \frac{\bt_s(z_{s-})\gm_s}{g_s} -\dl\right\}
        I_{\{\Dl \gm_s=0\}} \notag \\
    & \quad  -\frac{\gm_s}{\Dl \gm_s} \,
        \log \frac{|1+\Dl B_s|}{\big(1-\frac{\Dl \gm_s}{\gm_s}\big)^\dl} \,
        I_{\{\Dl\gm_s\neq 0\}}\Bigg) d\gm_s. \label{2-2.17}
\end{align}

Using the formula of integration by parts
$$
    d(D_t C_t)=D_td C_t+C_{t-} dD_t
$$
and the relation
$$
    d\left(\frac{1}{\gm_t}\right) =-\frac{1}{\gm_{t-}}\,\frac{d\gm_t}{\gm_t}
$$
we get from Eq. (\ref{2-2.16}) that
$$
    \Gm_t^{(\dl)} =\exp \bigg( -\frac{C_t^{(\dl)}}{\gm_t}-
        \int_0^t C_{s-}^{(\dl)}\,\frac{1}{\gm_{s-}}\,\frac{d\gm_s}{\gm_s}\bigg).
$$

Therefore,
\begin{equation}\label{2-2.18}
    \Gm_t^{(\dl)}=\zt_t \ol{\Gm}{}_t^{(\dl)},
\end{equation}
where
\begin{align*}
    \ol{\Gm}{}_t^{(\dl)} & =\exp \bigg( -\int_0^t
        \bigg[ \frac{C_{s-}^{(\dl)}}{\gm_{s-}}\bigg]^+ \frac{d\gm_s}{\gm_s}\bigg), \quad
    \zt_t  =\exp \bigg( -\frac{C_t^{(\dl)}}{\gm_t} +\int_0^t
        \bigg[ \frac{C_{s-}^{(\dl)}}{\gm_{s-}}\bigg]^+ \frac{d\gm_s}{\gm_s}\bigg).
\end{align*}

The following proposition is an immediate consequence of Eq. (\ref{2-2.18}).

\begin{proposition}\label{2-p2.3}
Suppose that for each $\dl$, $0<\dl<\dl_0/2$, $0<\dl_0\leq 1$, the following conditions hold:
\begin{enumerate}
\item[(a)] There exist random constants $\ul{C}(\dl)$ and $\ol{C}(\dl)$ such that
$$
    -\infty <\ul{C}(\dl) <\frac{C^{(\dl)}}{\gm} <\ol{C}(\dl)<\infty
$$
eventually, where $C^{(\dl)}/\gm=(C_t^{(\dl)}/\gm_t)_{t\geq 0}$.
\smallskip

\item[(b)] $\ds \qquad \qquad \int_0^\infty \bigg[ \frac{C_{t-}^{(\dl)}}{\gm_{t-}}\bigg]^-
    \frac{d\gm_t}{\gm_t}<\infty \;\; (\Pas).$
\smallskip

\item[(c)] $\ds \qquad \qquad \int_0^\infty \bigg[ \frac{C_{t-}^{(\dl)}}{\gm_{t-}}\bigg]^+
    \frac{d\gm_t}{\gm_t}=\infty \;\; (\Pas).$
\end{enumerate}
\smallskip

    Then $\Gm^{(\dl)}\simeq \ol{\Gm}{}^{(\dl)}$ for each $\dl$, $0<\dl<\dl_0/2$.
\end{proposition}

\begin{corollary}\label{2-c2.4}
Suppose that
$$
    0<\frac{C^{(\dl_0/2)}}{\gm} <\frac{C^{(0)}}{\gm} <\ol{C}(0)<\infty
$$
eventually, where $\ol{C}(0)$ is some random constant and the processes $C^{(\dl_0/2)}$ and $C^{(0)}$ are defined in Eq. $(\ref{2-2.17})$ for $\dl=\dl_0/2$ and $\dl=0$, respectively.

Then $\Gm^{(\dl)}\simeq \ol{\Gm}{}^{(\dl)}$ for each $\dl$, $0<\dl<\dl_0/2$, $0<\dl_0\leq 1$.
\end{corollary}

This result follows since, as it is easy to check,
$$
    C_t^{(\dl_0/2)} <C_t^{(\dl)}<C_t^{(0)} \;\;\;\text{and} \;\;\;
        C_t^{(\dl)}-C_t^{(\dl_0/2)} \geq \left(\frac{\dl_0}{2}-\dl\right)\gm_t
$$
for each $\dl$, $0<\dl<\dl_0/2$, which gives
$$
    \frac{\dl_0}{2}-\dl <\frac{C^{(\dl)}}{\gm} <\ol{C}(0)
$$
and
$$
    \bigg[ \frac{C^{(\dl)}}{\gm}\bigg]^+>\frac{\dl_0}{2}-\dl \;\;\;\text{and} \;\;\;
    \bigg[ \frac{C^{(\dl)}}{\gm}\bigg]^-=0
$$
eventually.

We shall now formulate the main result of this approach which is an immediate consequence of Propositions \ref{2-p2.2} and \ref{2-p2.3}.

\begin{theorem}\label{2-t2.2}
Suppose that the conditions $(\ref{2-2.14})$, $(\ref{2-2.15})$ and the conditions of Proposition $\ref{2-p2.3}$ hold for all $\dl$, $0<\dl<\dl_0/2$, $0<\dl_0\leq 1$. Then $\Pas$,
$$
    \gm_t^\dl z_t \to 0 \;\;\; \text{as} \;\;\; t\to \infty
$$
for all $\dl$, $0<\dl<\dl_0/2$, $0<\dl_0\leq 1$.
\end{theorem}

Consider in more detail two cases: (1) all the processes under the consideration are continuous; (2)~the discrete time case. In addition  assume that $M(t,u)=M(t)$ for all $u\in R^1$, $t\geq 0$.

In the case of continuous processes conditions (\ref{2-2.7}) and (\ref{2-2.8}) are satisfied trivially, the condition (\ref{2-2.6}) takes the form
\begin{equation}\label{2-2.19}
    \int_0^\infty \bigg[ \dl-\frac{\gm_t \bt_t(z_{t-})}{g_t} \bigg]^+
        \frac{d\gm_t}{\gm_t}<\infty \;\;(\Pas)
\end{equation}
and also
$$
    \{(\ref{2-2.19}) \;\;\text{for} \;\; \dl=\dl_0/2\}\Ra
        \{(\ref{2-2.19})\;\;\text{for all} \;\; \dl, \; 0<\dl<\dl_0/2\}.
$$

Further, since
$$
    \frac{C_t^{(\dl)}}{\gm_t} =\frac{1}{\gm_t} \int_0^t
        \frac{\bt_s(z_s)\gm_s}{g_s}\,d\gm_s-\dl\geq -\dl,
$$
the conditions (a)--(c) of Proposition \ref{2-p2.3} can be simplified to:

\begin{enumerate}
\item[(a$'$)] The process
$$
    \bigg(\frac{1}{\gm_t} \int_0^t \frac{\bt_s(z_s)\gm_s}{g_s}\,d\gm_s\bigg)_{t\geq 0}
$$
is eventually bounded.
\smallskip

\item[(b$'$)] $\ds \qquad
\int_0^\infty \bigg[ \frac{1}{\gm_t} \int_0^t \frac{\bt_s(z_s)\gm_s}{g_s}\,
    d\gm_s-\dl\bigg]^- \frac{d\gm_t}{\gm_t} <\infty \;\;(\Pas).$
\smallskip

\item[(c$'$)] $\ds \qquad
\int_0^\infty \bigg[ \frac{1}{\gm_t} \int_0^t \frac{\bt_s(z_s)\gm_s}{g_s}\,
    d\gm_s-\dl\bigg]^+ \frac{d\gm_t}{\gm_t} =\infty \;\;(\Pas). $ \\
\smallskip

Also, if (a$'$) holds and
\smallskip

\item[(bc$'$)] $\ds \qquad
\frac{C^{(\dl_0/2)}}{\gm} =\Bigg(\frac{1}{\gm_t} \int_0^t
    \frac{\bt_s(z_s)\gm_s}{g_s}\,d\gm_s -\frac{\dl_0}{2} \Bigg) _{t\geq 0} >0,$
    eventually, \\[2mm]
then (b$'$) and (c$'$) hold for each $\dl$, $0<\dl<\dl_0/2$.
\end{enumerate}

In the discrete time case we assume additionally that
\begin{equation}\label{2-2.20}
    \sum_{t\geq 0} \left( \frac{\Dl\gm_t}{\gm_t}\right)^2<\infty \;\;\;
        \text{and} \;\;\;
    \sum_{t\geq 0} (\bt_t(z_{t-1}))^2<\infty \;\;(\Pas).
\end{equation}

Then the conditions of Corollary \ref{2-c2.3} are trivially satisfied. Hence, the conditions (\ref{2-2.3}) and (\ref{2-2.11}) are equivalent and can be written as
\begin{equation}\label{2-2.21}
    \sum_{t\geq 0}  \left[ \dl-\frac{\gm_t\bt_t(z_{t-1})}{g_t}\right]^+
        \frac{\Dl\gm_t}{\gm_t}<\infty \;\;(\Pas)
\end{equation}
and also,
$$
    \left\{(\ref{2-2.21}) \;\;\text{for} \;\; \dl=\dl_0/2\right\} \Ra
        \left\{(\ref{2-2.21}) \;\;\text{for all} \;\; \dl, \;\;0<\dl<\dl_0/2\right\}.
$$

Note that the reverse implication ``$\Leftarrow$'' does not hold in general (see Example~\ref{2-e3}).

It is not difficult to verify that (a), (b) and (c) are equivalent to $(\wt{\rm a})$, $(\wt{\rm b})$ and $(\wt{\rm c})$ defined as follows.

\begin{enumerate}
\item[$(\wt{\rm a})$] The process
$$
    \bigg( \frac{1}{\gm_t} \sum_{s\leq t} \bt_s(z_{s-1})\gm_s\bigg)_{t\geq 0}
$$
is bounded eventually.
\smallskip

\item[$(\wt{\rm b})$] $\ds \quad
\sum_{t\geq 1} \bigg[ \frac{1}{\gm_{t-1}} \sum_{s<t} \bt_s(z_{s-1})\gm_s-\dl\bigg]^-
    \frac{\Dl\gm_t}{\gm_t}<\infty \;\; (\Pas).$
\smallskip

\item[$(\wt{\rm c})$] $\ds \quad
\sum_{t\geq 1} \bigg[ \frac{1}{\gm_{t-1}} \sum_{s<t} \bt_s(z_{s-1})\gm_s-\dl\bigg]^+
    \frac{\Dl\gm_t}{\gm_t}=\infty \;\; (\Pas).$ \\[2mm]
Also if $(\wt{\rm a})$ holds and
\smallskip

\item[$(\wt{\rm b}{\rm c})$] $\ds \quad
\bigg( \frac{1}{\gm_t} \sum_{s\leq t} \bt_s(z_{s-1})\gm_s-\dl\bigg)_{t\geq 0}>
    \dl_0/2\;$ eventually, \\[2mm]
then $(\wt{\rm b})$ and $(\wt{\rm c})$ hold for each $\dl$, $0<\dl<\dl_0/2$.
\end{enumerate}

Hence $\{(\wt{\rm a}),(\wt{\rm b}{\rm c})\}\Ra \{(\wt{\rm a}),(\wt{\rm b}),(\wt{\rm c})$ for all $\dl$, $0<\dl<\dl_0/2\}$. However, the inverse implication is not true (see Examples \ref{2-e3} and \ref{2-e4}).

Note that the conditions imposed on the martingale part of Eq. (\ref{2-1.1}) in Theorems \ref{2-t2.1} (see Eq. (\ref{2-2.4})) and \ref{2-t2.2} (see Eq. (\ref{2-2.15})) are identical. We, therefore, assume that these conditions hold in all examples given below.

\setcounter{example}{0}
\begin{example}\label{2-e1}
This example illustrates that Eq. (\ref{2-2.19}) holds whereas (a$'$) is violated.

Let
$$
    K_t=\gm_t=t+1 \;\;\;\text{and} \;\;\; \bt_t(u)\equiv (t+1)^{-(1/2+\al)},
$$
where $0<\al<1/2$.

Substituting $K_t$, $\gm_t$, $\bt_t$ in the left-hand side of Eq. (\ref{2-2.19}) we get
$$
    \int_0^\infty \big[ \dl-(t+1)^{-(1/2-\al)}(t+1)\big]^+ \frac{dt}{t+1} =
    \int_0^\infty \big[ \dl-(t+1)^{1/2-\al}\big]^+ \frac{dt}{t+1}\,.
$$
Since $([\dl-(t+1)^{1/2-\al}]^+)_{t\geq 0} =0$ eventually, the condition (\ref{2-2.19}) holds.

The conditions (a$'$) does not hold since
$$
    \frac{1}{\gm_t} \int_0^t \frac{\bt_s(z_s)\gm_s}{g_s}\,d\gm_s =\frac{1}{t+1}
        \int_0^t  (s+1)^{1/2-\al}ds\varpropto (t+1)^{1/2-\al}\to \infty
            \;\;\text{as} \;\;t\to \infty.
$$
Note that the conditions (b$'$) and (c$'$) are satisfied.

It should be pointed out that although Eq. (\ref{2-2.19}) holds for all $\dl$, $\dl>0$, if, e.g.,
$$
    d\la M\ra_t =\frac{dt}{(t+1)^{3/2+\al}}\,,
$$
the conditions (\ref{2-2.4}) only holds for $\dl$'s satisfying $0<\dl<\dl_0=1/2+\al$.
\end{example}

\begin{example}\label{2-e2}
In this example the conditions $(\wt{\rm a})$ and $(\wt{\rm b}{\rm c})$ hold for $\dl_0=1$ while Eq. (\ref{2-2.21}) fails for some $\dl$, $0<\dl<1/2=\dl_0/2$.

Consider a discrete time model with $K_t=\gm_t=t$, $\bt_t(u)\equiv \bt_t$ and
$$
    \bt_t\gm_t =\begin{cases}
                    1/2+a & \text{if} \;\; t \;\;\text{is odd}, \\
                    1/2-b & \text{otherwise},
                \end{cases}
$$
where $0<b<1/2\leq a$. Then, since
$$
    \frac{1}{2}+a>\frac{1}{\gm_t} \sum_{s\leq t} \bt_s\gm_s=\frac{1}{2}+
    \begin{cases}
        \frac{a-b}{2} & \text{if} \;\; t=2k, \;\; k=1,2,\dots \\
        \frac{k(a-b)+a}{2k+1} & \text{if} \;\; t=2k+1, \;\; k=1,2,\dots
    \end{cases} >\frac{1}{2}\,,
$$
the conditions $(\wt{\rm a})$ and $(\wt{\rm b}{\rm c})$ hold for $\dl_0=1$.

It is easy to verify that if $1/2-b<\dl<1/2$, then
$$
    \sum_{t\geq 1} [\dl-\bt_t\gm_t]^+\frac{1}{t} =\sum_{t\geq 1}
        \left[\dl-\frac{1}{2}+b\right]^+\frac{1}{t}\, I_{\{t \; \text{is even}\}}
        =\infty
$$
implying that Eq. (\ref{2-2.21}) does not hold for all $\dl$ with $1/2<b<\dl<1/2$.
\end{example}

\begin{example}\label{2-e3}
In this discrete time example $\dl_0=1$ and
$$
    \{(\ref{2-2.21}) \;\;\text{for all} \;\; \dl, \; 0<\dl<1/2\} \not\Ra
        \{(\ref{2-2.21}) \;\;\text{for} \;\; \dl=1/2\}.
$$

Suppose that $K_t=\gm_t=t$, $\bt_t(u)\equiv \bt_t$ and
$$
    \bt_t\gm_t=\left[\frac{1}{2} -\frac{1}{\log(t+1)}\right]^+.
$$

Then for $0<\dl<1/2$ and large $t$'s,
$$
    [\dl-\bt_t\gm_t]^+=0
$$
and it follows that
$$
    \sum_{t\geq 1} [\dl-\bt_t\gm_t]^+\frac{1}{t}<\infty.
$$

But for $\dl=1/2$,
$$
    \sum_{t\geq 1} \left[\frac{1}{2}-\bt_t\gm_t\right]^+\frac{1}{t} \geq
    \sum_{t\geq 1} \frac{1}{t\log(t+1)}\,I_{\{\log(t+1)>1\}} =\infty.
$$

Note also that by the Toeplitz Lemma,
$$
    \frac{1}{t} \sum_{s\leq t} \bt_s\gm_s =\frac{1}{t} \sum_{s\leq t}
        \left[\frac{1}{2} -\frac{1}{\log(s+1)}\right]^+ \uparrow \frac{1}{2}
            \;\;\;\text{as} \;\;\; t\to \infty.
$$

Therefore, for all $\dl$, $0<\dl<1/2$, the conditions $(\wt{\rm a})$, $(\wt{\rm b})$ and $(\wt{\rm c})$ hold whereas $(\wt{\rm b}{\rm c})$ does not.
\end{example}

\begin{example}\label{2-e4}
This is a discrete time example illustrating that Eq. (\ref{2-2.21}) holds for $\dl=1/2$ (hence for all $0<\dl<1/2$) and for all $\dl$, $0<\dl<1/2$, the conditions $(\wt{\rm a})$, $(\wt{\rm b})$ and $(\wt{\rm c})$ hold whereas $(\wt{\rm b}{\rm c})$ does not.

Suppose that $K_t=\gm_t=t$, $\bt_t(u)\equiv \bt_t$ and for $t>0$,
$$
    \bt_t\gm_t=\frac{1}{2}-\frac{1}{t}\,.
$$

Then for $\dl=1/2$ the condition (\ref{2-2.21}) follows since
$$
    \sum_{t> 2} \left[\frac{1}{2} -\bt_t\gm_t\right]^+ \frac{1}{t}=
    \sum_{t>2} \frac{1}{t^2} <\infty.
$$

It remains to note that
$$
    \frac{1}{t} \sum_{s\leq t} \bt_s\gm_s \uparrow \frac{1}{2}
$$
by the Toeplitz Lemma.
\end{example}

\begin{example}\label{2-e5}
Here we drop the ``traditional'' assumptions
$$
    \sum_{t>0} \left(\frac{\Dl \gm_t}{\gm_t}\right)^2 <\infty \;\;\;\text{and} \;\;\;
        \sum_{t\geq 0} (\bt_t(z_{t-1}))^2<\infty \;\;(\Pas)
$$
and give an example when the conditions of Theorems \ref{2-t2.1} and \ref{2-t2.2} are satisfied.

Suppose that $K_t=t$ and process $\gm$ and $\bt(u)=\bt$ are defined as follows: $\gm_1=1$,
$$
    \gm_t=\sum_{s=1}^t q^s =\frac{q}{1-q}\,(1-q^t), \;\;\;\text{where} \;\;
        q>1,
$$
and
$$
    \bt_t=\frac{\al}{\bt}\,\frac{\Dl\gm_t}{\gm_t}\,,
$$
where $\al=q/(q-1)$ and $\bt$, $\bt>1$, are some constants satisfying $(1-1/\al)^{1/2}>1-1/\bt$. In this case,
$$
    \frac{\Dl\gm_t}{\gm_t}\to \frac{1}{\al} \;\;\;\text{as} \;\;\; t\to \infty
$$
and
$$
    \bt_t\Dl K_t=\frac{\al}{\bt} \,\frac{\Dl\gm_t}{\gm_t} \to \frac{1}{\bt} <1 \;\;\;
        \text{as} \;\;\; t\to \infty.
$$

Therefore the conditions of Corollary \ref{2-c2.3} hold and it follows that the conditions (\ref{2-2.3}) and (\ref{2-2.11}) are equivalent.

To check Eq. (\ref{2-2.11}) note that for all $0<\dl<1/2$,
\begin{align*}
    & \sum_{t>0} \Bigg[ 1-\bt_t(z_{t-}) \Dl K_t -
        \bigg( 1-\frac{\Dl\gm_t}{\gm_t} \bigg)^\dl\Bigg]^+
        I_{\{\bt_t(z_{t-})\Dl K_t\leq 1\}} \\
    & \quad \leq \sum_{t>0} \Bigg[ 1-\bt_t(z_{t-}) \Dl K_t -
        \bigg( 1-\frac{\Dl\gm_t}{\gm_t} \bigg)^{1/2}\Bigg]^+
        I_{\{\bt_t(z_{t-})\Dl K_t\leq 1\}} \\
    & \quad \leq \sum_{t>0} \Bigg[ 1-\frac{1}{\bt}\,\frac{q^t}{q^t-1} -
        \bigg( 1-\frac{1}{\al}\,\frac{q^t}{q^t-1} \bigg)^{1/2}\Bigg]^+
        I_{\{\bt_t(z_{t-})\Dl K_t\leq 1\}} .
\end{align*}

But since
$$
    1-\frac{1}{\bt}\,\frac{q^t}{q^t-1} -
        \bigg( 1-\frac{1}{\al}\,\frac{q^t}{q^t-1} \bigg)^{1/2} \to
        1-\frac{1}{\bt} -\left(1-\frac{1}{\al}\right)^{1/2} <0,
$$
we have
$$
    \Bigg[ 1-\frac{1}{\bt}\,\frac{q^t}{q^t-1} -
        \bigg( 1-\frac{1}{\al}\,\frac{q^t}{q^t-1} \bigg)^{1/2}\Bigg]^+=0
$$
for large $t$'s. Hence Eq. (\ref{2-2.11}) holds.

To  check the conditions (a), (b) and (c) of Theorem \ref{2-t2.2} note that by the Toeplitz Lemma,
$$
    \frac{1}{\gm_t} \,C_t^{(\dl)} =-\frac{1}{\gm_t} \sum_{s\leq t} \Dl \gm_s
        \log\frac{|1-\bt_s(z_{s-1})|}{\big(1-\frac{\Dl\gm_s}{\gm_s}\big)^\dl}\,
        \frac{\gm_s}{\Dl \gm_s} \to a,
$$
where
$$
    a=-\al\log \frac{1-1/\bt}{(1-1/\al)^\dl} >-\al\log
        \frac{1-1/\bt}{(1-1/\al)^{1/2}}>0,
$$
which implies (a), (b) and (c).
\end{example}

\subsection{Asymptotic expansion}\label{2-s3}

In subsection \ref{2-s1} we have derived  the representation
\begin{equation}\label{2-3.1}
    \chi_tz_t=\frac{L_t}{\la L\ra_t^{1/2}} +R_t,
\end{equation}
where all objects are defined there.

Throughout this subsection we assume that
$$
    \la L\ra_\infty =\infty \;\;(\Pas)
$$
and there exists a predictable increasing process $\gm=(\gm_t)_{t\geq 0}$ such that $\gm_0=1$, $\gm_\infty=\infty$ $(\Pas)$, the process $\gm/\gm_-$ is eventually bounded and
$$
    \gm\simeq \Gm^2\la L\ra^{-1}.
$$

In this subsection, assuming that $\gm_t^\dl z_t\to 0$ $\Pas$ for all $0<\dl<\dl_0/2$ (for some $0<\dl_0\leq 1$), we establish  sufficient conditions for the convergence $R_t\str{P}{\to} 0$ as $t\to \infty$.

Consider the following conditions:

\begin{enumerate}
\item[(d)] There exists a non-random increasing process $(\la\la L\ra\ra_t)_{t\geq 0}$ such that
$$
    \frac{\la L\ra_t}{\la\la L\ra\ra_t} \str{d}{\to} \zt \quad \text{as} \quad t\to \infty,
$$
where $\str{d}{\to}$ denotes the convergence in distribution and $\zt>0$ is some random variable.
\smallskip

\item[(e)] $\ds \qquad
\sum_{t\geq 0} I_{\{\bt_t \Dl K_t=1\}}<\infty \;\; (\Pas).$
\smallskip

\item[(f)] There exists $\ve$, $1/2-\dl_0/2<\ve<1/2$, such that
$$
    \frac{1}{\la L\ra_t} \int_0^t |\bt_s-\bt_s(z_{s-})| \gm_{s-}^\ve
        \la L\ra_s dK_s\to 0 \;\;\text{as} \;\; t\to \infty \;\;(\Pas).
$$
\smallskip

\item[(g)] $\ds
\frac{1}{\la L\ra_t} \int_0^t \Gm_s^2(h_s(z_{s-},z_{s-})-2h_s(z_{s-},0)+
    h_s(0,0)) dK_s \str{P}{\to} 0 \;\;\text{as} \;\; t\to \infty.$
\end{enumerate}

\begin{theorem}\label{2-t3.1}
Suppose that $\gm_t^\dl z_t\to 0$ $\Pas$ for all $\dl$, $0<\dl<\dl_0/2$ $(0<\dl_0\leq 1)$, and the conditions {\rm (d)--(g)} are satisfied. Then
$$
    R_t\str{P}{\to} 0 \;\;\;\text{as} \;\;\; t\to \infty.
$$
\end{theorem}

\begin{proof}
Recall from subsection \ref{2-s1} that
$$
    R_t=\frac{1}{\la L\ra_t^{1/2}} \,z_0+R_t^{(1)}+R_t^{(2)}+R_t^{(3)},
$$
where
\begin{align*}
    R_t^{(1)} & =-\frac{1}{\la L\ra_t^{1/2}}
        \sum_{s\leq t} \Gm_sz_{s-} I_{\{\bt_s \Dl K_s=1\}}, \\
    R_t^{(2)} & =\frac{1}{\la L\ra_t^{1/2}}
        \int_0^t \Gm_s(\bt_s-\bt_s(z_{s-}))z_{s-} dK_s, \\
    R_t^{(3)} & =\frac{1}{\la L\ra_t^{1/2}}
        \int_0^t \Gm_s(M(ds,z_{s-})-M(ds,0)).
\end{align*}

Since $\la L\ra_t \to \infty$, we have $z_0/\la L\ra_t^{1/2} \to 0$ as $t\to \infty$. Further, it follows from (e) that the process $(I_{\{\bt_t \Dl K_t=1\}})_{t\geq 0}=0$ eventually and therefore $R_t^{(1)}\to 0$ as $t\to \infty$.

Since the process $\gm/\gm_-$ is bounded eventually and $\gm_t^{1/2-\ve}z_t\to 0$ as $t\to \infty$ $(\Pas)$, we obtain that the process $\gm^{1/2-\ve}z_-$ is bounded eventually for each $\ve$, $1/2-\dl_0/2<\ve<1/2$. Also, $|\Gm|\la L\ra^{-1/2}\simeq \gm^{1/2}$. It therefore follows that there exists an eventually bounded positive process $\eta=(\eta_t)_{t\geq 0}$ such that
\begin{align*}
    |R_t^{(2)}| & \leq \frac{1}{\la L\ra_t^{1/2}} \int_0^t  |\Gm_s|\,
        |\bt_s-\bt_s(z_{s-})|\,|z_{s-}| \,dK_s \\
    & =\frac{1}{\la L\ra_t^{1/2}} \int_0^t |\bt_s-\bt_s(z_{s-})|
        \gm_s^\ve \la L\ra_s\eta_s \,
        \frac{dK_s}{\la L\ra_s^{1/2}} =
        \frac{1}{\la L\ra_t^{1/2}}\int_0^t D_s d C_s^\ve,
\end{align*}
where
$$
    D_t=\frac{1}{\la L\ra_t^{1/2}}, \quad
    C_t^\ve =\int_0^t |\bt_s-\bt_s(z_{s-})|\gm_s^\ve\la L\ra_s\eta_s dK_s.
$$

Using the formulae $d(D_tC_t)=D_t dC_t+C_{t-}d D_t$ we obtain
$$
    |R_t^{(2)}| \leq \Bigg( \frac{1}{\la L\ra_t}\,C_t^\ve -
        \frac{1}{\la L\ra_t^{1/2}} \int_0^t C_{s-}^\ve d\la L\ra_s^{-1/2} \Bigg).
$$

It is easy to check that
$$
    d(\la L\ra_t^{-1/2})=-\frac{1}{\la L\ra_{t-}^{1/2}}\,
        \frac{d\la L\ra_t^{1/2}}{\la L\ra_t^{1/2}}
$$
and
$$
    |R_t^{(2)}| \leq \frac{1}{\la L\ra_t}\,C_t^\ve +
        \frac{1}{\la L\ra_t^{1/2}} \int_0^t \frac{1}{\la L\ra_{s-}}
            C_{s-}^\ve d\la L\ra_s^{1/2}.
$$

Now, from the condition (f) and the Toeplitz Lemma, $R_t^{(2)}\to 0$, $\Pas$

To prove the convergence $R_t^{(3)}\to 0$ note that by the condition (d), it suffices only to consider the case when $\la L\ra_t$ is non-random. Denote
$$
    N_t=\int_0^t \Gm_s(M(ds,z_{s-})-M(ds,0)).
$$

Using the Lenglart--Rebolledo inequality (see, e.g., \cite{proc12}, Ch. 1, Section 9, \cite{22n}) we obtain
\begin{align*}
    P\{\la L\ra_t^{-1/2}N_t>a\} & =P\{\la L\ra_t^{-1} N_t^2>a^2\}=
        P\{N_t^2-\la L\ra_t\ve >(a^2-\ve)\la L\ra_t\} \\
    & \leq \frac{b}{(a^2-\ve)\la L\ra_t} +P\{\la N\ra_t-\la L\ra_t\ve>b\}
\end{align*}
for any $a>0$, $b>0$ and $0<\ve<a^2$. The result now follows since $\la L\ra_\infty=\infty$ $\Pas$ and
\begin{gather*}
    \frac{1}{\la L\ra_t} \,\la N\ra_t =\frac{1}{\la L\ra_t}
        \int_0^t \Gm_s^2(h_s(z_{s-},z_{s-})-2h_s(z_{s-},0)+
            h_s(0,0)) dK_s \str{P}{\to} 0 \\
    \text{as} \;\; t\to \infty. \qedhere
\end{gather*}
\end{proof}

\begin{remark}\label{2-r3.1}
Suppose that $\Pas$,
$$
    \bt\circ K_\infty =\infty, \quad
        \inf_{t\geq 0} \bt_t I_{\{\Dl K_t\neq 0\}}>0, \quad
        \sup_{t\geq 0} \bt_t \Dl K_t I_{\{\Dl K_t\neq 0\}}<2.
$$
Then, as it is easy to see, $|\Gm|$ is an increasing process with $|\Gm_\infty|=\infty$ $(\Pas)$.
\end{remark}

\begin{remark}\label{2-r3.2}
\

\begin{enumerate}
\item[1.] The condition (f) can be replaced by the following one:
(f$'$) there exists $\ve>(1-\dl_0)/\dl_0$ such that
$$
    \frac{1}{\la L\ra_t} \int_0^t |\bt_s-\bt_s(z_{s-})|\,|z_{s-}|^{-\ve} \la L\ra_s dK_s
        \to 0 \;\;\text{as} \;\; t\to \infty \;\;(\Pas).
$$

\item[2.] It follows from Eq. (\ref{2-3.1}) that under the conditions of Theorem \ref{2-t3.1} the asymptotic behaviour of the normalized process $(\chi_tz_t)_{t\geq 0}$ coincides  with the asymptotic behaviour of $(L_t/\la L\ra_t)_{t\geq 0} $ as $t\to \infty$.

\item[3.] Assume that the first two conditions in Remark \ref{2-r3.1} hold and besides,
$$
    \sup_{t\geq 0} \bt_t \Dl K_t I_{\{\Dl K_t\neq 0\}}<1 \;\;(\Pas).
$$
In this case, $\ol{\bt}_t=\bt_t I_{\{\bt_t\Dl K_t\neq 1\}}=\bt_t$, $\Gm=\ve^{-1}(-\bt\circ K)$ is a positive increasing process, $\Gm_t\uparrow \infty$ $(\Pas)$ as $t\to \infty$ and if we suppose that $\Gm\simeq \la L\ra$, then taking $\gm=\la L\ra$ we obtain
$$
    \gm\simeq \Gm^2 \la L\ra^{-1}\simeq \Gm
$$
and under the conditions of Theorem \ref{2-t3.1},
$$
    \Gm_t^{1/2} z_t =\frac{L_t}{\la L\ra_t^{1/2}} +R_t, \quad
        R_t\str{P}{\to} 0 \;\;\;\text{as} \;\;\; t\to \infty.
$$
\end{enumerate}
\end{remark}

Note that for the recursive  parametric estimation procedures in the discrete time case, $\Gm^2\la L\ra^{-1} =\Gm$ (see \cite{20}).

\begin{example}\label{2-e6}
The RM stochastic approximation procedure with slowly varying gains (see \cite{2-19}).

Consider the SDE
$$
    dz_t=-\frac{\al R(z_t)}{(1+K_t)^r}\,dK_t+\frac{\al}{(1+K_t)^r} \, dm_t.
$$
Here $K=(K_t)_{t\geq 0}$ is a continuous and increasing non-random function with $K_\infty=\infty$, $1/2<r<1$, $0<\al<1$, $m=(m_t)_{t\geq 0}\in \cM_{\loc}^2(P)$, $d\la m\ra_t=\sg_t^2 d K_t$, $\sg_t^2\to \sg^2>0$ as $t\to \infty$ and non-random regression function $R$ satisfies the following conditions:
$$
    R(0)=0, \quad uR(u)>0 \;\;\;\text{if} \;\;\; u\neq 0,
$$
for each $\ve>0$ $\inf\limits_{\ve<|u|<\frac{1}{\ve}} u R(u)>0$ and
$$
    R(u)=\bt u+\cv(u) \;\;\; \text{with} \;\;\; \cv(u)=O(u^2) \;\;
        \text{as} \;\; u\to 0.
$$

In our notation,
$$
    \bt_t=\frac{\al\bt}{(1+K_t)^r} \quad \text{and} \quad
    \bt_t(u)=\frac{\al R(u)}{u(1+K_t)^r}\,.
$$

It follows from Theorem \ref{t3.1} that $\Pas$
$$
    z_t\to 0 \;\;\;\text{as} \;\;\; t\to \infty.
$$

From subsection \ref{2-s1}
$$
    \chi_t z_t=\frac{L_t}{\la L\ra_t^{1/2}}+R_t,
$$
with $\Gm_t=\ve_t^{-1} (-\bt\circ K)$,
$$
    L_t=\int_0^t \Gm_s \,\frac{\al}{(1+K_s)^r}\,d m_s, \quad
        \chi_t^2 =\Gm_t^2\la L\ra_t^{-1}
$$
and
$$
    R_t=\frac{1}{\la L\ra_t^{1/2}} \int_0^t \Gm_s(\bt_s-\bt_s(z_{s-}))
        z_{s-} dK_s+\frac{z_0}{\la L\ra_t^{1/2}}\,.
$$

On can check that
$$
    (1+K_t)^{-r} \chi_t^2 \to \frac{2\bt}{\al\sg^2}
$$
as $t\to \infty$. Since
$$
    \frac{L_t}{\la L\ra_t^{1/2}} \str{w}{\to} \cN(0,1),
$$
if the convergence $R_t\str{P}{\to} 0$ holds, then
\begin{equation}\label{2-3.2}
    (1+K_t)^{r/2} z_t \str{w}{\to} \cN\left(0,\frac{\al\sg^2}{2\bt}\right).
\end{equation}

It remains to prove that $R_t\str{P}{\to} 0$ as $t\to \infty$. Let us first prove that if $1/2<r<1$, then $\Pas$,
\begin{equation}\label{2-3.3}
    (1+K_t)^{r\dl} z_t\to 0 \;\;\;\text{for all} \;\;\; \dl<1-\frac{1}{2r}\,.
\end{equation}

It is easy to verify that
$$
    (1+K_t)^{2r\dl} =\ve_t^{-1} \left(-\frac{2r\dl}{(1+K)} \circ K\right).
$$

Therefore, the conditions (\ref{2-2.3}) and (\ref{2-2.4}) of Theorem \ref{2-t2.1} can be rewritten as
\begin{equation}\label{2-3.4}
    \int_0^\infty \left[ \frac{2r\dl}{(1+K_t)} -\frac{2\al\bt}{(1+K_t)^r} -
        \frac{2\al \cv(z_t)}{z_t(1+K_t)^r}\right]^+ dK_t<\infty
        \;\;(\Pas)
\end{equation}
and
\begin{equation}\label{2-3.5}
    \int_0^\infty (1+K_t)^{2r\dl} \,\frac{\al^2\sg_t^2}{(1+K_t)^{2r}}\,dK_t<\infty
        \;\;(\Pas).
\end{equation}

The condition (\ref{2-3.4}) holds since
$$
    \left[ \frac{2r\dl}{(1+K_t)} -\frac{2\al\bt}{(1+K_t)^r} -
        \frac{2\al \cv(z_t)}{z_t(1+K_t)^r}\right]^+=0
$$
eventually. The condition (\ref{2-3.5}) is satisfied since $2r-2r\dl>1$ if $\dl<1-1/(2r)$. So, Theorem \ref{2-t2.1} yields Eq. (\ref{2-3.3}). The conditions (d) and (e) of Theorem \ref{2-t3.1} are trivially fulfilled. To check (f) note that from the Kronecker Lemma it suffices to verify that
$$
    \int_0^\infty |\bt_t-\bt_t(z_t)| \gm_t^\ve d K_t <\infty \;\;(\Pas)
$$
for some $\ve$ with $1/2-\dl_0/2<\ve<1/2$, $\dl_0=2-1/r$. For each $\dl$, $0<\dl<\dl_0/2=1-1/(2r)$, we have
\begin{align*}
    \int_0^\infty |\bt_t-\bt_t(z_t)| \gm_t^\ve d K_t & =
        \int_0^\infty \frac{|\cv(z_t)|}{|z_t|^2} \,|z_t|\gm_t^\ve(1+K_t)^{-r}dK_t \\
    & \leq \xi\int_0^\infty (1+K_t)^{-r(a+\dl-\ve)} dK_t
\end{align*}
for some random variables $\xi$. It therefore follows that if there exists a triple $r$, $\dl$, $\ve$ satisfying inequalities
\begin{gather*}
    \frac{1}{2}<r<1, \quad 0<\dl<\frac{1}{2}\,, \quad \ve>0, \quad
        r(1+\dl-\ve)>1, \\
    \frac{1}{2r}-\frac{1}{2}<\ve<\frac{1}{2}\,, \quad \dl<1-\frac{1}{2r}\,,
\end{gather*}
then Eq. (\ref{2-3.2}) holds. It is easy to verify that such a triple exists only for $r>4/5$. It therefore follows that Eq. (\ref{2-3.2}) holds for $r>4/5$.
\end{example}


\section{The Polyak Weighted Averaging Procedure}

\subsection{Preliminaries} \label{proc-s1}
Consider the RM type SDE
\begin{equation}\label{proc1.1}
    z_t=z_0+\int_0^t H_s(z_s)\,dK_s+\int_0^t \ell_s(z_s)\,dm_s,
\end{equation}
where

(1) $\{H_t(u),\;t\geq 0,\;u\in R^1\}$ is a random field described in Section 0;

        (2)$\{M(t,u),\;t\geq 0,\;u\in R^1\}$ is  a random field  such that
$$  M(u)=(M(t,u))_{t\geq 0}\in M_{\loc}^2(P)  $$
for each $u\in R^1$ and $M(t,u)=
\int\limits_0^t\ell_s(u)\,dm_s$, where $m=(m_t)_{t\geq 0}\in M_{\loc}^{2,c}(P)$,
$M(\cdot, 0)\neq 0$; $\ell(u)=(\ell_t(u))_{t\geq 0}$ is a predictable process
for each $u\in R^1$. Denote $\ell_s: =\ell_s(0)$.

        (3) $K=(K_t)_{t\geq 0}$ is a continuous increasing process.

        Suppose this equation has a unique strong solution $z=(z_t)_{t\geq 0}$ on the
whole time interval $[0,\infty)$, such that
$$  \big(M(t)\big)_{t\geq 0}=\bigg(\int_0^t\ell_s(z_s)\,
        dm_s\bigg)_{t\geq 0}\in M_{\loc}^{2,c}(P).         $$

In Section 1 the conditions were established
which guarantee the convergence
\begin{equation}\label{proc1.2}
    z_t\to 0,\;\;\;\text{as}\;\;\;t\to\infty\;\;\; \Pas
\end{equation}

In Section 2, assuming (\ref{proc1.2}) the conditions were stated under which the following
property of $z=(z_t)_{t\geq 0}$ takes place:

        (a) for each $\dl$, $0<\dl<\dl_0$, $0<\dl_0\leq 1$
$$  \gm_t^\dl z_t^2\to 0,\;\;\;\text{as}\;\;\;t\to\infty\;\;\; \Pas
$$
where $\gm=(\gm_t)_{t\geq 0}$ is a predictable increasing process with
$\gm_0=1$, $\gm_\infty=\infty$ $\Pas$

        Further, assuming that $z=(z_t)_{t\geq 0}$ has property (a) with the
process $\gm=(\gm_t)_{t\geq 0}$, equivalent to the process
$\Gm^2\langle L\rangle^{-1}\!\!=\!\!(\Gm_t^2\langle L\rangle_t^{-1})_{t\geq 0}$
(i.e., $\us{t\to\infty}{\lim}\!\frac{\Gm_t^2\langle L\rangle_t^{-1}}{\gm_t}=
\wt{\gm}^{-1}$, $0<\wt{\gm}<\infty$), in Section 2 the conditions were
established under which the asymptotic expansion
\begin{equation}\label{proc1.3}
    \Gm_t\langle L\rangle_t^{1/2}z_t=\frac{L_t}{\langle L
        \rangle_t^{1/2}}+R_t,
\end{equation}
where $R_t \os{P}{\to} 0$ as $t\to\infty$, holds true.

        Here the objects $\gm_t$, $L_t$, $\langle L\rangle_t$ are defined as
follows:
$$  \Gm_t=\ve_t(\bt\circ K):=\exp\bigg(\int_0^t\bt_s \,dK_s\bigg),    $$
where $\bt_t=-H_t'(0)$, $L_t=\int\limits_0^t\Gm_s \ell_s(0)\,dm_s$, $\langle L\rangle$
is the shifted square cha\-rac\-te\-ris\-tics of $L$, i.e.,
$\langle L\rangle_t=1+\langle L\rangle_t^{F,P}$, where $\la L\ra_t^{F,P} =\int\limits_0^t \Gm_s^2 \ell_s^2 d K_s$.

        Consider now the following weighted averaging procedure:
\begin{equation}\label{proc1.4}
    \ol{z}_t=\frac{1}{\ve_t(g\circ K)}\int_0^t z_s \,d\ve_s(g\circ K),
\end{equation}
where $g=(g_t)_{t\geq 0}$ is a predictable process, $g_t\geq 0$ for all
$t\geq 0$, $\Pas$, $\ve_t=\ve_t(g\circ K)=\exp\big(\int\limits_0^t g_s \,dK_s\big)$,
$\int\limits_0^t g_s\,dK_s<\infty$, $t\ge 0$,
$\int\limits_0^\infty g_s\,dK_s=\infty$ $\Pas$

        The aim of this section is to study the asymptotic properties of
the process $\ol{z}=(\ol{z}_t)_{t\geq 0}$, as $t\to\infty$.

        First it should be noted that if $z_t\to 0$ as $t\to\infty$ $\Pas$,
then by the Toeplitz lemma (see, e.g., \cite{proc12}) it immediately follows that
$$  \ol{z}_t\to 0,\;\;\text{as}\;\;t\to\infty,\;\; P-\text{a.s}.    $$

        In subsection \ref{proc-s2} we establish asymptotic distribution of the process
$\ol{z}$ in the ''linear'' case, when $H_t(u)=-\bt_t u$, $M(t,u)\equiv M(t)=
\int\limits_0^t\ell_s \,dm_s$, with deterministic $g,\bt,\ell$ and $K$, and
$d\langle m\rangle_t=dK_t$.

The general case, i.e., when the process $z$ in (\ref{proc1.4}) is the strong solution
of SDE (\ref{proc1.1}), is considered in subsection \ref{proc-s3}.

\subsection{Asymptotic properties of $\ol{z}$. ``Linear'' Case}\label{proc-s2}

        In this subsection we consider the ``linear'' case, when SDE (\ref{proc1.1}) is
of the form
\begin{equation}\label{proc2.1}
    dz_t=-\bt_t z_t \,dK_t+\ell_t \,dm_t,\;\;z_0,
\end{equation}
where $K=(K_t)_{t\geq 0}$ is a deterministic increasing function,
$\bt=(\bt_t)_{t\geq 0}$ and $\ell=(\ell_t)_{t\geq 0}$ are deterministic
functions, $\bt_t\geq 0$ for all $t\geq 0$, $\int\limits_0^\infty\bt_s dK_s=\infty$,
$\int\limits_0^t\bt_s dK_s<\infty$, for all $t\geq 0$ and $\int\limits_0^\infty
\ell_s^2 dK_s<\infty$.

        Define the following objects:
$$  \Gm_t=\exp\bigg(\int_0^t \bt_s \,dK_s\bigg),\;\;
        L_t=\int_0^t \Gm_s\ell_s \,dm_s,\;\;t\geq 0.      $$

        Under the above conditions we have $\Gm_\infty=\infty$,
$\Gm_\infty^2 \langle L\rangle_\infty^{-1}=\infty$.

        Indeed, application of the Kronecker lemma (see, e.g., \cite{proc12}) yields
$$  \Gm_t^{-2} \langle L\rangle_t=\frac{1}{\Gm_t^2}\int_0^t \Gm_s^2\ell_s^2 \,dK_s\to 0\;\;\text{as}
        \;\;t\to\infty,     $$
since $\int\limits_0^\infty \ell_s^2 \,dK_s<\infty$.

        Solving equation (\ref{proc2.1}), we get
\begin{equation}\label{proc2.2}
    z_t=\Gm_t^{-1}\bigg\{z_0+\int_0^t \Gm_s\ell_s\, dm_s\bigg\},\;\;\;t\geq 0,
\end{equation}
From (\ref{proc2.2}) and CLT for continuous martingales (see, e.e., \cite{proc12}) it directly
follows that
\begin{align}
    & z_t\to 0,\;\;\text{as}\;\;t\to\infty,  \label{proc2.3} \\
    & \Gm_t \langle L\rangle_t^{-1/2} z_t\os{d}{\to} \xi,\;\;\text{as}
        \;\;t\to\infty,\label{proc2.4}
\end{align}
where ``$\os{d}{\to}$'' denotes the convergence in distribution, $\xi$ is a
standard normal random variable ($\xi\in N(0,1)$).

        Let now $\ol{z}=(\ol{z}_t)$ be an averized process defined by (\ref{proc1.4})
with the deterministic function $g=(g_t)_{t\geq 0}$,
$\int\limits_0^\infty g_t dK_t=\infty$, $\int\limits_0^t g_s dK_s<\infty$ for all
$t\geq 0$.

        Denote $B_t=\int\limits_0^t\Gm_s^{-1}d\ve_s$, $\wt{B}_t=
\int\limits_0^t(B_t-B_s)^2 d\langle L\rangle_s$, $\ve_t=\ve_t(g\circ K)$.

\begin{proposition}\label{proc-p2.1}
Suppose that $\langle L\rangle_\infty=\infty$, $\langle L\rangle\circ B_\infty=
\infty$, $\wt{B}_\infty=\infty$.

Then
\begin{equation}\label{proc2.5}
 \ve_t\wt{B}_t^{-1/2}\ol{z}_t \os{d}{\to} \xi,\;\;\text{as}\;\;t\to\infty,
        \;\;\xi\in N(0,1),
\end{equation}
\end{proposition}

\noindent{\it Proof.\/}
Substituting (\ref{proc2.2}) in (\ref{proc1.4}) and integrating by parts, we get
$$  \ol{z}_t=\frac{z_0 B_t}{\ve_t}+\ve_t^{-1}\int_0^t(B_t-B_s)\,dL_s    $$

Hence
\begin{equation}\label{proc2.6}
\ve_t\wt{B}_t^{-1/2}\ol{z}_t\!=\!z_0 B_t/(\wt{B}_t)^{1/2}\!+
    (\wt{B}_t)^{-1/2}\!\int_0^t(B_t-B_s)\,dL_s\!=\!I_t^1+I_t^2.
\end{equation}
First we will show that
$$  I_t^1\to 0,\;\;\text{as}\;\;t\to\infty.     $$
It is easy to check that
\begin{equation}\label{proc2.7}
    \wt{B}_t=\int_0^t(B_t-B_s)^2 d\langle L\rangle_s=2\int_0^t
        \bigg(\int_0^s\langle L\rangle_u\,dB_u\bigg)\,dB_s.
\end{equation}
We rewrite $(I_t^1)^2$ in the form
$$  (I_t^1)^2=B_t^2(\wt{B}_t)^{-1}=
    \frac{2\int_0^t B_s(\int_0^s\langle L\rangle_u\,dB_u)^{-1}d
        \wt{B}_s}{\wt{B}_t}.    $$
Since $\wt{B}_\infty=\infty$, applying the Toeplitz lemma, we obtain
$$  \lim_{t\to\infty}(I_t^1)^2=\lim_{t\to\infty}
        \frac{B_t}{\int_0^t\langle L\rangle_u\,dB_u}.    $$
Further, as $\int\limits_0^\infty\langle L\rangle_u\,dB_u=\infty$, applying again the
Toeplitz lemma we get
$$  \lim_{t\to\infty}\frac{B_t}{\int_0^t\langle L\rangle_u\,dB_u}=
    \lim_{t\to\infty}\frac{\int_0^t\langle L\rangle_u^{-1}
        \langle L\rangle_u\,dB_u}{\int_0^t\langle L\rangle_u\,dB_u}=
        \lim_{t\to\infty}\frac{1}{\langle L\rangle_t}=0.       $$

It remains to show that
$$  I_t^2 \os{d}{\to} \xi,\;\;\text{as}\;\;t\to\infty,\;\;\xi\in N(0,1).  $$

        For any sequence $t_n\to\infty$ as $n\to\infty$ we define the sequence
of martingales as follows:
$$  M^n(u)=\frac{\int_0^{t_n u}(B_{t_n}-B_s)\,dL_s}{(\int_0^{t_n}(B_{t_n}-
        B_s)^2d\langle L\rangle_s)^{1/2}}, \quad u\in [0,1].      $$
Obviously, $\langle M^n\rangle_1=1$ for each $n\geq 1$, and from the CLT for
continuous martingales we have
$$  M^n(1)=I_{t_n}^2\os{d}{\to} \xi\;\;\text{as}\;\;n\to\infty,\;\;
        \xi\in N(0,1).\;\;\qed      $$

\begin{remark}\label{proc-r2.1}
It should be noted that $\ve_\infty\wt{B}_\infty^{-1/2}=\infty$.

        Indeed, by the Toeplitz lemma,
\begin{gather*}
    \lim_{t\to\infty}\frac{\wt{B}_t}{\ve_t^2}=\lim_{t\to\infty}\frac{\int_0^t
        (\int_0^s\langle L\rangle_u dB_u)\Gm_s^{-1}\ve_s^{-1}d\ve_s^2}
        {\ve_t^2}=
    \lim_{t\to\infty}\frac{1}{\Gm_t\ve_t}\int_0^t\langle L\rangle_s
        \Gm_s^{-1}d\ve_s   \\
    =\lim_{t\to\infty}\frac{1}{\Gm_t\ve_t}\int_0^t\langle L\rangle_s
        \Gm_s^{-2}\Gm_s\,d\ve_s\leq\lim_{t\to\infty}
        \frac{1}{\ve_t}\int_0^t\langle L_s\rangle\Gm_s^{-2}d\ve_s=0.
\end{gather*}
since $\ve_{\infty}=\infty$ and $\langle L\rangle_{\infty} \Gm_{\infty}^{-2}=0$.
\end{remark}

        Define now the process $\ve_t^{(\al)}:=\ve_t(g^{(\al)}\circ K)$ as
follows: Let $(\al_t)_{t\geq 0}$ be a function, $\al_t\geq 0$ for all
$t\geq 0$, and $\us{t\to\infty}{\lim}\al_t=\al$, $0<\al<\infty$. We define
$\ve^{(\al)}$ by the relation
\begin{equation}\label{proc2.8}
    \ve_t^{(\al)}=1+\int_0^t \al_s\bt_s\langle L\rangle_s^{-1}
        \Gm_s^2 dK_s.
\end{equation}
        Note that
\begin{equation}\label{proc2.9}
    \langle L\rangle_t\Gm_t^{-2}\ve_t^{(\al)}g_t^{(\al)}/\bt_t=\al_t.
\end{equation}
        Indeed, it is easily seen that if
$$  \ve_t(\psi)=1+\int_0^t\vf_s dK_s,\;\;\text{then}\;\;\psi_t=
        \frac{\vf_t}{\ve_t(\psi\circ K)}.   $$
Hence, if $\ve_t(g^{(\al)}\circ K)=\ve_t^{(\al)}$, then
$$  g_t^{(\al)}=\al_t\bt_t\langle L\rangle_t^{-1}\Gm_t^2/\ve_t^{(\al)}, $$
and (\ref{proc2.9}) follows.

It should be also noted that for each $(\al_t)_{t\geq 0}$ with $\lim\limits_{t\to \infty} \al_t=\al$,
$$
    \lim_{t\to \infty} \frac{\ve_t^\al} {1+\int_0^t \al \bt_s \la L\ra_s^{-1} \Gm_s^2 d K_s}=1.
$$

\begin{proposition}\label{proc-p2.2}
Let $\ol{z}^{(\al)}=(\ol{z}_t^{{}{(\al)}})_{t\geq 0}$ be an averized process
corresponding to the averaging process $\ve^{(\al)}$ $($see $(\ref{proc1.4}))$, i.e.,
$$  \ol{z}_t^{(\al)}=\frac{1}{\ve_t^{(\al)}}\int_0^t z_s d\ve_s^{(\al)},
        \;\;\;t\geq 0.    $$
Then
$$  \bigg(1+\int_0^t\bt_s\langle L\rangle_s^{-1}\Gm_s^2 dK_s\bigg)^{1/2}
        \ol{z}_t^{(\al)}\os{d}{\to} \sqrt{2}\,\xi,\;\;\text{as}\;\;
        t\to\infty,\;\;\xi\in N(0,1).    $$
\end{proposition}

\begin{proof}
By virtue of Proposition \ref{proc-p2.1}, it is sufficient to show that
\begin{equation}\label{proc2.10}
    \frac{\ve_t^{(1)}}{(\ve_t^{(\al)})^2(\wt{B}_t^{(\al)})^{-1}}\to 2,
        \;\;\text{as}\;\;t\to\infty.
\end{equation}
where $B_t^{(\al)}=\int\limits_0^t\Gm_s^{-1}\,d\ve_s^{(\al)}$,
$\wt{B}_t^{(\al)}=\int\limits_0^t(B_t^{(\al)}-B_s^{(\al)})^2\,d\langle L\rangle_s$.

We have
\begin{gather*}
    \lim_{t\to\infty}
        \frac{\ve_t^{(1)}}{(\ve_t^{(\al)})^2(\wt{B}_t^{(\al)})^{-1}}=
    \lim_{t\to\infty}\frac{\ve_t^{(1)}}{\ve_t^{(\al)}}
        \frac{\wt{B}_t^{(\al)}}{\ve_t^{(\al)}}=
        \frac{1}{\al}\lim_{t\to\infty}
        \frac{\wt{B}_t^{(\al)}}{\ve_t^{(\al)}}   \\
    =\frac{1}{\al}\lim_{t\to\infty}
        \frac{2\int_0^t(\int_0^s\langle L\rangle_u d B_u^{(\al)})
        \Gm_s^{-1}d\ve_s^{(\al)}}{\ve_t^{(\al)}}
    =\!\frac{2}{\al}\lim_{t\to\infty}\frac{1}{\Gm_t}
        \!\int_0^t\langle L\rangle_s dB_s^{(\al)}.
\end{gather*}

        Applying now relation (\ref{proc2.9}) and Toeplitz lemma, we get
\begin{gather*}
    \lim_{t\to\infty}\frac{1}{\Gm_t}\int_0^t\langle L\rangle_s dB_s^{(\al)}=
    \lim_{t\to\infty}\frac{1}{\Gm_t}\int_0^t\langle L\rangle_s
        \Gm_s^{-1} d\ve_s^{(\al)}   \\
    =\lim_{t\to\infty}\frac{1}{\Gm_t}\!\int_0^t\langle L\rangle_s
        \Gm_s^{-2} \ve_s^{(\al)}\!\frac{g_s^{(\al)}}{\bt_s}\Gm_s\bt_s dK_s\!
    =\!\lim_{t\to\infty}\frac{1}{\Gm_t}\!\int_0^t\!\al_s d\Gm_s=\al.\;\;\qedhere
\end{gather*}
\end{proof}

\begin{corollary}\label{proc-c2.1}
Let $\gm=(\gm_t)_{t\geq 0}$ be an increasing process such that $\gm_0=1$,
$\gm_\infty=\infty$ and
$$  \lim_{t\to\infty}\frac{\langle L\rangle_t^{-1}\Gm_t^2}{\gm_t}=
        \wt{\gm}^{-1}\;\;\text{as}\;\;t\to\infty,     $$
where $\wt{\gm}$ is a constant, $0<\wt{\gm}<\infty$. Then
\vskip+0.1cm
        $(1)$ $\gm_t^{1/2}z_t \os{d}{\to} \wt{\gm}^{1/2}\,\xi$, as $t\to\infty$;
\vskip+0.1cm
        $(2)$ $(1+\int\limits_0^t\gm_s\bt_s dK_s)^{1/2}\ol{z}_t^{(\al)}\os{d}{\to}
\sqrt{2\wt{\gm}}\,\xi$ as $t\to\infty$;
\vskip+0.1cm
        $(3)$ if $\gm_s\bt_s\!=\!1$ eventually, then $(1\!+\!K_t)^{1/2}\ol{z}_t^{(\al)}\!\to\!
\sqrt{2\wt{\gm}}\,\xi$ as $t\to\infty$, $\xi\!\in\! N(0,1)$.
\end{corollary}

\begin{remark}\label{proc-r2.2}
(1) Let $\gm=(\gm_t)_{t\geq 0}:=(\frac{\bt_t}{\ell_t^2})_{t\geq 0}$ be an
increasing process, $\gm_0=1$, $\gm_\infty=\infty$, $d\gm\ll dK$. Then $\gm$
can be represented as the solution of the SDE $d\gm_t=\gm_t \lb_t dK_t$,
$\gm_0=1$, with some $\lb=(\lb_t)_{t\geq 0}$.

Assume that $\lb_t\to 0$ as $t\to\infty$ and $\lb_t/\bt_t\to 0$ as
$t\to\infty$. Then
$$  \lim_{t\to\infty}\frac{\langle L\rangle_t^{-1}\Gm_t^2}{\gm_t}=2.   $$
        Indeed,
$$  \frac{\langle L\rangle_t^{-1}\Gm_t}{\gm_t}=
        \frac{\Gm_t^2\gm_t^{-1}}{\langle L\rangle_t}.   $$
and integration by parts and application of the Toeplitz lemma yield
\begin{gather*}
    \frac{\langle L\rangle_t^{-1}\Gm_t^2}{\gm_t}=
        \frac{\int_0^t2\Gm_s^2\bt_s\gm_s^{-1}dK_s-
        \int_0^t\Gm_s^2\gm_s^{-2}\gm_s\lb_s dK_s}
        {\langle L\rangle_t}     \\
    =2-\frac{1}{\langle L\rangle_t}\int_0^t
        \frac{\lb_s}{\gm_s\ell_s^2}\,d\langle L\rangle_s=
        2-\frac{1}{\langle L\rangle_t}
        \int_0^t\frac{\lb_s}{\bt_s}d\langle L\rangle_s\to 2 \;\;\text{as} \;\; t\to \infty.
\end{gather*}

Thus if we put $\gm_t=\frac{\bt_t}{\ell_t^2}$ in the above Corollary \ref{proc-c2.1}, then all assertions hold true with
$$  \gm_t=\frac{\bt_t}{\ell_t^2},\;\;\;\wt{\gm}=\frac{1}{2};     $$

        (2) Let $\ell_t=\sg\bt_t$, where $\bt_t$ is a decreasing function,
$\bt_t\to 0$ as $t\to\infty$, $d\bt_t=-\bt_t' dK_t$, $\bt_t'>0$.

Then, if
$$  \bt_t'/\bt_t^2\to 0\;\;\text{as}\;\;t\to\infty,  $$
we have
$$  \lim_{t\to\infty}\langle L\rangle_t^{-1}\Gm_t^2\bt_t=2\sg^2.     $$
\end{remark}

        From Proposition \ref{proc-p2.2} immediately follows
$$  (1+K_t)^{1/2}\ol{z}_t^{(\al)}\os{d}{\to} \sqrt{2}\,\sg\,\xi\;\;\text{as}\;\;t\to\infty.  $$

\begin{remark}\label{proc-r2.3}
Summarizing the above statements, we conclude that: as \linebreak $t\to\infty$,
\vskip+0.1cm
        (a) \;$(\ve_t^{(1)})^{1/2}\ol{z}_t^{(\al)}\os{d}{\to}\sqrt{2}\,\xi$;
\vskip+0.1cm
        (b) \;$(\ve_t^{(\al)})^{1/2}\ol{z}_t^{(\al)} \os{d}{\to}
\sqrt{\frac{2}{\al}}\,\xi$;
\vskip+0.1cm
        (c) \;$(\ve_t^{(1)})^{1/2}\ol{z}_t^{(1)}\os{d}{\to} \sqrt{2}\,\xi$;
\vskip+0.1cm
        (d) \;$\Gm_t \langle L\rangle_t^{-1/2}z_t\os{d}{\to} \,\xi$,
\vskip+0.1cm
\hskip-0.3cm where $\xi\in N(0,1)$
\end{remark}

\setcounter{example}{0}
\begin{example}\label{ex1}
Standard ``Linear'' Procedure.

Let $\bt_t=\al\bt(1+K_t)^{-1}$, $\ell_t=\al\sg(1+K_t)^{-1}$, $\al\bt>0$,
$2\al\bt>1$. Then $\Gm_t^2\langle L\rangle_t^{-1}=
\frac{2\al\bt-1}{\al^2\sg^2}(1+K_t)$. Hence from (\ref{proc2.4}) follows
$$  (1+K_t)^{1/2} z_t\os{d}{\to} \frac{\al\sg}{\sqrt{2\al\bt-1}}\,\xi\;\;
    \text{as}\;\;t\to\infty,\;\;\xi\in N(0,1).     $$

        On the other hand, $\ve_t^{(1)}=1+\int\limits_0^t\bt_s
\Gm_s^2\langle L\rangle_s^{-1}\,dK_s=1+\frac{\al^2\sg^2}{\bt(2\al\bt-1)}\,K_t$,
and it follows from Proposition \ref{proc-p2.2} that if we define
$$  \ol{z}_t^{(1)}=\frac{1}{\ve_t^{(1)}}\int_0^t z_s\,d\ve_s^{(1)},\;\;\;
        \ol{z}_t=\frac{1}{1+K_t}\int_0^t z_s\,dK_s,       $$
then
\begin{align*}
    & (1+K_t)^{1/2}\ol{z}_t^{(1)}\os{d}{\to} \sg
        \sqrt{\frac{2\al}{\bt(2\al\bt-1)}}\,\xi\;\;\;\text{as}\;\;\;t\to\infty,  \\
    & (1+K_t)^{1/2}\ol{z}_t \os{d}{\to}\sg
        \sqrt{\frac{2\al}{\bt(2\al\bt-1)}}\,\xi\;\;\;\text{as}\;\;\;t\to\infty,
        \;\;\;\xi\in N(0,1).
\end{align*}
Hence the rate of convergence is the same, but the asymptotic variance of averized procedure $\ol{z}$ is smaller than of the initial one.
\end{example}

\begin{example}\label{ex2}
``Linear'' Procedure with slowly varying gains.

Let $\bt_t=\al\bt(1+K_t)^{-r}$, $\ell_t=\al\sg(1+K_t)^{-r}$, $\al\bt>0$,
$\frac{1}{2}<r<1$. Then the process $\gm=(\gm_t)_{t\geq 0}$ defined
in Remark \ref{proc-r2.2} is $\gm_t=\frac{\bt}{\al\sg^2}(1+K_t)^r$, $d\gm_t=
\frac{r\bt}{\al\sg^2}(1+K_t)^r\frac{dt}{1+K_t}$. Hence
$\lb_t=\frac{r\bt}{\al\sg^2}(1+K_t)^{-1}$, $\lb_t/\bt_t\to 0$ as $t\to\infty$.
From Remark \ref{proc-r2.2} it follows that
\begin{equation}\label{proc2.11}
    \lim_{t\to\infty}\frac{\Gm_t^2\langle L\rangle_t^{-1}}{\gm_t}=2,
\end{equation}
and from (\ref{proc2.4}) we have
$$  (1+K_t)^{r/2}z_t \os{d}{\to}\sg\sqrt{\frac{\al}{2\bt}}\,\xi,\;\;\text{as}\;\;
        t\to\infty,\;\;\xi\in N(0,1).    $$
        On the other hand,
$$  \ve_t^{(1)}=1\!+\!\int_0^t\bt_s\Gm_s^2\langle L\rangle_s^{-1}dK_s=1+
        \!\int_0^t\!\bt_s\gm_s\!
        \frac{\Gm_s^2\langle L\rangle_s^{-1}}{\gm_s}dK_s\!
    =1+\!\frac{\bt^2}{\sg^2}\!\int_0^t\!\frac{\Gm_s^2\langle L\rangle_s^{-1}}
        {\gm_s}dK_s.      $$

        Hence take into the account (\ref{proc2.11}), by the Toeplitz Lemma we have
$$  \frac{\ve_t^{(1)}}{1+K_t}\to 2\frac{\bt^2}{\sg^2}\;\;\;\text{as}\;\;\;t\to\infty.  $$
Therefore from Remark \ref{proc-r2.3}\,(c) we get
$$  (1+K_t)^{1/2}\ol{z}_t^{(1)} \os{d}{\to}\frac{\sg}{\bt}\,\xi\;\;\;
        \text{as}\;\;\;t\to\infty,\;\;\;\xi\in N(0,1).    $$
and
$$  (1+K_t)^{1/2}\ol{z}_t \os{d}{\to}\frac{\sg}{\bt}\,\xi\;\;\;
        \text{as}\;\;\;t\to\infty,\;\;\;\xi\in N(0,1).    $$
\end{example}

Note that if $\al\bt>2$, then the asymptotic variance of $\ol{z}$ is smaller than of the initial one.

\begin{example}\label{ex3}
Let $\bt_t=(1+t)^{-(\frac{1}{2}+\al)}$, where $\al$ is a constant,
$0<\al<\frac{1}{2}$, $\ell_t^2=(1+t)^{-(\frac{3}{2}+\al)}$. Then if we take
$\gm_t=\bt_t/\ell_t^2=(1+t)^{-(\frac{1}{2}+\al)}(1+t)^{\frac{3}{2}+\al}=
1+t$, $d\gm_t=\gm_t\frac{1}{1+t}dt$, then $\lb_t=(1+t)^{-1}$,
$\frac{\lb_t}{\bt_t}=(1+t)^{-1}(1+t)^{\frac{1}{2}+\al}=
(1+t)^{\al-\frac{1}{2}}\to 0$ as $t\to\infty$. Therefore, from Remark \ref{proc-r2.2}\,(1)
follows
$$  \lim_{t\to\infty}\frac{\Gm_t^2\langle L\rangle_t^{-1}}{1+t}=2.   $$
and from Corollary \ref{proc-c2.1}\,(1) we have
$$  (1+t)^{1/2}z_t \os{d}{\to}\sqrt{\frac{1}{2}}\,\xi,\;\;
        \text{as}\;\;t\to\infty,\;\;\xi\in N(0,1).    $$
If we now define
\begin{gather*}
    \ve_t^{(1)}=1+\int_0^t\bt_s\langle L\rangle_s^{-1}
        \Gm_s^2 ds    \\
    =1+\int_0^t\bt_s\gm_s\frac{\Gm_s^2\langle L\rangle_s^{-1}}{\gm_s}ds=
        1+\int_0^t(1+s)^{\frac{1}{2}-\al}
        \frac{\Gm_s^2\langle L\rangle_s^{-1}}{\gm_s}ds,
\end{gather*}
then $\ve_t^{(1)}/(1+t)^{3/2-\al}\to\frac{4}{3-2\al}$, and from Corollary \ref{proc-c2.1}\,(2)
we obtain
$$  (1+t)^{3/2-\al}\ol{z}_t^{(1)}\to\sqrt{\frac{4}{3-2\al}}\,\xi\;\;\;
        \text{as}\;\;\;t\to\infty,\;\;\;\xi\in N(0,1).    $$
\end{example}

In the last two examples the rate of convergence of the averized procedure is higher than of the initial one.

\subsection{Asymptotic properties of $\ol{z}$. General case}\label{proc-s3}

        In this subsection we study the asymptotic properties of the averized
process $\ol{z}=(\ol{z})_{t\geq 0}$ defined by (\ref{proc1.4}), where
$z=(z_t)_{t\geq 0}$ is the strong solution of SDE (\ref{proc1.1}).

        In the sequel we will need the following objects:
$$  \bt_t=-H_t'(0),\;\;\;\bt_t(u)=\begin{cases}
        -\frac{H_t(u)}{u},\;\;&\;\;\text{if}\;\;u\neq 0, \\
        \bt_t,\;\;&\;\;\text{if}\;\;u=0, \end{cases}       $$
$\Gm_t=\ve_t(\bt\circ K)=\exp\big\{\int\limits_0^t\bt_s d K_s\big\}$,
$L_t=\int\limits_0^t\Gm_s\ell_s dm_s$, $\ell_t=\ell_t(0)$, $d\langle m\rangle_t=
dK_t$.

        Assume that processes $K$, $\bt$ and $\ell$ are deterministic. We
rewrite equation (\ref{proc1.1}) in terms of these objects.
\begin{equation}\label{proc3.1}
dz_t=-\bt_t z_t d K_t+\ell_t d m_t+(\bt_t-\bt_t(z_t))z_t dK_t+
        (\ell_t(z_t)-\ell_t)d m_t.
\end{equation}
Further, solving formally the last equation as the linear one w.r.t. $z$, we
get
\begin{equation}\label{proc3.2}
    z_t=\Gm_t^{-1}\bigg[z_0+L_t+\int_0^t\Gm_s d\ol{R}_1(s)+
        \int_0^t\Gm_s d\ol{R}_2(s)\bigg],
\end{equation}
where
\begin{align*}
    \Gm_t & =\exp \bigg( \int_0^t \bt_s\,d S_s\bigg), \\
    L_t & =\int_0^t \Gm_s \ell_s \, dm_s, \\
    d\ol{R}_1(t) &=\big(\bt_t-\bt_t(z_t)\big)z_t d K_t,   \\
    d\ol{R}_2(t) &=\big(\ell_t(z_t)-\ell_t\big)d m_t.
\end{align*}
Consider now the following averaging procedure:
\begin{equation}\label{proc3.3}
    \ol{z}_t=\frac{1}{\ve_t}\int_0^t z_s d\ve_s,
\end{equation}
where the process $\ve_t:=\ve_t=1+\int\limits_0^1\Gm_s^2
\langle L\rangle_s^{-1}\bt_s dK_s$, i.e., is defined by relation (\ref{proc2.8}) with
$\al_t=1$.

   In the sequel it will be assumed that the functions $\bt$, $\ell$, $K$,
$g$ satisfy all conditions imposed on the corresponding functions in
Propositions \ref{proc-p2.1} and \ref{proc-p2.2}.

        Let $\gm=(\gm)_{t\geq 0}$ be an increasing function such that
$\gm_0=1$, $\gm_\infty=\infty$ and $\us{t\to\infty}{\lim}
\frac{\Gm_t^2\langle L\rangle_t^{-1}}{\gm_t}=\wt{\gm}^{-1}$.

\begin{theorem}\label{proc-t3.1}
Suppose that $\gm_t^\dl z_t^2\to 0$ as $t\to\infty$ for all $\dl$,
$0<\dl<\dl_0$, $0<\dl_0\leq 1$. Assume that the
following conditions are satisfied:
\vskip+0.2cm
{\rm (i)} there exists $\dl$, $0<\dl<\dl_0/2$ such that
$$  \int_0^\infty\ve_t^{-1/2}\gm_t^{-\dl}\big|\bt_t(z_t)-\bt_t\big|
    d K_t<\infty,\;\;\Pas;        $$

{\rm (ii)} $\frac{\langle N\rangle_t}{\langle L\rangle_t}\to 0$ as
$t\to\infty$, where
$N_t=\int\limits_0^t\Gm_s\big(\ell_s(z_s)-\ell_s\big)dm_s$.
\vskip+0.2cm
Then
$$  \ve_t^{1/2}\ol{z}_t \os{d}{\to}\;\;\sqrt{2}\,\xi\;\;\;\text{as}\;\;\;
        t\to\infty,  \;\;\; \xi\in N(0,1).    $$
\end{theorem}

\begin{proof}
Substituting (\ref{proc3.2}) in (\ref{proc3.3}), we obtain
\begin{equation}\label{proc3.4}
    \ol{z}_t=\frac{z_0 B_t}{\ve_t}+\frac{1}{\ve_t}\int_0^t L_s dB_s+R_t^1+
        R_t^2,
\end{equation}
where
$$  R_t^i=\frac{1}{\ve_t}\int_0^t\int_0^s\bigg(L_u d\ol{R}_i(u)\bigg)dB_s,\;\;
        i=1,2,    $$
$dB_t\equiv\Gm_t^{-1}d\ve_t$.

        Integration of the second term in (\ref{proc3.4}) by parts results in
\begin{equation}\label{proc3.5}
    \ol{z}_t=\frac{z_0 B_t}{\ve_t}+\frac{1}{\ve_t}\int_0^t(B_t-B_s)dL_s+
        R_t^1+R_t^2.
\end{equation}
        Denoting $\wt{B}_t=\int_0^t(B_t-B_s)^2 d\langle L\rangle_s$, we have
\begin{equation}\label{proc3.6}
\ve_t\wt{B}_t^{-1/2}\ol{z}_t=z_0\frac{B_t}{(\wt{B}_t)^{1/2}}+
        \frac{\int_0^t(B_t-B_s)dL_s}{(\wt{B}_t^{1/2})}+
        \frac{R_t^1}{(\wt{B}_t)^{1/2}}+\frac{R_t^2}{(\wt{B}_t)^{1/2}}.
\end{equation}

        As is seen, the first two terms in the right-hand side of (\ref{proc3.6})
coincide with those in (\ref{proc2.6}), and since by our assumption the conditions of
Propositions \ref{proc-p2.1} and \ref{proc-p2.2} are satisfied, taking into the account (\ref{proc2.10}) with
$\al=1$ one can conclude that it suffices to show that
\begin{equation}\label{proc3.7}
    \ve_t^{1/2} R_t^i \os{P}{\to} 0,\;\;\text{as}\;\;t\to\infty,\;\;i=1,2.
\end{equation}

        Let us investigate the case $i=1$.
\begin{gather*}
    \ve_t^{1/2} R_t^1=\frac{1}{\ve_t^{1/2}}\int_0^t\bigg(\int_0^s\Gm_u
        d\ol{R}_1(u)\bigg)dB_s=\frac{1}{\ve_t^{1/2}}\int_0^t
        \bigg(\int_0^s\Gm_u d\ol{R}_1(u)\bigg)\Gm_s^{-1}d\,\ve_s   \\
    =\frac{2}{\ve_t^{1/2}}\int_0^t
        \bigg(\int_0^s\Gm_u d\ol{R}_1(u)\bigg)\Gm_s^{-1}\ve_s^{1/2}
        d\ve_s^{1/2}.
\end{gather*}

        Since $\ve_t$ is an increasing process, $\ve_\infty=\infty$, by
virtue of the Toeplitz Lemma it is sufficient to show that
$$  A_t=\frac{1}{\Gm_t\ve_t^{1/2}}\int_0^t\Gm_s d\ol{R}_1(s)\to 0,\;\;
        \text{as}\;\;t\to\infty,\;\; \Pas      $$

        For all $\dl$, $0<\dl<\dl_0/2$, since $\gm_t^\dl|z_t|\to 0$ as $t\to \infty$, we have
\allowdisplaybreaks
\begin{gather*}
    |A_t|\leq\frac{1}{\Gm_t\ve_t^{1/2}}\int_0^t\Gm_s|\bt_s-\bt_s(z_s)|
        |z_s|dK_s \\
    \leq\const(\om)\frac{1}{\Gm_t\ve_t^{1/2}}\int_0^t\Gm_s\gm_s^{-\dl}
        |\bt_s-\bt_s(z_s)|dK_s   \\
    =\const(\om)\frac{1}{\Gm_t\ve_t^{1/2}}\int_0^t\Gm_s\ve_s^{1/2}
        \ve_s^{-1/2}\gm_s^{-\dl}|\bt_s-\bt_s(z_s)|dK_s
\end{gather*}
Now the desirable convergence $A_t\to 0$ as $t\to\infty$ follows from
condition (i) and the Kronecker lemma applied to the last term of the
previous inequalities.

        Consider now the second term
\begin{equation} \label{proc3.8}
    \ve_t^{1/2} R_t^2=\frac{1}{\ve_t^{1/2}}\int_0^t\bigg(\int_0^s\Gm_u
        \big(\ell_u(z_u)-\ell_u\big) dm_u\bigg)\Gm_s^{-1}d\ve_s.
\end{equation}

        Denoting $N_t=\int\limits_0^t\Gm_s(\ell_s(z_s)-\ell_s) dm_s$ and integrating
by parts, from (\ref{proc3.8}) we get
$$  \ve_t^{1/2} R_t^2=\frac{1}{\ve_t^{1/2}}\int_0^t(B_t-B_s)dN_s.     $$

        Further, for any sequence $t_n$, $t_n\to\infty$ as $n\to\infty$ let
us consider a sequence of martingales $Y_u^n$, $u\in[0,1]$ defined as follows:
$$  Y_u^n=\frac{1}{\ve_{t_n}^{1/2}}\int_0^{t_n u}(B_{t_n}-B_s)dN_s,\;\;
        \langle Y^n\rangle_1=\frac{1}{\ve_{t_n}}\int_0^{t_n}(B_{t_n}-
        B_s)^2 d\langle N\rangle_s.      $$

        Now, if we show that $\langle Y^n\rangle_1 \os{P}{\to} 0$ as $n\to\infty$,
then from the well-known fact that
$\langle Y^n\rangle_1  \os{P}{\to} 0\Rightarrow Y_1^n \os{P}{\to} 0$
(see, e.g., \cite{proc12}) we get $\ve_{t_n}^{1/2}R_{t_n}^2\to 0$ as $n\to\infty$, and
hence $\ve_t^{1/2}R_t^2\to 0$, as $t\to\infty$.

        Thus we have to show that
$$  \frac{1}{\ve_t}\int_0^t(B_t-B_s)^2 d\langle N\rangle_s\to 0\;\;
        \text{as}\;\;t\to\infty,\;\;\Pas     $$

           Using the relation $\int\limits_0^t(B_t-B_s)^2 d\langle N\rangle_s=
2\int\limits_0^t\big(\int\limits_0^s\langle N\rangle_u dB_u\big)dB_s$, we have to show:
\begin{equation}\label{proc3.9}
\frac{1}{\ve_t}\!\int_0^t\!(B_t-B_s)^2 d\langle N\rangle_s\!=\!
        2\frac{1}{\ve_t}\!\int_0^t\bigg(\!\int_0^s\!\langle N\rangle_u
        dB_u\bigg)\Gm_s^{-1}d\ve_s\!\to\! 0\;\;\text{as}\;\;t\!\to\!\infty
\end{equation}
Applying the Toeplitz lemma to (\ref{proc3.9}) it suffices to show that
\begin{equation}\label{proc3.10}
    \frac{1}{\Gm_t}\int_0^t\langle N\rangle_s dB_s\to 0\;\;
        \text{as}\;\;t\to\infty,\;\;\Pas
\end{equation}
But
\begin{equation}\label{proc3.11}
 \frac{1}{\Gm_t}\int_0^t\langle N\rangle_s dB_s=
        \frac{1}{\Gm_t}\int_0^t\langle N\rangle_s \Gm_s^{-1} d\ve_s=
        \frac{1}{\Gm_t}\int_0^t\langle N\rangle_s \la L\ra_s^{-1}d\Gm_s
\end{equation}
(recall that $d\ve_s=\Gm_s^2\la L\ra_s^{-1} \bt_s d K_s$).

        Applying again the Toeplitz lemma to (\ref{proc3.11}) we can see that  (\ref{proc3.10})
follows from condition (ii).
\end{proof}

\begin{corollary}\label{proc-c3.1}
Let $H_t(u)=-\bt_t u+v_t(u)$, where for each $t\in[0,\infty)$, \linebreak
$|\frac{v_t(u)}{u^2}-v_t|\to 0$ as $u\to 0$, $\Pas$

        Assume that the following condition is satisfied:

{\rm (i$')$} there exists $\dl$, $0<\dl<\dl_0$ such that
$$  \int_0^\infty\ve_t^{1/2}\gm_t^{-2\dl}|v_t|dK_t<\infty.     $$

Then condition {\rm (i)} of Theorem $\ref{proc-t3.1}$ is satisfied.
\end{corollary}

\begin{proof}
Since $|\bt_t(u)-\bt_t|=|\frac{v_t(u)}{u}|$, we have for $\dl$,
$0<\dl<\frac{\dl_0}{2}$,
\begin{gather*}
    \int_0^\infty\ve_t^{1/2}\gm^{-\dl}|\bt_s(z_t)-\bt_t|dK_t\leq
        \int_0^\infty\ve_t^{1/2}\gm_t^{-\dl}
        \Big|\frac{v_t(z_t)}{z_t^2}\Big|
        |z_t|dK_t\\
    \leq\const(\om)\int_0^\infty\ve_t^{1/2}\gm_t^{-2\dl}
        \Big|\frac{v_t(z_t)}{z_t^2}\Big|dK_t \\
    \leq\const(\om)\int_0^\infty\ve_t^{1/2}\gm_t^{-2\dl}|v_t|dK_t<\infty.
            \;\;\;\qedhere
\end{gather*}
\end{proof}

\begin{corollary}\label{proc-c3.2}
Let $\ell_t(u)-\ell_t=\om_t(u)$, where for each $t\in[0,\infty)$
$$  \Big|\frac{\om_t(u)}{u}-\om_t\Big|\to 0\;\;\;\text{as}\;\;\;
        u\to 0,\;\;\Pas,  $$
Assume that the condition below is satisfied:

{\rm (ii$')$} there exists $\dl$, $0<\dl<\dl_0$ such that
$$  \frac{1}{\langle L\rangle_t}\int_0^t\Gm_s^2\gm_s^{-\dl}|\om_s|^2
        ds\to 0,\;\;\;\text{as}\;\;\;t\to\infty,
        \;\;\;(\Pas).  $$
Then condition {\rm (ii)} of Theorem $\ref{proc-t3.1}$ is satisfied.
\end{corollary}

\begin{proof}
For all $\dl$, $0<\dl<\dl_0$ we have
\begin{gather*}
    \langle N\rangle_t=\int_0^t\Gm_s^2(\ell_s(z_s)-\ell_s)^2 dK_s=
        \int_0^t\Gm_s^2\Big(\frac{\ell_s(z_s)-\ell_s}{z_s}\Big)^2
        z_s^2 dK_s\\
    \leq\const(\om) \int_0^t\Gm_s^2\gm_s^{-\dl}|\om_s|^2 ds,
\end{gather*}
since $\gm_t^\dl z_t^2\to 0$ as $t\to\infty$, $\Pas$, and
$$  \Big|\frac{\ell_t(z_t)-\ell_t}{z_t}-\om_t\Big|=
        \Big|\frac{\om_t(z_t)}{(z_t)}-\om_t\Big|\to 0\;\;\text{as}\;\;
        t\to\infty.      $$

        Finally, we can conclude that the assertion of Theorem \ref{proc-t3.1} is valid
if we replace conditions (i), (ii) by (i$^{\prime})$, (ii$^{\prime})$,
respectively.
\end{proof}

\begin{example}\label{ex4}
Averaging Procedure for RM Stochastic Approximation Algorithm with Slowly
Varying Gain. 

        Let $H_t(u)=\frac{\al}{(1+K_t)^r}R(u)$, where $\frac{1}{2}<r<1$,
$R(u)=-\bt u+v(u)$, where $v(u)=0(u^2)$ as $u\to 0$,
$\ell_t=\frac{\sg_t}{(1+K_t)^r}$, $\sg_t^2$ is deterministic,
$\sg_t^2\to\sg^2$ as $t\to\infty$, $K=(K_t)$ is a continuous increasing
function with $K_\infty=\infty$. That is, we consider the following SDE:
$$
    z_t=z_0+\int_0^t \frac{\al}{(1+K_s)^r}\, R(z_s)\, dK_s +
        \int_0^t \frac{\sg_t}{(1+K_t)^r}\, d m_t
$$
with $d\la m\ra_t= dK_t$.

        If $r>\frac{4}{5}$, then according to Example 6 of Section 2
$$  (1+K_t)^{r/2}z_t \os{d}{\to} \sqrt{\frac{\al\sg^2}{2\bt}}\,\xi,\;\;\;
        \text{as}\;\;\;t\to\infty,\;\;\xi\in N(0,1),     $$
and moreover, for all $\dl$, $0<\dl<\frac{\dl_0}{2}$, $\dl_0=2-\frac{1}{r}$,
$$  (1+K_t)^\dl z_t\to 0\;\;\;
        \text{as}\;\;\;t\to\infty\;\;(\Pas),     $$

Thus for the convergence
$$  (1+K_t)^{1/2}\ol{z}_t \os{d}{\to} \sqrt{\frac{\sg^2}{\bt^2}}\,\xi\;\;\;
        \text{as}\;\;\;t\to\infty,\;\;\xi\in N(0,1),     $$
it is sufficient to verify condition (i$^{\prime})$ of Theorem \ref{proc-t3.1},
since condition (ii) is satisfied trivially.
\end{example}

        In this example the object $v_t(u)$ defined in Corollary \ref{proc-c3.1} is
$$  v_t(u)=\frac{\al v(u)}{(1+K_t)^r},       $$
and for condition $(i^{\prime})$ of Corollary \ref{proc-c3.1} to be satisfied it is
sufficient to require the following: there exists $\dl$, $0<\dl<\dl_0$,
$\dl_0=2-\frac{1}{r}$ such that
$$  \int_0^t(1+K_t)^{1/2}(1+K_t)^{-2\dl}(1+K_t)^{-r}dK_t<\infty      $$
or equivalently, there exists $\dl$, $0\!<\!\dl\!<\!\dl_0$,
$\dl_0\!=\!r-\frac{1}{r}$ such that $r(1+\dl)-\frac{1}{2}\!>~1$.

        It is not difficult to check that if $r>\frac{5}{6}$ such a $\dl$
does exist.


\end{document}